\documentclass[11pt]{amsart}

\usepackage{amsmath}
\usepackage{amssymb,mathtools,amsthm,latexsym}
\calclayout

\usepackage{stackrel}

\usepackage{color}

\numberwithin{equation}{section}
\usepackage{tikz}
\usetikzlibrary{matrix,arrows}

\theoremstyle{plain}
\newtheorem{theorem}[equation]{Theorem}
\newtheorem{proposition}[equation]{Proposition}
\newtheorem{lemma}[equation]{Lemma}
\newtheorem{corollary}[equation]{Corollary}

\theoremstyle{remark}
\newtheorem{remark}[equation]{Remark}

\theoremstyle{definition}
\newtheorem{definition}[equation]{Definition}

\newtheorem*{question*}{Question}

\usepackage{graphicx}

\makeatletter

\let\c@equation\c@figure
\makeatother

\title[Loewner Carpets]{Thin Loewner carpets and their quasisymmetric embeddings in $S^2$
}
\author{Jeff Cheeger}
\address{Courant Institute, 251 Mercer St, New York, NY 10012}
\email{jc6@nyu.edu}

\author{Sylvester Eriksson-Bique}
\address{Department of Mathematics, University of California, Los Angeles, Box 95155, Los Angeles, CA, 90095-1555}
\email{syerikss@math.ucla.edu}

\newcounter{prob}
\newcommand{\Z}{\ensuremath{\mathbb{Z}}}

\newcommand{\N}{\ensuremath{\mathbb{N}}}
\newcommand{\R}{\ensuremath{\mathbb{R}}}
\newcommand{\C}{\ensuremath{\mathbb{C}}}

\newcommand{\diam}{\ensuremath{\mathrm{diam}}}
\newcommand{\Lip}{\ensuremath{\mathrm{Lip\ }}}

\newcommand{\defeq}{\mathrel{\mathop:}=}

\newcommand{\str}{\ensuremath{\mathrm{Str}}}

\newcommand{\len}{\ensuremath{\mathrm{length}}}

\newcommand{\Mod}{\ensuremath{\mathrm{Mod}}}

\def\Xint#1{\mathchoice
{\XXint\displaystyle\textstyle{#1}}%
{\XXint\textstyle\scriptstyle{#1}}%
{\XXint\scriptstyle\scriptscriptstyle{#1}}%
{\XXint\scriptscriptstyle\scriptscriptstyle{#1}}%
\!\int}
\def\XXint#1#2#3{{\setbox0=\hbox{$#1{#2#3}{\int}$ }
\vcenter{\hbox{$#2#3$ }}\kern-.58\wd0}}
\def\avint{\Xint-}

\newcommand{\co}{\mskip0.5mu\colon\thinspace}   

\def\vint_#1{\mathchoice%
        {\mathop{\kern 0.2em\vrule width 0.6em height 0.69678ex depth -0.58065ex
                \kern -0.8em \intop}\nolimits_{\kern -0.4em#1}}%
        {\mathop{\kern 0.1em\vrule width 0.5em height 0.69678ex depth -0.60387ex
                \kern -0.6em \intop}\nolimits_{#1}}%
        {\mathop{\kern 0.1em\vrule width 0.5em height 0.69678ex
            depth -0.60387ex
                \kern -0.6em \intop}\nolimits_{#1}}%
        {\mathop{\kern 0.1em\vrule width 0.5em height 0.69678ex depth -0.60387ex
                \kern -0.6em \intop}\nolimits_{#1}}}
\def\vintslides_#1{\mathchoice%
        {\mathop{\kern 0.1em\vrule width 0.5em height 0.697ex depth -0.581ex
                \kern -0.6em \intop}\nolimits_{\kern -0.4em#1}}%
        {\mathop{\kern 0.1em\vrule width 0.3em height 0.697ex depth -0.604ex
                \kern -0.4em \intop}\nolimits_{#1}}%
        {\mathop{\kern 0.1em\vrule width 0.3em height 0.697ex depth -0.604ex
                \kern -0.4em \intop}\nolimits_{#1}}%
        {\mathop{\kern 0.1em\vrule width 0.3em height 0.697ex depth -0.604ex
                \kern -0.4em \intop}\nolimits_{#1}}}

\begin{document}

\maketitle

\begin{abstract}
\noindent A {\it carpet} is a metric space which is homeomorphic to the standard 
Sierpinski carpet in
$\R^2$, or equivalently, in $S^2$. A carpet is called thin if its Hausdorff dimension is $<2$. 
A metric space is called {\it  Q-Loewner}
if its $Q$-dimensional Hausdorff measure is Q-Ahlfors regular and if it satisfies
a $(1,Q)$-Poincar\'e inequality. As we will show, $Q$-Loewner planar metric spaces are always carpets, and admit quasisymmetric embeddings into the plane.


 In this paper, for every pair $(Q,Q')$, with $1<Q<Q'< 2$
we construct infinitely many pairwise quasi-symmetrically distinct
$Q$-Loewner carpets $X$ which admit explicit snowflake embeddings,
$f\co X\to S^2$, for which the image, $f(X)$, admits an explicit description and
is $Q'$-Ahlfors regular. In particular, these $f$ are quasisymmetric embeddings.
 By a result of Tyson, 
the Hausdorff dimension of
a Loewner space cannot be lowered by a quasisymmetric homeomorphism. 
By definition, this means that
the carpets $X$ and $f(X)$ realize their conformal dimension.
Each of images $f(X)$ can be further uniformized via post composition with a quasisymmetric 
homeomorphism of $S^2$, so as to yield a circle carpet and also a square carpet.

Our Loewner carpets $X$ are constructed via what we call an {\it admissible 
quotiented inverse system}. This mechanism
 extends the inverse limit construction for PI spaces given in \cite{cheegerkleinerinverse}, which however, 
does not yield
 carpets. Loewner spaces
are a particular subclass of PI spaces. They have strong rigidity 
properties which do not hold for PI spaces
in general.

 In many
cases the construction of the carpets and their
snowflake embeddings, $f$, can also be described in terms 
of substitution rules. The 
statement above concerning $(Q,Q')$ is already a consequence of these examples. 
The images of these snowflake embeddings can be de-snowflaked using a deformation by a strong $A_\infty$ weight, which multiplies the metric infinitesimally by a conformal factor of the form $\omega=d(\mathrm{Image}(f),\cdot)^\alpha$. Consequently, our examples also yield new examples of strong $A_\infty$-weights for which the associated metrics admit no 
bi-Lipschitz embeddings  
 into  Banach spaces with the Radon Nikodym Property such as $L_p$, for $1<p<\infty$ 
and $\ell_1$.
\end{abstract}

\tableofcontents

\section{Introduction}

The central
 objects considered in this paper are quasisymmetric maps, thin carpets and Loewner spaces.
By a {\it carpet} we mean a metric space homeomorphic to the standard Sierpinksi carpet.
By a \emph{thin} carpet, we mean one with Hausdorff dimension $Q \in (1,2)$.
Additional definitions and background are recalled below.
Previously, there were many studies of  quasisymmetric homeomorphisms between 
thin carpets, especially carpets in $S^2$, as well as studies of the relations between
 quasisymmetric maps and Loewner spaces.  These connections are strengthened in
 Theorem \ref{thm:confemb} which states that if a Loewner space with Hausdorff dimension $Q \in (1,2)$ is {\it planar},
i.e.\ if it embeds topologically in $S^2$,  then
it is a thin carpet 
which embeds quasisymmetrically  in $S^2$. The proof of Theorem \ref{thm:confemb}
depends on Proposition  \ref{lemma:carpet} and Theorem \ref{lem:quasicircle}.

Theory notwithstanding,
 prior to the present paper, there were no 
 examples in the literature of thin carpets in $S^2$ 
which were known to be quasisymmetrically equivalent to  Loewner spaces.\footnote{ Bruce Kleiner (unpublished), has  constructed some examples of thin Loewner carpets using a substitution rule. He showed that his 
examples admit quasisymmetric embeddings in $S^2$, though explicit such embeddings, or values of the Hausdorff dimension $Q'$ of their images,  are not known. Also, for other interesting explicit examples
of $2$-dimensional (``fat'') Loewner carpets in $S^2$, see  \cite{mackaytysonwildrick,sylvesterjasun}, \cite[Theorem 1.6]{merenkovwildrick}.} %
Here for every $Q, Q'$, with $1<Q<Q'<2$, we construct uncountably many 
explicit examples of Loewner spaces, $X_\alpha$,
of Hausdorff dimension $Q$
which admit explicit
quasisymmetric embeddings $f\co X\to S^2$,
such that $f(X)$ has Hausdorff
\hbox{dimension $Q'$}; see Theorem \ref{thm:existence}. 
In fact, these embeddings are {\it \hbox{$\alpha$-snowflake mappings}}, i.e.\ 
they
become bi-Lipschitz
if the metric on $X$ is snowflaked with exponent $Q/Q'$. Both the explicit
Loewner carpets and their images $f(X)$ can be described in terms of substitution
rules. For each $Q,Q'$ as above, we 
 show that an infinite subcollection of our explicit examples are pairwise 
quasisymmetrically distinct. Conjecturally, this holds for an uncountable subset.

Our explicit examples mentioned in the previous paragraph can also be viewed
as special cases of a   
general construction of Loewner carpets $X$  as
limits of an admissible quotiented inverse system; see 
Section \ref{sec:generalconstruction}. 
In the general case,  the proof
that the limit is planar follows from 
a general result (Proposition \ref{thm:GHplanar}). It states that
if $X_i$ are planar graphs equipped with a path metric which converge in the Gromov-Hausdorff sense to $X_\infty$, and if $X_\infty$ does not have cut-points, then $X_\infty$ is planar. 


The construction of admissible quotiented inverse systems 
generalizes the admissible inverse system construction 
of PI spaces given in \cite{cheegerkleinerinverse}.
That construction does not yield planar Loewner spaces.

The fact that for every fixed $Q, Q'$, infinitely many of our
thin carpet examples are pairwise
 quasisymmetrically distinct, is a consequence of the 
 strong rigidity properties of Loewner spaces.  In our case, the carpets in question are given explicitly
as $f_1(X_1),\,  f_2(X_2)\subset S^2$, and the issue boils down to being 
able to decide whether there are tangent cones of the Loewner spaces $X_1, X_2 $ 
which are {\it bi-Lipschitz equivalent}. Despite the the gain in rigidity and the reduction to the bi-Lipschitz case, the problem is still non-trivial. In general, deciding   bi-Lipschitz equivalence for spaces or their tangents with the same Hausdorff dimension is not easy. Solving this problem for an infinite family of examples involves introduction of (quasi-)invariants tailored to our specific context, and which are of independent interest. The construction of this invariant makes use of 
an unpublished result of Kleiner which we state as Theorem \ref{thm:rigidityquantitative}. 
Since no proof exists in the literature, for completeness we provide one; see Section \ref{sec:rigidity}.



In order to provide context and motivation for our main results, in the first part of this 
introduction,
we will review some of what is known about quasisymmetries,
carpets, and Loewner spaces.    In the second part, which begins
with Subsection \ref{ss:pls2}, 
we will state our main results, 
describe the \hbox{quotiented inverse system} construction which
gives rise to our Loewner carpet examples and briefly summarize the remainder of the paper.

\noindent
\subsection{Quasisymmetries} The classical theory of quasiconformal
(equivalently, \hbox{quasisymmetric}) mappings
between subsets of $S^2$ can be generalized to the context of metric spaces.
In particular, it is meaningful for
Sierpinski carpets; see for example \cite{bonkmerenkov}. 
In this connection seminal work  
was done by Heinonen-Koskela, \cite{heinonenkoskela}.  Let $B(x,r)$ denote the open ball with
center $x$ and radius $r$ in a metric space $X$.
\begin{definition}
A \textit{quasisymmetry} between metric spaces,
$f\co X_1\to X_2$  is a homeomorphism such that there exists $C>1$ such that
for every ball $B(x,r)\subset X_1$ there exists $R$ and $y\in X_2$ such that 
$B(y,R/C)\subset f(B(x,r))\subset B(y,CR)$. A \textit{quasisymmetric embedding} $f \co X \to Y$ is a quasisymmetry onto its image $\mathrm{Image}(f)$.
\end{definition}

Note that the collection of quasisymmetric homeomorphisms of a metric space
form a group. Also,
 while bi-Lipschitz maps preserve the
 Hausdorff dimension of a space, a quasisymmetric map might  increase or 
decrease the Hausdorff dimension. Replacing a given metric with its snowflake
provides an example of this
for which $f$ is the identity map.

The {\it conformal dimension} is 
a quasisymmetric
 invariant which  was
first defined by Pansu \cite{pansuconfdim} 
and discussed in a different form by Bourdon and Pajot \cite{bourdoncohomologie}.

Let $\mathcal{H}_{dim}(Y)$ denote the Hausdorff dimension of $Y$.
 The {\it Ahlfors regular conformal dimension} 
of a metric space $X$ is defined as
follows.
\begin{definition}
\label{d:arcd}
The {\it Ahlfors regular conformal dimension}      of $X$ is:
\begin{equation}
\label{def:confdim}
 {\rm confdim}(X) \defeq 
\inf 
\{\mathcal{H}_{dim}(Y) \ | \  Y 
\,\, {\rm
 is \, \, Ahlfors\, \, regular\, \, and\, \, there\,\, exists\, \, a\,\, quasisymmetry  }
\,\,  f\co X \to Y\,
\}
\end{equation}
\vskip1mm

If there exists some $Y$ as above with
$\mathcal H_{\dim}(Y)={\rm confdim}(X)$, then we say $Y$ 
{\it realizes} the conformal dimension of $X$.
\end{definition}
 \vskip2mm

If two metric spaces $X_1,X_2$ can be shown to have different conformal
dimensions then one can conclude that they must be
 quasisymmetrically distinct. However, there is a catch:
 in practice, the conformal
dimension is often difficult to compute.

In general, given $X$, there may be no space  $Y$
which  realizes the conformal dimension of $X$;
see for example \cite{tysonwu}. 
However, if such a $Y$ exists, it will possess additional  properties. 
This was initially observed by Keith and
 Laakso, who characterized the inability to lower Hausdorff dimension as being equivalent
 to the condition that some weak 
tangent cones possesses a curve family with positive modulus 
for some positive exponent \cite{keithlaakso}. (By passing to another weak
tangent cones if necessary, the exponent can always be taken to be $1$.)
As a simple example, any metric space of the form $X\times I$,
with $I$ an interval, realizes its conformal dimension.

Despite the above mentioned characterization, for many  spaces of interest,
it is difficult to compute the conformal dimension and also to decide whether 
it is realized. On the other hand, for
many such spaces it is known that {\it if} the conformal dimension is realized,
then it is realized by a Loewner space; see below for a discussion of Loewner spaces.

In light of the known rigidity properties
of Loewner spaces, if the conformal dimension is realized, then
strong consequences ensue. Specifically,
as we will explain, this is related to Cannon's conjecture and
the Kapovich-Kleiner conjecture \cite{haissinskyhyper}, 
which pertain to geometric group theory. Both of these conjectures could be solved by proving that the conformal dimension of a planar  hyperbolic group boundary is realized. This is explained by the fact that in both cases, it is known that if the conformal dimension is realized, then it must be realized by a Loewner space.


\subsection{Carpets}
By definition, a {\it carpet} is a metric space homeomorphic to the standard Sierpinski carpet $S_3$. 
Here, we are particularly concerned with {\it thin} carpets i.e.\ those with 
$1< \dim_\mathcal H(X)<2$. (There do exist {\it fat}  carpets $X$ in $S^2$,
with $\dim_\mathcal H(X)=2$  and also,
{\it very fat} carpets with $\dim_\mathcal H(X)>2$.) 

A fundamental theorem of Whyburn \cite{whyburnsierp} provides necessary and sufficient intrinsic 
conditions on a metric space which guarantee that it is a carpet.  
To state his result we need some definitions.

A topological space is called {\it planar} if it is homeomorphic to a subset of the plane.
A point $p$ in a metric space $X$ is {\it local cut point} if there is a neighborhood $U$
of $p$ such that $U\setminus \{p\}$ is not connected.

\begin{theorem}[Whyburn \cite{whyburnsierp}]
 \label{thm:whyburn}
  A compact metric space is a carpet if and only if it is planar, connected,
  locally connected, has no local cut-points and no open subset
 is homeomorphic to $\R^2$. 
 \end{theorem}
Whyburn gave additional characterizations of carpets in \cite{whyburnsierp}. 
For example.
\begin{theorem}[Whyburn \cite{whyburnsierp}]
\label{thm:whyburnsierp}
A compact subset $K \subset S^2$ with empty interior is a carpet if and only if $K = S^2 \setminus \bigcup D_i$, where $D_i$ are countably many Jordan domains with disjoint closures in $S^2$
with $\lim_{n \to \infty}\, \diam(D_n) = 0$. 
\end{theorem}
 Of particular interest are those carpets which arise via
 removal of sets of 
a particular form.  
Namely, a square (respectively circle) carpet is a planar 
subset $K=\Omega \setminus \bigcup_i D_i$, where $\Omega$ and $D_i$ are each squares
 (respectively circles). Consider the class of carpets $K\subset \R^2$  whose "peripheral circles" (see Section \ref{sec:tplsqcs2})
are uniformly relatively separated uniform quasicircles and which are Ahlfors regular
of dimension $<2$. Such carpets can be uniformized by circle carpets in the sense
that each such carpet  is quasisymmetric to a circle carpet which is unique up to 
Moebius transformations, see Theorem \ref{thm:confemb} below and  \cite{bonkuniformization}, \cite[Corollary 3.5]{haissinskyhyper}. Although these assumptions 
do not seem to be very restrictive, for hyperbolic group boundaries which are homeomorphic to carpets, the Kapovich-Kleiner conjecture would follow if such carpets were known
to satisfy them with some ``visual metric''. So called {\it slit carpets} are another interesting class of carpets \cite{merenkovwildrick}, which, for example, can fail to possess any quasisymmetric embedding to the plane.

Whyburn's theorems
could  be viewed a partial
explanation for the fact that
carpets arise
naturally in various contexts including 
 Julia sets of postcritically-finite rational maps and hyperbolic group boundaries; compare \cite{kapovitchkleiner}.
In the case of hyperbolic group boundaries, the carpets
which arise come naturally equipped with a quasisymmetry class of metrics, but not with
a canonical representative of this class. A special class of circle carpets 
arise naturally as limit sets of groups 
acting isometrically, discretely, and co-compactly on a closed convex subset of 
hyperbolic $3$-space with totally geodesic boundary. 
Circle carpets are quasisymmetrically rigid in the sense that if two of them 
are quasisymmetric then they are M\"obius-equivalent; see \cite{bonkkleinermerenkov}.  

Explicitly describing a circle carpet which is 
quasisymmetric  to a given carpet is typically not easy. 
Consequently, even though  \cite{bonkuniformization} 
and \cite[Corollary 3.5]{haissinskyhyper} show that 
many carpets possess a quasisymmetrically 
equivalent circle carpet, it is not practical to use this to 
distinguish carpets up to quasisymmetries. Further, for 
group boundaries which are homeomorphic to  
carpets, the Kapovich-Kleiner  
conjecture \cite{bonkICM} could be resolved by proving that these carpets admit quasisymmetircally equivalent circle carpets.\



\begin{remark}
\label{r:phsics}
  There are numerous studies
of analysis/probability on self-similar carpets; see for instance \cite{barlow}.  Also, 
the physics literature contains
 various models based on carpets; see e.g. \cite{porousmaterials}, \cite{antennas}.
\end{remark}

\subsection{Loewner spaces}
Loewner spaces are a certain subclass of PI spaces with strong rigidity
properties that do not hold in general for PI spaces.

 Recall that PI spaces are metric measure
spaces satisfying a doubling condition,
\begin{equation}\label{eq:doubling}
\mu(B(x,2r)) \leq C\mu(B(x,r))\, ,
\end{equation}
 and a Poincar\'e inequality. The latter condition
means the following.
Set
\begin{align}
\tau B(x,r)& \defeq B(x,\tau r)\notag\\
f_B &\defeq \vint_B f ~d\mu \defeq \frac{1}{\mu(B)} \int_B f ~d\mu\, ,\notag\\
\Lip[f](x) &\defeq \limsup_{\stackrel{y \to x}{ y \neq x}} \frac{|f(x)-f(y)|}{d(x,y)}\, .\notag
\end{align}
A metric measure space is said to satisfy a {\it $(1,p)$-Poincar\'e inequality}, if there are constants 
$C_{PI}$ and $\tau$, such that for every Lipschitz function $f$ and every ball $B \defeq B(x,r)$, 
we have
\begin{equation}
\label{eq:PI}
\vint_B |f-f_B| ~d\mu \leq C_{PI} \, r \left( \vint_{\tau B} \Lip[f](x)^p ~d\mu \right)^{\frac{1}{p}}.
\end{equation}
 As a consequence of H\"olders inequality, this inequality becomes stronger with a smaller exponent $p$. 
See \cite{heinonenkoskela,keith2003modulus} for discussion of alternative (ultimately
equivalent) 
versions of the Poincar\'e inequality in which the concept of an 
``upper gradient'' for the function $f$ replaces $\Lip f$.

A PI space carries among other things:
\vskip1mm

\begin{itemize}

\item
A unique
first order differential structure; \cite{ChDiff99}.
\vskip1mm

\item
 A measurable cotangent bundle;  \cite{ChDiff99}.
\vskip1mm

\item
A good
theory Sobolev spaces $H^{1,p}$ for $p>1$, \cite{ChDiff99}; see also 
\cite{shanmugalingamsobolev}.
\vskip1mm

\item
A theory of $p$--harmonic functions and blow-ups of Sobolev functions
 \cite{ChDiff99,bjornbjornbook,asymptBV}.
\end{itemize}

\begin{remark}
\label{r:pi}
A large class of examples of PI spaces which typically are not Loewner spaces
 is provided by the admissible inverse limit spaces of 
 \cite{cheegerkleinerinverse}. In fact, for 
a given limiting metric arising from their
construction,  they construct an uncountable collection
of  distinct measures which 
 make the limit into a PI space. Schioppa
observed that the measures in an uncountable subset of
 these are in actuality, mutually singular; \cite{schioppapoincare}. 
For additional
examples of PI spaces,  see \cite{schioppasingular}, \cite{kleinerschioppatopdim}. 
\end{remark}

\begin{remark}
\label{r:pil}
Although in this paper, we are {\it primarily} interested in Loewner spaces, to a limited extent, 
more general PI spaces are also relevant. For one thing, the differentiability
theory for $H^{1,p}$ maps between PI spaces has implications for 
quasisymmetric maps between
Loewner spaces.  
This plays a role in our proof that for every 
$Q,Q'$ as above, there is an infinite collection of carpets which are pairwise 
quasisymmetrically distinct.  
Also, our general
quotiented inverse system construction  provides new examples of PI spaces 
which are not carpets. Specifically, these spaces are always doubling and satisfy a $(1,1)$-Poincar\'e inequality. To obtain carpets, we additionally need to enforce Ahlfors regularity and planarity. 
\end{remark}

\begin{definition}
\label{d:ls}
A $Q$-Loewner space is PI space for which the measure $\mu$  is $Q$-dimensional
Hausdorff measure
and is \emph{Ahlfors regular:}
$$
c^{-1}r^Q\leq \mu(B(r,x))\leq cr^Q\, ,
$$
 and in addition, a $(1,Q)$-Poincar\'e 
inequality holds.
\end{definition}

In their seminal work \cite{heinonenkoskela}, Heinonen and Koskela gave a different
definition of Loewner spaces which they show to be equivalent to the one given above, when the space is $Q$-Ahlfors regular.
To state it, we need some definitions.

A {\it continuum} is a compact and connected set. A continuum is {\it nondegenerate}
if it has more than one point.
If $E, F \subset \C$ are two disjoint and compact sets,
we define their {\it relative separation} $\Delta(E,F)$ as follows:
\begin{align}
\label{e:Deltadef}
\Delta(E,F) := \frac{d(E,F)}{\min(\diam(E), \diam(F))}\, .
\end{align}



\begin{remark}
\label{r:lpi}
The Loewner space concept pertains specifically 
to metric spaces and Hausdorff measure, while the
PI space concept encompasses a more class general metric measure spaces on which one
can do first order calculus.
\end{remark}

Many examples of 
Loewner spaces appear in the literature including Euclidean spaces, certain ``uniform''
subsets of metric spaces \cite{bjornnages},
Riemannian manifolds
with lower Ricci bounds and their Gromov-Hausdorff limit spaces \cite{ChColI}, Carnot 
groups \cite{jerison}, Laakso spaces \cite{laaksospace}, the 
Boudon-Pajot Fuchsian and hyperbolic buildings \cite{bourdonPI,bourdonfuchs},  
certain constructions by Kleiner--Schioppa \cite{kleinerschioppatopdim} and the $2$-Ahlfors regular ``fat carpets'' of \cite{mackaytysonwildrick} 
and \cite{sylvesterjasun}.  However, none of these are thin Loewner carpets.


\subsection{Rigidity of Loewner spaces under quasisymmetries}
 The following properties exemplify the quasisymmetric rigidity of Loewner spaces.
\begin{itemize}
\vskip2mm

\item[1)]
A fundamental result from  the seminal work of Heinenen-Koskela, \cite{heinonenkoskela},
 states that if a map $f$ between Loewner
spaces is quasiconformal, then $f$ is in fact quasisymmetric.
Here, {\it quasiconformal} is an infinitesimal
version of the quasisymmetry   condition. 
\vskip1mm

\item[2)]
According to Semmes, \cite{davidsemmesinfty, semmesainfty} (see also the argument in \cite{bonkheinonensaksmann}),
 if there exists a quasisymmetric homeomorphism
between Loewner space $X_1$ and $X_2$, then the metric on $X_2$ is obtained from that 
on $X_1$ by a process known as deformation by a strong $A_\infty$ weight. 
\vskip1mm

\item[3)]
By a result of Tyson,
\cite{tysonconfdim},
a Loewner space realizes its conformal dimension; equivalently it 
is not  quasisymmetrically equivalent to a metric space
of strictly smaller Hausdorff dimension.  
\vskip1mm

\item[4)]
 A quasisymmetric homeomorphism between Loewner spaces lies in
the Sobolev space $H^{1,Q}$ and in particular, has an almost everywhere defined 
strong differential which is an almost everywhere defined map 
of the measurable cotangent bundles; see
\cite{ChDiff99}, \cite{heinonenquasi}, \cite{kleinerICM}.
\end{itemize}
\vskip2mm



Theorem \ref{thm:existence} below states that for
 all $Q,Q'$ with $1<Q<Q'<2$, there exist infinitely many pairs of Loewner carpets
$X_1,X_2$ as above, which are quasisymmetrically distinct.  Given the explicit
examples, the proof that they are quasisymmetrically distinct uses a new bi-Lipschitz invariant, see Equation \ref{eq:quasiinv}. The proof of its (quasi-)invariance in Theorem \ref{thm:comparability} uses a  refined version of 4), namely Theorem \ref{thm:rigidityquantitative}. 

Further rigidity can be obtained using Theorem \ref{thm:rigidityquantitative}. Indeed, in Proposition \ref{prop:uniformquasiconformal} we show that the quasiconformal constant is uniformly controlled. While not needed in this paper, it is of independent interest.

We also mention that  a {\it different 
application} of our Loewner carpet examples
provides new examples of {\it strong $A_\infty$-weights}; see 
Theorem \ref{thm:ainfty}.
\vskip2mm

 \subsection{Cannon's conjecture and the Kapovich-Kleiner conjecture}
 As previously mentioned, metric spaces which are boundaries of Gromov hyperbolic spaces, 
and of Gromov hyperbolic groups have a natural quasisymmetry class of visual metrics,
but not a canonical representative of this class;
for background on Gromov hyperbolic spaces, see \cite{buyalohyper}. 
Bonk and Kleiner have shown that if the conformal dimension
of a hyperbolic group boundary is realized, then it must be realized by a Loewner space \cite{bonkkleinerhyper}. 
Below, we will briefly discuss two important specific instances in which
the rigidity of Loewner spaces under 
quasisymmetries would have significant consequences if one could show
that the conformal dimension is realized.


In case the boundary is $S^2$,    Cannon's conjecture  asserts 
that the group (up to finite index)
acts isometrically and properly discontinuously on hyperbolic 3-space.
 In \cite{bonkkleinertwosphere}, Bonk and Kleiner showed that Cannon's
conjecture is true in those
cases in which the conformal dimension of the boundary
 is realized. They conjecture that this always holds.
In this case, the conformal dimension, $2$, would be realized, by the Loewner space $S^2$ 
with its standard metric.

The Kapovich-Kleiner conjecture is the analogous statement for 
hyperbolic groups whose boundaries are carpets. Namely, that in this
case, the group (up to finite index)
acts  discretely, cocompactly, and isometrically on a convex subset of
$\mathbb{H}^3$  with nonempty totally
geodesic boundary; see \cite{kapovitchkleiner}.
 Ha\"issinski showed that the
Kapovich-Kleiner conjecture follows
from the conjecture that if the boundary of a word hyperbolic group
is planar then it admits a quasisymmetric embedding in $S^2$; see
Conjecture 1.7 of \cite{haissinskyhyper}.
For boundaries which are carpets of conformal dimension $<2$, this holds
by work of Ha\"issinski \cite{haissinskyhyper}, which uses an approach sketched by Bonk and Kleiner \cite{bonkICM}, compare also 
Theorem  \ref{thm:confemb} and Proposition \ref{lemma:carpet}.

The fact that the conformal dimension is $<2$ would follow,
 if one could show that these carpets attain their conformal dimensions with a Loewner space. While proving that the minimizers for conformal dimension exist may be a harder problem than showing that the conformal dimension is less than two, finding such minimizers could also have other consequences for the rigidity of the spaces \cite{bourdonkleiner}. 
 It is then tantalizing
 to try to understand when and if these carpet boundaries attain their conformal dimensions.

\subsection{The combinatorial Loewner property implies
minimizers are Loewner}
For spaces, $X$, which satisfy a combinatorial version of the Loewner condition,
  it can be shown that if there
exists $Y$  which realizes the conformal dimension of $X$ then 
of necessity, $Y$ is Loewner. Conversely, if there exists a minimizer, $Y$, which
 is Loewner, then the space $X$ is, of necessity, combinatorially Loewner; 
see \cite{bourdonkleiner,kleinerICM}. So for such spaces,
the remaining question is whether the conformal dimension is realized.
 
The  {\it combinatorial Loewner condition}  was
 introduced in \cite{bourdonkleiner} (see \cite{claiscombmod, haissinskycombmod} for a variation).
The first conclusion 
 above follows immediately
from the comparability of the discrete and continuous modulus given in 
\cite[Proposition B.2]{haissinskycombmod}, together with a different definition of the Loewner property (see \cite[Definition 3.1]{heinonenkoskela}).

Spaces satisfying the combinatorial Loewner property arise naturally in 
contexts with sufficient symmetry.  A sufficient, and widely applicable, notion of symmetry is described and proven in \cite{bourdonkleiner}. However, the necessary symmetry has not been fully described in the literature.  For example, any Sierpi\'nski carpet $S_p$, 
Menger curve and certain boundaries of Coxeter groups satisfy this property \cite{bourdonkleiner}.  Indeed, for some planar group boundaries, this provides a second proof, different
from that of Bonk-Kleiner \cite{bonkkleinerhyper}, that a minimizer  must
 be Loewner, provided it exists. 

In all of the above cases, whether the 
conformal dimension is realized is still 
a hard open question. In particular,  it is not known if a standard carpet $S_p$
carpet attains its conformal dimension. Indeed, our constructions of planar 
Loewner spaces which 
have  some self-similarity and which are quasisymmetric to square carpets, $f(X)$,
shows at least that these properties of $S_3$ are not incompatible with 
the property that the conformal
dimension is realized. 

Further, our examples can be used to give many carpets (all $X_Q$ in Theorem \ref{thm:existence}) which attain their conformal dimension and are combinatorially Loewner, as well as carpets which attain their conformal dimension but fail to be combinatorially Loewner. For example, glue $X_{Q_1}$ and $X_{Q_2}$ from Theorem \ref{thm:existence} with distinct $Q_1,Q_2$ quasisymmetrically.

\subsection{Thin Loewner carpets
embed quasisymmetrically in $S^2$}
\label{ss:pls2}
In addition to our explicit examples of carpets, we will show 
 that if $X$ is a compact planar metric space which is $Q$-Loewner, 
 for some $1<Q<2$, then $X$  is a carpet which
embeds quasisymmetrically in $S^2$; see Theorem \ref{thm:confemb}
and Proposition \ref{prop:uniformization}.\footnote{The fundamental question of
whether {\it every thin carpet}, perhaps with a mild assumption, embeds quasisymmetrically in $S^2$ remains open.} As mentioned, this is the first result
in which all three classes of spaces --- thin planar Loewner spaces,  carpets and
spaces with explicit quasisymmetric embeddings in $S^2$ ---  play a simultaneous role.
Thus, it can be viewed as one of our main results.
In particular, it provides  motivation for our explicit constructions of thin 
Loewner carpets.

\begin{remark}
\label{r:blsf}
By a result of \cite{ChDiff99}, no planar Loewner space bi-Lipschitz embeds into the plane.  
On the other hand, our results give examples of such spaces  which  admit  snowflake 
embeddings, among them, examples for which
snowflaking exponent, $Q/Q'$, can be chosen arbitrarily close to $1$.
\end{remark}

\subsection{Examples of thin Loewner carpets.}
Next we describe our examples of thin Loewner carpets. These
provide the first published examples of thin Loewner carpets,
and in particular, the first published examples which attain their conformal dimensions. 
 These are also the first examples admitting explicit quasi-symmetric embeddings 
 in $S^2$; compare
Theorem \ref{thm:confemb} and  Proposition \ref{prop:uniformization}. A figure with an approximation of one such construction is given in Figure \ref{fig:step2}.

\begin{theorem}\label{thm:existence}
 For every $1<Q <  Q'<2$, there exist infinitely many quasisymmetrically distinct
 $Q$-Ahlfors  regular carpets,  $X$, which 
satisfy a $(1,1)$-Poincar\'e inequality. Moreover:
\begin{itemize}
\item[1)]
 The  Loewner carpets $X$ can be chosen to be self-similar and to admit 
explicit quasisymmetric embeddings $f \co X \to \C$ such that $f(X)$
has Hausdorff dimension $Q'$. 
\vskip1mm

\item[2)]
The images $f(X)$ can be constructed by 
explicit substitution rules.
\vskip1mm

\item[3)]
The embeddings $f$ can be chosen
to be $(Q/Q')$-snowflake mappings.
\vskip1mm

\item[4)]
There exist planar quasiconformal maps $g,h \co \C \to \C$ so that $g \circ f(X)$ is a circle
 carpet and $h \circ f(X)$ is a square carpet.
\end{itemize}
\end{theorem}

\begin{remark}
\label{r:np}
Part of the interest in our examples stems from the fact that although they are planar,  
they do not attain their conformal dimensions with a planar quasiconformal map. In general, 
if $X \subset \C$, one can define also the notion of a conformal dimension with planar 
quasiconformal maps $f \co \C \to \C$ replacing the role of general quasisymetries.  In our case, the set $Y=f(X) \subset \C$, 
does attain its conformal dimension with an abstract spaces $X$, but there is no
 quasiconformal map $g\co \C \to \C$ so that $g(Y)$ is $Q$-Ahlfors regular.
\end{remark}

\begin{remark}
\label{r:pc}
 The square and circle carpet images $f(X)$ 
are obtained via post
composition with quasisymmetric maps of $S^2$ as in 4) of Theorem \ref{thm:existence},
 using results of \cite{bonkuniformization} or \cite{dimitriossquare}. In principle, 
 the square and circle carpets are explicitly computable, but as a consequence of 
their complexity carrying out the computations might not provide much insight. 
\end{remark}

\subsection{Admissible inverse and quotiented inverse systems}
\label{ss:iqs}
  One might 
hope, that Loewner carpets
could be explicitly constructed by starting with a suitable carpet in $S^2$ 
and then finding a metric, quasisymmetric to the given one, which 
is  Loewner. This approach, which one might call the 
``direct method'', seems very difficult to implement.
Here we   reverse the process by first constructing a Loewner carpet
$X$ and then a quasisymmetric embedding $f \co X\to S^2$.

 Recall that in \cite{cheegerkleinerinverse} the first author and  Kleiner
 introduced a class of {\it admissible inverse limit systems},
 $$
X_1\stackrel{\pi_1}{\longleftarrow} X_2\stackrel{\pi_2}{\longleftarrow}X_3
\stackrel{\pi_3}{\longleftarrow} \cdots
$$
whose objects are metric graphs equipped with suitable measures.
The spaces,
$X_i$, as well as the measures, $\mu_i$ and 
the maps $\pi_i$, were assumed to satisfy certain so-called admissibility
 conditions
from which it could be deduced that the doubling constants
and constants in the $(1,1)$-Poincar\'e inequality remained
  uniformly bounded as $i\to \infty$.
In some cases these limit spaces are planar while in other cases they are Loewner.   
 However, both conditions are never satisfied simultaneously for these examples.

The sequences of graphs in our construction are part of a more general
system which we refer to as an {\it admissible quotiented inverse system}. Here, the maps which
would be present in an inverse system are alternated with quotient maps, which are subjected to 
additional assumptions. At the formal level, the quotiented inverse system looks like.
\begin{align}
\label{e:iq}
X_1^1\stackrel{\pi_1^1}{\longleftarrow} X_1^2\stackrel{q^2_2}{\longrightarrow}X_2^2
\stackrel{\pi_2^2}{\longleftarrow}X_2^3\stackrel{q^3_3}{\longrightarrow} X^3_3
\stackrel{\pi_3^3}{\longleftarrow}\cdots
\end{align}

A more elaborate diagram displaying additional structure is given in the 
Figure \ref{fig:inversequotient}.  By using the  admissability conditions,
we will show that the graphs in the sequence
$X^n_n$  satisfy uniform $(1,1)$-Poincar\'e inequalities and are doubling. These graphs converge in the Gromov-Hausdorff sense (see for example \cite{bbi, fukaya, keith2003modulus, cheegerkleinerschioppa, ChColI}) to limit spaces $X_n^n\stackrel{d_{GH}}{\longrightarrow}X$. The limits $X$ are more general than carpets. Indeed, the main obstruction to being a carpet is planarity. Specifically, if $X_n^n$ are all planar, then the limit will be a carpet, except for some degenerate cases. We also show certain general properties for these limits, which are  independent of planarity and of independent interest. For example, the limits are always analytically one dimensional (Proposition \ref{p:adiq}).

The limits of quotiented inverse systems can be made ``uniform''.

\begin{definition}
\label{d:huniform}
A measure is said to be $h$-uniform if $\mu(B(x,r)) \asymp h(r)$ for some 
increasing function $h \co [0,\infty) \to [0,\infty)$.\footnote{For the notation $\asymp$, see the end
of this introductory section.}
\end{definition}

The function $h(r)$ can be controlled by varying the substitution rules used in our constructions. This gives Ahlfors regularity, which is needed to prove the Loewner property, but also yields an invariant in Equation \ref{eq:quasiinv}. This invariant plays a role in distinguishing the spaces up to quasisymmetries. See Subsection 
\ref{sec:uniformity} for more details. 

Additionally, to obtain carpets we need to enforce planarity of each graph $X_n^n$ in the sequence. From this, the planarity of the limits $X_\infty^\infty$ follows in one of two ways. For the explicit constructions we give, the planar embedding of the limit space is obtained  as a limit of embeddings from the sequence. For more general constructions we use Proposition \ref{thm:GHplanar}, which was suggested to us by Bruce Kleiner. We present the proof in Subsection \ref{sec:planarity}.

In the following, a cut-point for $X$ is a point $p$ such that $X \setminus p$ is disconnected.

\begin{proposition}\label{thm:GHplanar} If $X_i$ are planar graphs equipped with a path metric which converge in the Gromov-Hausdorff sense to $X_\infty$, and if $X_\infty$ does not have cut-points, then $X_\infty$ is planar.
\end{proposition} 

Combining this with uniformity yields the fact that limits of many uniform planar admissible inverse quotients give examples of Loewner carpets. In these cases one need not construct explicit embeddings to obtain planarity. See Corollary \ref{cor:limitplanar} for a detailed statement.





\subsection{Explicit examples via substitution rules}
In certain cases, the sequence $X^n_n$ of graphs in our admissible quotiented inverse
 system can also be described in terms
of {\it substitution rules}. In fact, we will use these cases
in the proof of Theorem \ref{thm:existence}.
  Roughly speaking, employing a substitution rule means that one passes from $X^n_n$ to
$X^{n+1}_{n+1}$, by removing successively disjoint pieces of $X^n_n$ of some finite collection of
specified types
and replacing them successively with pieces of
a different specified type chosen from another
 finite collection. In passing 
from $X^n_n$ to $X^{n+1}_{n+1}$,  edge lengths decrease by a definite factor 
and the number of edges increases by a definite factor. Essentially by definition,
spaces constructed by substitution rules exhibit at least some degree of self-similarity.


The particular examples of Loewner carpets that we construct via substitution
rules involve at each stage several distinct choices.  
By suitably varying them,
we obtain 
uncountably many examples, $X$, for each
$Q, Q'$, as in Theorem \ref{thm:existence}. 
However, as previously mentioned, for now, we
can only show that a countable subclass of the above examples
 are pairwise quasisymmetrically distinct;
see Section \ref{sec:rigidity}.

\subsection{Quasisymmetric embeddings}
As mentioned at the beginning of this introduction,\linebreak \hbox{Theorem  \ref{thm:confemb}}
states that a thin planar Loewner space is a carpet which embeds
quasisymmetrically in $S^2$. The limits of an admissible, uniform and planar inverse quotient system will be carpets. However, the proof of this in Corollary \ref{cor:limitplanar} does not yield an effective embedding for this limit space. Thus,
 Theorem \ref{thm:confemb} is needed to find quasisymmetric embeddings. This approach is non-explicit.
 
 For specific examples we will present, without resorting to general methods,
 one can construct explicit
  uniformly quasisymmetric embeddings for $X_n^n$  into the plane in such a way that one can pass to 
the limit. This can be done for infinitely many quasisymmetrically distinct pairs as above.

The above mentioned
 embeddings map the graphs $X^n_n$ to finer and finer grids whose lines are parallel to the
coordinate axes. These embeddings
 become bi-Lipschitz after the original metric on $X$ is snowflaked. 
The fine grids enable one to
make the images wiggle appropriately. An example of such an embedding in which
the first few stages are simple enough to actually be
 drawn  is given in Figure \ref{fig:step2}.

\subsection{Non-embedding of metrics arising from strong $A_\infty$-weights}
  \label{sec:embedding}
 
Strong $A_\infty$ weights were introduced by David-Semmes \cite{davidsemmesinfty, semmesainfty}.
A strong $A_\infty$ weight on a $Q$-Loewner space $(X,D)$ is a nonnegative locally $L_1$ function $\omega$
such that for the measure $\mu$ satisfying $d\mu :=\omega\cdot d\mathcal H_Q$,
the following hold. 
\vskip1mm

\begin{enumerate}
\item The measure $\mu$ is doubling with respect to the metric $D$.
\vskip1mm

\item There exists a metric $D_\omega$ and a constant $C$, such that for all $x,y \in X$:

$$
\frac{1}{C} \mu(B(x,D(x,y)))^\frac{1}{Q}\leq D_\omega(x,y) \leq C\mu(B(x,D(x,y)))^\frac{1}{Q}\, .
$$
\end{enumerate}
\vskip1mm

  The metric $D_\omega$ is well defined up to bi-Lipschitz equivalence.
A representative of the equivalence class can always be taken to be:
$$
D_\omega(x,y) = \inf_{B_0, \cdots, B_m} \sum_{i=0}^{m} \mu(B_i)^\frac{1}{Q}\, .
$$
Here the balls $B_i$ are defined with respect to the metric $D$ and the
 infimum is taken over chains of balls $B_0, \dots, B_m$ such 
 that $x \in B_0, y\in B_m$ and $B_{i} \cap B_{i+1} \neq \emptyset$ for $i=0, \dots, m-1$.
 If $\omega$ is continuous, then the metric has a simpler expression:
$$
D_\omega(x,y) = \inf_{\gamma} \int_\gamma \omega^\frac{1}{Q} ~ds\, ,
$$
where the infimum is over all rectifiable curves $\gamma$ connecting $x$ to $y$.
 
Strong $A_\infty$ weights
 form a strict subset of Muckenhoupt $A_\infty$ weights in the classical 
 sense, see \cite{kansanen}. Equivalently, $\omega$ satisfies a reverse H\"older i.e.\ there exists $p>1$ and a constant $C$
such that for every ball $B$,
$$
\Big(\frac{1}{\mu(B)}\int_{B}\omega^p \, d\mathcal H^Q \Big)^{1/p}
\leq C\cdot \frac{1}{\mu(B)}\int_{B}\omega
\, d\mathcal H^Q\, .
$$

If $X$ is $Q$-Loewner and $\omega$ is strong $A_\infty$ 
weight then $(X,D_\omega)$ is also $Q$-Loewner
and the identity map is quasisymmetric from $(X,D)$ to $(X,D_\omega)$. Conversely, if $h\co X\to Y$ is a quasisymmetric 
homeomorphism between $Q$-Loewner spaces, then the push forward, 
$h^{-1}_*(\mathcal H_Q)$ of $\mathcal H_Q$,
under the map, $h^{-1}$, satisfies:
$$
d(h^{-1}_*(\mathcal H_Q))=\omega \cdot d\mathcal H_Q\,, 
$$ 
for some 
strong $A_\infty$ weight $\omega$.\footnote{This fact will play a role in Section \ref{sec:rigidity}.} When $Y=X$, then $\omega$ becomes a Jacobian $J_f$ of a quasisymmetric self-map $f\co X \to X.$

Specifically, if $f\co \R^N \to \R^N$ is quasisymmetric (which is equivalent to quasiconformal), then its Jacobian $J_f$ is a strong-$A_\infty$ weight. The converse question was asked and answered (in the negative) by Semmes. 
Namely, is {\it every} strong $A_\infty$ weight
(perhaps even assumed to be continuous)
up to a constant multiple of  the Jacobian $J_h$ for some quasisymmetric map $h\co \R^N\to \R^N$. Otherwise put, are there examples in which we can be 
certain that the metric obtained by deformation by a strong $A_\infty$ weight is
not just the original metric in some disguised form i.e.\ disguised by composition with some unknown 
bi-Lipschitz homeomorphism.
Semmes gave two different types of counter examples. 
The second, which we now describe,
is particularly  flexible and is the pertinent one for this paper.

Let $(Y,d)$ denote a complete doubling metric space and $f\co Y\to \R^N$ be an
 $\alpha$ snowflake embedding for some $\alpha<1$.
Semmes showed that the following function is a strong $A_\infty$ weight.
$$
\omega(x):= d(x,f(Y))^{1/\alpha}\, .
$$
Additionally, $f \co (Y,d)\to (f(Y),D_\omega)$ is bi-Lipschitz. Therefore,
 if there exists quasisymmetric $h\co (\R^N,d)\to (\R^N,d)$ with
Jacobian $\omega$, it follows that $h\circ f \co Y\to \R^N$ is a bi-Lipschitz embedding. Semmes gave a specific example with $n=4$, in
which $(Y,d)$ could be shown to admit no such embedding. (Another counter-example using a different obstruction of his works with with
$n=3)$.) Subsequently, Laakso, gave a different counter example in the plane $\R^2$; see  \cite{laaksoainfty}.  The spaces $(Y,d)$ in these examples
in fact were observed not to admit bi-Lipschitz embeddings in any uniformly convex Banach space (let alone $\R^2$).

By Assouad's theorem, any doubling metric space $Z$ admits 
a snowflake embedding
in some $\R^N$. Thus, as Semmes observed, if $Z$ does not bi-Lipschitz that $\R^N$,
 his construction of the associated 
$D_\omega$ gives provides an example of a strong $A_\infty$ weight which is not the Jacobian 
of a quasisymmetric homeomorphism
$h\co (R^N,d)\to (R^N,d)$ as above.

In particular, our examples of thin $Q$-Loewner carpets which admit $(Q/Q')$-snowflake
 embeddings provide such 
counter examples, but in $\R^2$, which do not in fact embed in any 
Banach space with the
Radon Nikodym Property ($RNP$) of which uniformly convex spaces are a special case);
see \cite{pisier} or the references in \cite{cheegerkleinerRN} for
 additional details. Recall, that a Banach space $B$ is said to satisfy the Radon-Nikodym property if every Lipschitz function $f \co [0,1] \to B$ is differentiable almost everywhere. See \cite{pisier} or the references in \cite{cheegerkleinerRN} for more details. Thus, we have:

 \begin{theorem}
\label{thm:ainfty} 
There exist strong $A_\infty$-weights $\omega$ on $\R^2$, such that $(\R^2, D_\omega)$ does not bi-Lipschitz embed into any Banach space with the Radon-Nikodym property. In particular $\omega$ is not a Jacobian of 
 any quasiconformal map.
 \end{theorem}

 These counter examples are particularly strong since the (explicit) snowflake embedding
is in the minimal possible dimension, $2$ and also because $(Q/Q')$ can be taken arbitrarily closed to $1$.  By Assouad, any doubling space admits a snowflake 
embedding in some Euclidean space $\R^N$. It is trivial that our carpets would give also an
example in higher dimensions of a weight that wasn't comparable to a 
Jacobian. Here, the difficulty is doing this construction in the minimal dimension, i.e. in the plane. 

We remark that the question of giving nontrivial
 sufficient conditions for a strong $A_\infty$-weight to be comparable to the Jacobian of a 
quasiconformal mapping is still wide open.

\subsection{Overview of the remainder of the paper}
We now briefly summarize the contents of the remaining sections of the paper.

 In Section \ref{sec:tplsqcs2} we show that thin planar Loewner spaces
are carpets admitting quasisymmetric embeddings in $S^2$. 
\vskip1mm

In Section \ref{sec:generalconstruction} we define and study 
admissible quotiented inverse systems. We prove that limits of planar admissible quotiented inverse systems are carpets. 
\vskip1mm

In Section \ref{s:gsiqs} we give a general scheme involving
 substitution rules for constructing admissible quotiented inverse systems. 
\vskip1mm

In Section \ref{sec:planarinversequot} we construct our uncountably
many examples of
explicit snowflake embeddings.
\vskip1mm

In Section \ref{sec:rigidity}, for each $Q,Q'$, we show that a countable subcollection
of our examples are pairwise quasisymmetrically distinct.
This section also contains the proof of the main rigidity Theorem \ref{thm:rigidityquantitative}.
\vskip1mm

\subsection{Notation and conventions} We list here some
conventions which are used in the remainder of the paper. 

Throughout we will only be discussing complete, proper metric spaces $(X,d)$ and metric measure spaces $(X,d,\mu)$ equipped with a Radon measure $\mu$. 

%
 For two quantities $A,B$, sequences $A=A_n,B=B_n$ or functions $A=A(x_i),B=B(x_i)$, 
 we will say $A \lesssim B$ if $A \leq C B$ for some fixed $C$. Further, denote $A \asymp B$ if $A \lesssim B$ and $B \lesssim A$. 
 If we want to make the constants explicit, we write $A \lesssim_C B$ for $A \leq CB$ and
 $A \asymp_C B$ if $A \lesssim_C B$ and $B \lesssim_C  A$. Implicitly this notation  means
 that $C$ does not depend on the parameters $x_i$ or $n$, but may depend on other constants
 in the statement of the theorem/lemma.
 
 We will need then notions of pointed measured Gromov-Hausdorff convergence, for which we refer to standard references \cite{bbi, fukaya, keith2003modulus, cheegerkleinerschioppa, ChColI}. At a point $p \in X$, the the tangent cones are limits of $(X,d/r_n, \mu/\mu(B(p,r_n)))$ along some subsequence $r_n \searrow 0$. The collection of them is denoted by $\mathcal{T}X_p$, and an individual tangent cone is denoted by $T_X$. We also refer to \cite{liconf} for a discussion of quasisymmetries and their blow-ups. We will briefly discuss Hausdorff convergence in the plane, for which the reader can consult \cite[p. 281]{munkres}. 

\vskip1mm
\subsection{Summary of Main Results}
The hall marks of our work are explicitness, and the ability to attain and control both the embedding and the space. The main new results of this paper can be summarized as follows.

\begin{enumerate}
\item \textbf{Embeddability:} Any thin planar Loewner space admits a quasisymmetric embedding to $\R^2$. Previous results involve either topological or metric conditions, but this theorem is the first to involve the Loewner condition together with the planarity.
\item \textbf{Explicit examples:} For any dimensions $1<Q,Q' < 2$ we construct infinitely many $Q$-Loewner carpets $X$, which admit a quasisymmetric embedding $f\co X \to \R^2$ so that $f(X)$ is $Q'$ Ahlfors regular. Indeed, $f$ is a snowflake embedding and the map and the space are explicit and arise through substitution rules.  The proofs  of embeddability for these examples can be done directly from the constructions. This special structure allows for further constructions of interest, such as strong $A_\infty$-weights that don't arise as Jacobians. 
\item \textbf{General Framework:} The explicit constructions fall into a framework of a quotiented inverse system. This framework permits more general construction, including other planar Loewner spaces with inexplicit embeddings, as well as other general spaces with Poincar\'e inequalities. 
\item \textbf{Quasisymmetric distinctness:} Using a new bi-Lipschitz invariant we show that infinitely many of our examples, with the same $(Q,Q')$ will be quasisymmetrically distinct. The invariance is proven using the strong rigidity Theorem for quasiconformal maps \ref{thm:rigidityquantitative}, which reduced the problem to a bi-Lipschitz mapping problem. We observe, that despite the explicit setting, a lack of more invariants prevents us from proving distinctness for all of our examples.
\end{enumerate}
\vskip4mm

\noindent \textbf{Acknowledgments:}We are very grateful to Mario Bonk and Bruce Kleiner for many
conversations, which provided crucial motivation for some of the questions addressed in this paper,
as well as for some of the results. 
In particular, we thank Kleiner for describing his unpublished constructions, using substitution rules,  
of some thin Loewner carpets
which he showed admit quasisymmetric embeddings in $S^2$. He also suggested Proposition \ref{thm:GHplanar} and its proof to us.  Further, our quasisymmetric distinctness makes use of Theorem \ref{thm:rigidityquantitative}, whose statement is due to Kleiner, but whose proof has not appeared previously. 

The tools to prove quasisymmetric embeddability of general (inexplicit) Loewner carpets case were inspired by discussions with Bonk, and his description of partial results with Kleiner. Bonk also pointed out the usefulness of the reference \cite{kinnerbergisoperim} and encouraged pursuing substitution rules. 

We also thank Qingshan Zhou for a careful reading of parts of the paper and for helpful comments. The first author was partially supported by NSF grant DMS-1406407. The second author was partially supported by NSF grant DMS-1704215. 

\section{Thin planar Loewner spaces are quasisymmetric to carpets in $S^2$.} \label{sec:tplsqcs2}

In this section which show that thin planar Loewner spaces, are carpets which embed
quasisymmetrically in $S^2$; see Theorem \ref{thm:confemb}.
 We will begin by recalling a number relevant definitions. Then we state a basic result of 
Tukia and V\"ais\"al\"a, Theorem \ref{thm:tukiavaisala} which is needed for the proof of our
key technical result, Theorem \ref{lem:quasicircle}. Theorem \ref{thm:confemb} is
an easy consequence of Theorem \ref{lem:quasicircle}.

\subsection{Quasisymmetries and quasicircles}

\begin{definition}
\label{d:llc}
A metric space $X$ is  $C$-linearly locally connected (LLC)
if there is a constant $C$ with the following two properties.
 \vskip2mm

 \begin{enumerate}
  \item[$LC_1$] For all $x,y \in X$ there is a continuum $E$ with $x,y \in E$ and
 such that \hbox{$\diam(E) \leq C\cdot d(x,y)$}.
\vskip1mm

  \item[$LC_2$] For every $0<r<\diam(X)/C$ and any ball $B(z,r)$ 
and any points $x,y \not\in B(z,Cr)$, there is a continuum 
$E$ with $x,y \in E$, 
and $E \cap B(z,r)= \emptyset$.
 \end{enumerate}
\vskip2mm

$X$ is {\it linearly locally connected} if $C$ as above exists. 
\end{definition}

\begin{definition}
\label{d:allc}
A metric space $X$ is  annularly linearly locally connected (ALLC) if for every $x,y,z \in 
X$, and any $R>0$, such that $R<d(x,y)<d(x,z)<2R$, there is a curve $\gamma$ such that
the following hold.
\begin{itemize}

\item[1)]
 $y,z \in {\rm Image}(\gamma)$.
\vskip1mm

\item[2)]
 $\diam(\gamma) \leq C\cdot d(y,z)$.
\vskip1mm

\item[3)] $\gamma \cap B(x,R/C) = \emptyset$. 
\end{itemize}
\end{definition}

Usually, the above definition
is stated using a continuum instead of a curve. (Recall, that a continuum is a compact connected 
set.)
  However, using a curve makes 
our proofs below easier. Also, for quasiconvex spaces this version is equivalent to the standard one. Quasisymmetries preserve the LLC and ALLC conditions.

\begin{definition}
\label{d:etaqs}
Given a homeomorphism $\eta \co [0,\infty) \to [0,\infty)$, we say that a homeomorphic map 
$f \co (X,d_X) \to (Y,d_Y)$ is $\eta$-quasisymmetric 
if for all $x,y,z \in X$, with $x\neq z$
\begin{equation} \label{def:quasisym}
 \frac{d_Y(f(x),f(y))}{d_Y(f(x),f(z))} \leq \eta\Bigg( \frac{d_X(x,y)}{d_X(x,z)} \Bigg).
\end{equation}
\end{definition}

\begin{definition}
\label{d:qsym} Given a metric space $X$, the image of a quasisymmetric embedding $\phi \co S^1 \to 
{\rm Image}(\gamma) \subset X$ is called a \emph{quasicircle}.
A collection $\Gamma$ of quasicircles is called uniform
if for some fixed $\eta$, it consists of images of $\eta$-quasisymmetric
maps of $S^1$.
\end{definition}

\begin{definition}
\label{d:bt}
 A curve $\gamma \co S^1 \to X$
has bounded turning, if there is a $C \geq 1$ such that
for any distinct $s,t \in S^1$, and arcs $I,J$ of $S^1$ defined by $s,t$,
$$
\min\left(\diam( \gamma(I) ),\diam(\gamma(J) ) \right) \leq C d(\gamma(s),\gamma(t))\, .
$$
\end{definition}

The following result of Tukia and V\"ais\"al\"a \cite{tukiavaisala}
provides a characterization of quasicircles.
\begin{theorem}[Tukia--V\"ais\"al\"a, \cite{tukiavaisala}]
\label{thm:tukiavaisala} 
If $\gamma \co S^1 \to X$ is a doubling embedded circle, then its image is a quasicircle if and only if $\gamma$ is of bounded turning. The 
function $\eta$ in Definition \ref{d:etaqs} depends solely on the constant $C$ in
the bounded turning condition inequality \eqref{d:bt}.
\end{theorem}

Next we show that compact thin planar Loewner spaces are carpets.

 \begin{proposition}\label{lemma:carpet}
For $1<Q<2$, a  compact planar $Q$-Loewner space $X$ is a carpet.
 \end{proposition}
 
 \begin{proof}
  Since  $X$ is $Q$-Loewner,
  by \cite[Theorem 3.13]{heinonenkoskela} it
is annularly linearly locally connected. In particular, it is locally connected and connected. We note that, while the definition of Loewner in \cite{heinonenkoskela} is stated slightly differently, the authors also establish the equivalence of our definition and theirs. 
  
An annularly linearly locally connected space cannot have local cut points. 
Also, $X$ cannot have manifold points since in that case,
 it would then have at least topological dimension $\geq 2$, and  hence,
Hausdorff dimension $\geq 2$ as well; see \cite{hurewicz}. Thus,
 by Whyburn's Theorem \ref{thm:whyburn}, $X$ is a carpet.
 \end{proof}
\vskip2mm

We will now show that such planar Loewner spaces $X$, which are also always Loewner carpets, admit a quasisymmetric embedding into the plane.

If $E, F \subset X$ are two disjoint and compact sets of a metric space $X$, which contain more than one point, 
recall from \eqref{e:Deltadef} that  their relative separation is defined as:
$\Delta(E,F) := \frac{d(E,F)}{\min(\diam(E), \diam(F))}$.

\begin{definition}
\label{d:rs}
A collection $\mathcal{F}$ of sets is uniformly $\delta$-relatively separated, 
if $\Delta(E,F) \geq \delta>0$ for 
each distinct $E,F \in \mathcal{F}$. 
\end{definition}

The planarity of a Loewner space $X$ guarantees, that
there is a {\it topological} embedding $\phi \co X \to \C$. Consequently, via stereographic projection, we also have an embedding $\phi \co X \to S^2$. As shown above, since $X$ is a carpet so is $\phi(X)$ and by  Whyburn's Theorem \ref{thm:whyburnsierp}, we can express $\phi(X) = S^2 \setminus \bigcup D_i$, where $D_i$ are countably many Jordan domains with disjoint closures\footnote{Recall, that a domain $D$ is a \textit{Jordan domain} if $\partial D$ is a Jordan curve. A \textit{Jordan curve} is any embedded copy of $S^1$ in $S^2$.}. So, each boundary $\partial D_i$ is a Jordan curve in $X$.  We will give each $\partial D_i$ some parametrization by an embedded circle $\gamma_i$, and denote the collection of these circles by $\Gamma$. The curves $\gamma_i$ lie in $\phi(X)$, but we will identify them via the homeomorphism $\phi$ with curves in $X$.

These curves are called {\it peripheral circles}. In general, any 
Jordan curve $\gamma$ in $X$ is a peripheral circle if and only if it doesn't 
separate $X$ into two components. This is an easy consequence of the Jordan curve theorem. 

In order to prove our main result, Theorem \ref{thm:confemb}, we will need the following theorem. 

\begin{theorem}
 \label{lem:quasicircle}
 If $X$ is a Loewner carpet, then the collection of peripheral circles $\Gamma = \{\gamma_i\}$ 
is uniformly relatively separated and consists of uniform quasicircles, that is, there is a function $\eta$ such that $\gamma_i$ are all $\eta$-quasicircles and another $\delta$ so that $\Delta(\gamma_i,\gamma_j) \geq \delta$ for each $\gamma_i,\gamma_j \in \Gamma$ distinct.
\end{theorem}

\begin{proof}
In order to simplify notation, and recalling the above discussion, we will identify $X$ with its image $\phi(X) = S^2 \setminus \bigcup D_i$ in the plane. The metric notions of diameter, and distance, will refer to the metric on $X$, which is distinct from the restricted metric in $S^2 \setminus \bigcup D_i$.

By assumption together with \cite[Theorem 3.13]{heinonenkoskela} and \cite{ChDiff99} the space
$X$ is $R$-ALLC and $R$-quasiconvex for some $R\geq 1$. 
\vskip2mm

\noindent
{\bf Step 1.} The peripheral circles are uniform quasicircles. 
\vskip1mm

By Theorem \ref{thm:tukiavaisala} it suffices to prove that this collection of  circles has  
uniformly bounded turning. (Indeed, since $\gamma_i \subset X$ and since $X$ is doubling, then $\gamma_i$ are all doubling.) We prove this by contradiction. 
Fix a large $C \geq 1$ (in fact $C=4R^2$ suffices), and suppose that some $\gamma_i$ is not $C$-bounded turning. 
Then, there would 
exist two  points $a,b \in \gamma_i$ which separate $\gamma_i$ 
into Jordan arcs $I,J \subset \gamma_i$ with 
$\diam(I) \geq \diam(J) \geq Cd(a,b)$. 
By the R-LLC condition, there exists a curve $\beta \subset X$ such that $\diam(\beta) \leq R\cdot
 d(a,b)$ connecting $a$ to $b$.
 By \cite[Theorem 1]{moorecurvesubset}, 
we can find another curve $\beta_1$  which is a simple 
curve contained in the image of $\beta$, and which connects 
the two points $a,b$. 

 
Now, since $D_i$ is a Jordan domain and $\partial D_i = \gamma_i$, by the Jordan curve theorem, we can find a simple curve $\beta_2$ which is contained in the interior of $D_i$, except for  its endpoints, which connects $a$ to $b$. Consider the simple Jordan curve $\sigma$ formed by concatenating $\beta_1$ and $\beta_2$. This curve divides $S^2$ into exactly two components.  
 
If $C = 4R^2$, then 
we can find two points $s \in I,\, t \in J$, 
which satisfy
$$
d(s,a)=d(t,a) \geq \frac{1}{2}C\cdot d(a,b)\, .
$$

\noindent
{\it Claim:}
 $s,t$ cannot both lie in the same component $U$ of $S^2 \setminus \sigma$.

To prove the claim,
 suppose that they both lay in the same component $U$ and denote the other component by $V$. Then the
 Jordan curve
 theorem implies that $\partial U = \partial V = \sigma$. Since every point in
 $D_i \setminus \sigma$ can be 
connected to either $s$ or $t$ by a simple curve contained in the closure
 and intersecting $\overline{D}_i$ only at the end-points, we would have that 
$D_i \setminus \sigma$ would also lie in $U$. 
However, then each point of $\sigma \cap D_i$ would have
 a neighborhood $D_i$ which only contains points of $U$, and it would follow
 that $\sigma \cap D_i \cap \partial V = \emptyset$. This
 contradicts the equality $\partial V = \sigma$.

 Since $X$ is $R$-ALLC,
 there exists another curve $\gamma$ connecting
$s,t$ within $X$ such that $\gamma \cap B(a,2Rd(a,b)) = \emptyset$. 
Then $\beta_1 \cap \gamma = \emptyset$ since $\beta_1 \subset \beta \subset B(a,2R d(a,b))$, 
and $\beta_2 \cap \gamma = \emptyset$ since $\gamma \subset X$. In particular 
$\gamma \cap \sigma = \emptyset$,
which is impossible since $\sigma$ separated the points $s,t$. This completes the
 proof of Step 1. i.e.\ peripheral
circles are uniform quasicircles.
\vskip2mm

\noindent
{\bf Step 2.} The collection $\Gamma$ is uniformly relatively separated.
\vskip1mm 

Again, we argue by contradiction. 
Assume that for some very small $\delta$ (in fact $\delta= \frac{1}{2^4R^4}$ suffices), we have,
$\Delta(\gamma_i,\gamma_j) \leq \delta$ for some distinct $\gamma_i,\gamma_j \in \Gamma$ i.e.\ 
$$
d(\gamma_i,\gamma_j) \leq \delta \min(\diam(\gamma_i), \diam(\gamma_j))\, .
$$
Let $a \in \gamma_i, b \in \gamma_j$
be such that $d(a,b) = d(\gamma_i,\gamma_j)$. By the $R$-quasiconvexity condition 
we can connect $a,b$ by a curve $\alpha$ with $\alpha \subset B(a,R\cdot d(a,b))$.

For points $s,t \in \gamma_i$, denote by $A_{st}$ the  
subarc of $\gamma_i$ containing $s,t$ containing $a$.  
If 
$\delta$ is chosen sufficiently 
small, we can pick $s,t \in \gamma_i$ to be the closest points to $a$, such that
  $d(s,a)=d(t,a) = 2R^2\cdot d(a,b)$. By choosing the closest $s,t$, we can ensure that  the 
sub-arc $A_{st}$ containing $a$ satisfies:
$$
A_{st} \subset B(a,2R^2\cdot d(a,b))\, .
$$

As above it follows from the ALLC condition that there is a simple curve $\beta_1$ connecting 
$s,t$ which does not intersect $\alpha$ and satisfies
with 
\begin{align}
\diam(\beta_1) &\leq 4R^3 d(a,b)\notag\\
 \beta_1& \subset B(a,8R^3 \cdot d(a,b))\, . \notag
\end{align} 
As above, we can connect $s,t$ by a simple curve $\beta_2$ within $\overline{D_i}$, and form a Jordan curve $\sigma$ by concatenating $\beta_1$ with $\beta_2$. This divides $S^2 \setminus \sigma$ into two components $U$,$V$. 

Without loss of generality, we can assume that
  $\delta<\frac{1}{2}$ and 
$$
d(a,b)=d(\gamma_i,\gamma_j) \leq \delta \min(\diam(\gamma_i), \diam(\gamma_j))\, .
$$ 
Then  we can find $x \in \gamma_i, y \in \gamma_j$ with $d(x,a) \geq \frac{1}{2\delta} d(a,b)$ and 
$d(y,a) \geq \frac{1}{2\delta} d(a,b)$.  

By the same argument as above we see that $x$ and $a$ lie in separate components of 
$S^2\setminus\sigma$, say $x \in U$ and $a \in V$. However, $\alpha$ does not intersect $\sigma$, 
and so $b \in V$ as well. Finally, we can find a simple   curve $\beta_3$  contained in $D_j \cup \{b,y\}$
 connecting
 $b$ and $y$. Then since $\sigma \subset D_i \cup X$ and $\beta_3$ does not intersect $\sigma$,
 it follows that $y \in V$ as well.
Consequently $x \in U$ and $y \in V$ lie in separate components.


However,
 by the R-ALLC condition, we can connect $x$ to $y$ with a curve
 $\gamma \subset X$ which avoids the ball $B(a,\frac{1}{2\delta R } d(a,b)) \cap X$. 
By choosing any $\delta < \frac{1}{8R^4}$, we have 
$$
\sigma \cap X = \beta_1 \subset B(a,4R^3 \cdot d(a,b)) \subset B\big(a,\frac{1}{2\delta R }d(a,b)\big)\, .
$$

Therefore, since $\sigma$ separates $x,y$, the curve $\gamma$ must intersect $\sigma$.
 However, $\gamma \cap \sigma \subset \gamma \cap \sigma \cap X \subset \gamma \cap \beta_1$,
 and so $\gamma$ must also intersect $B(a,\frac{1}{2\delta R }d(a,b))$, which is a contradiction. 
This completes the proof of Step 2., and hence, the proof of Theorem \ref{lem:quasicircle} as well.
\end{proof}

Finally, we show that Loewner carpets
can be realized as planar subsets via a quasisymmetric embedding.

\begin{theorem}
\label{thm:confemb}
 If $X$ is 
$Q$-Loewner planar space with $Q \in (1,2)$, then there is a quasisymmetric embedding   $f \co X \to S^2$.
\end{theorem}

\begin{proof} By Theorem \ref{lem:quasicircle}, we know that the peripheral circles 
are uniformly relatively separated uniform quasicircles. Further,
 by 
\cite[Theorem 3.13]{heinonenkoskela} $X$ is ALLC. Now, 
\cite[Corollary 3.5]{haissinskyhyper} implies that any space which is ALLC and $Q$-Ahlfors regular with $Q \in (1,2)$, and whose collections of peripheral circles consists of uniformly separated quasicircles, admits a quasisymmetric embedding\footnote{A minor technical point in this argument should be noted. There is an additional porosity assumption used in \cite{haissinskyhyper}, which is only used for a lower bound for Ahlfors regularity, and is not actually needed for the uniformization result. In fact, this lower bound for Ahlfors regularity follows purely from topological considerations. In particular a topological manifold, which is ALLC always satisfies the lower Ahlfors regularity bound, see \cite{kinnerbergisoperim}.} to $S^2$. Since our space satisfies these assumptions, it also admits such an embedding.
\end{proof}

\begin{remark}
 The proof of \cite[Theorem 3.1]{haissinskyhyper} involves gluing in
 quasidisks to the peripheral circles in order to construct a $2$-Ahlfors regular and ALLC 
surface, which then can be quasisymmetrically embedded using  \cite{bonkkleinertwosphere}. 
\end{remark}

\begin{remark} In the discussion above and following theorem, assuming the Loewner condition
is not strictly necessary; ALLC would suffice.
\end{remark} 

Now, this can be combined with \cite{bonkuniformization} or \cite{dimitriossquare} 
to give the following.

\begin{corollary}\label{prop:uniformization}
If $Y$ is planar and  $Q$-Loewner,
then there exists a quasisymmetric embedding 
$f \co Y \to \C$ and a planar quasiconformal map $g \co \C \to \C$ so
 that $g \circ f (Y)$ is a circle carpet. Similarly, the image can be uniformized 
with a square carpet by another map $h \co \C \to \C$ so that $h \circ f(Y)$ is a square carpet.
\end{corollary}

\section{Admissible quotiented inverse systems yielding planar
Loewner spaces}
\label{sec:generalconstruction}

In this section we define admissible quotiented inverse systems and prove that their measured
Gromov-Hausdorff limits are doubling and satisfy Poincar\'e inequalities. 
Initially we do {\it not} address the issue of quasisymmetric embeddings, and do not assume any uniformity or planarity. The general case does not require these, and leads to other examples, while specializing the construction to enforce these conditions leads to the planar Loewner carpets in Theorem \ref{thm:existence}.

We will consider systems of spaces $X^i_j$, where each $X^i_j$ is a metric measure graph
 with each edge isometric to an edge of length $s_i$. For simplicity, set $s_0=1$ and
 $s_i \leq s_{i'}$ for $i \geq i'$. 

\begin{definition}
\label{d:monotone}
Graphs satisfying the above
conditions will be referred to as \emph{monotone}, 
if there are functions $h_j^i \co X^i_j \to \frac{2}{\pi}S^1$, 
or $h_j^i \co X^i_j \to \R$, which are isometries when restricted to edges.
\end{definition}

  We use 
$\frac{2}{\pi}S^1$ to denote the circle rescaled to have length $4$. This is convenient
 for most of our explicit examples, but not actually necessary.
 The spaces $\R$ and $\frac{2}{\pi}S^1$ will be considered as simplicial complexes 
with edge length $1$.  They also have a natural structure as directed graphs and 
so $X^i_j$ are also directed graphs; the orientation of an edge $e$ is the one inherited from $h_j^i(e)$.



Fix a sequence of integers $\{m_i\}_{i=1}^\infty$ with $3 \leq m_i \leq M$, and define $s_k = \prod_{i=1}^k \frac{1}{m_i}$. We assume $X^i_k$ to have edge lengths given by $s_i$ with measures $\mu^i_k$ when $i, k \in \N_0$ and $k \leq i$. 
\begin{definition}
\label{d:iq}
A system as above is called a (monotone) quotiented inverse system if there are mappings 
$$
\pi^i_k \co X^{i+1}_k \to X^i_k,\quad   q^i_k \co X^i_{k-1} \to X^i_k,
 \quad  
h_i^k \co X^i_k \to \frac{2}{\pi}S^1 (\text{or } \R),
$$
that  commute according to diagram in Figure \ref{fig:inversequotient}.
\end{definition} 


We denote the compositions of these maps by 
$q^i_{lk} = q^i_k \circ \cdots \circ q^i_{l+1} \co X^{i}_l \to X^i_k$ and $\pi_k^{ij} = \pi_k^j \circ \cdots \circ \pi_k^{i-1} \co X^{i}_k \to X^j_k$.
\vskip2mm

For a graph $G$, with edge length $s$, we will denote the graph obtained by subdividing each edge into $m$ pieces of length $s/m$ by $G^{/m}$. If $v \in G$ is a vertex, we denote by $\str^G_v$ the closed star 
 of a vertex in $G$. By definition, it consists of $v$ together
 with the edges adjacent to it. 
If the graph $G$ is evident from context we will simply denote 
$\str_v$.


\begin{definition}
\label{d:iqs}
An quotiented inverse system is called {\em admissible} if the following properties hold:

\begin{enumerate}
 \item Simplicial property: The maps $\pi^i_k$ and $q^i_k$ are simplicial, where $\pi^i_k$ is 
considered as a function onto the subdivided graph $(X^{i}_{k})^{/m_{i+1}}$.
\vskip1mm

 \item Connectivity: The graphs $X^i_k$ are connected.
\vskip1mm

 \item Bounded geometry: The graphs $X^i_k$ have $C$-bounded degree, and $\mu^i_k(e)\asymp_C \mu^i_k(f)$, when $e,f \in X^i_k$ are adjacent edges.
 \vskip1mm

\item Compatibility with the measure: $$(q^i_k)^*(\mu^i_{k-1})=\mu^i_k, \ \ \ \ \  (\pi_k^i)^*(\mu^{i+1}_k)=\mu^i_k.$$
\vskip1mm

 \item Compatibility with monotonicity: $$h^i_k \circ \pi_k^i = h^{i+1}_k, \ \ \ \ \ h^i_k \circ q^i_k = h^i_{k-1}.$$
\vskip1mm

 \item Diameter bound: $\diam((\pi_k^i)^{-1}(p)) \leq Cs_i.$
\vskip1mm

 \item Openness: The maps $\pi_k^i$ are open.
\vskip1mm

 \item Surjectivity: The maps $\pi_k^i$ and $q^i_k$ are surjective.
\vskip1mm

 \item Balancing condition: If $v \in (X^{i} _k)^{ / m_{i+1}}$ and $v' \in X^{i+1}_k$, then there is a constant $c_{v',v}$ such that $(\pi^i_k)^*(\mu^{i+1}_{k-1}|_{\str_{v'}})=c_{v',v}\mu^i_k|_{\str_v}$. The star at $v$ is in the subdivided graph $(X^{i} _k)^{ / m_{i+1}}$.
\vskip1mm

 \item Quotient condition: There are constants $\delta,C>0$ such that if $l \leq k \leq i$,
 $x \in X^i_k$ and $r \in (0,\delta s_l)$, then:
 \begin{equation} \label{eq:quotient}
   \diam(\big( q^i_{lk} \big)^{-1} (B(x,r))) \leq Cs_l\, .
 \end{equation}
\end{enumerate}
\end{definition}

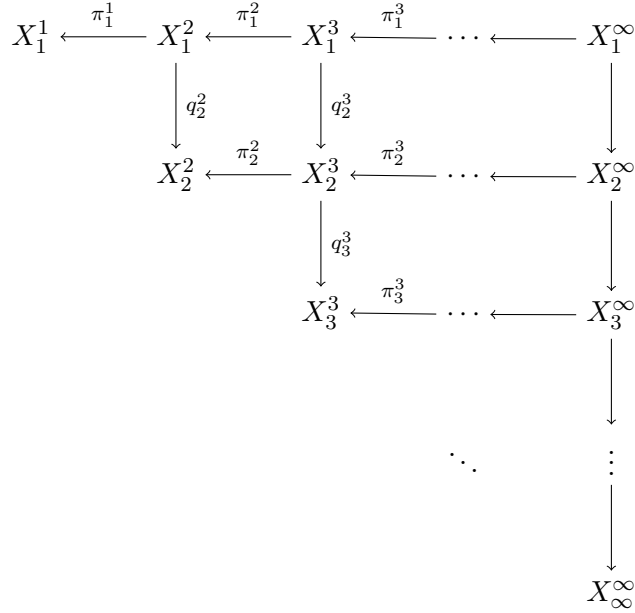
\begin{figure}[h!] 
  \begin{tikzpicture}[descr/.style={fill=white,inner sep=2.5pt}]
 \matrix (m) [matrix of math nodes, row sep=3em,
 column sep=3em]
 { X^1_1 & X^2_1 & X^3_1 & \cdots & X^\infty_1 \\
         & X^2_2 & X^3_2 & \cdots & X^\infty_2 \\
         &       & X^3_3 &  \cdots & X^\infty_3 \\
         &       &       & \ddots   & \vdots \\
         &	 	 & 	  &  	           & X^\infty_\infty \\};
 \path[->, font=\scriptsize]
 (m-1-2) edge node[above] {$\pi^1_1$} (m-1-1)
         edge node[right] {$q^2_2$} (m-2-2)
 (m-1-3) edge node[above] {$\pi^2_1$} (m-1-2)
         edge node[right] {$q^3_2$} (m-2-3)
 (m-1-4) edge node[above] {$\pi^3_1$} (m-1-3)
 (m-1-5) edge node[above] {} (m-1-4)
 (m-2-3) edge node[above] {$\pi^2_2$} (m-2-2)
 		edge node[right] {$q^3_3$} (m-3-3)
 (m-2-4) edge node[above] {$\pi^3_2$} (m-2-3)
  (m-3-4) edge node[above] {$\pi^3_3$} (m-3-3)
   (m-2-5) edge node[above] {} (m-2-4)
    (m-3-5) edge node[above] {} (m-3-4)
    (m-1-5) edge node[above] {} (m-2-5)
    (m-2-5) edge node[above] {} (m-3-5)
    (m-3-5) edge node[above] {} (m-4-5)
    (m-4-5) edge node[above] {} (m-5-5);
 \end{tikzpicture}
 \caption{Quotiented inverse system}\label{fig:inversequotient}
 \end{figure}

\begin{remark}
\label{r:conditions}
The conditions involving $\pi^i$ coincide with the Cheeger-Kleiner axioms for each 
of the rows being an inverse limit \cite{cheegerkleinerinverse} (see Figure \ref{fig:inversequotient}). 
This guarantees that the rows satisfy certain uniform Poincar\'e inequalities and 
doubling properties. The final assumption on $q$ is the crucial assumption that guarantees 
that the quotient maps preserve the Poincar\'e inequalities and doubling properties.
It is possible to verify the last assumption in many 
particular instances. For example, it suffices that the quotient maps $q$ only
identify vertices close enough to vertices. These will be discussed separately 
in Subsection \ref{subsec:quotientcondition}.
\end{remark} 

\subsection{Uniform doubling and Poincar\'e inequality}

Throughout this subsection,
 we will consider an admissible monotone quotiented inverse system $X^i_k$.
 
\begin{remark} The admissible monotone quotiented inverse systems are constrained only up to unit scale $s_0$. Thus, our analytic properties only hold up to that scale. Consider properties such as in Inequalities \eqref{eq:doubling} and \eqref{eq:PI}, which depend on a scale $r$ and location $x$. We adopt the convention that such a property is \textit{local} if it holds for all $r<r_0$ for some $r_0$ uniform in $x$ with uniform constants. If we wish to specify the scale, we will say the property holds locally up to scale $r_0$. The constants in the property are assumed here independent of the scale. By analogy with \cite{bjornbjornlocal}, a \emph{semi-local} property is one where the property holds for every $r$ but with constants that are bounded only when $r$ lies in some bounded subset of $(0,\infty)$. In \cite{bjornbjornlocal} the constants are also allowed to depend on the location $x$, but we can avoid this dependence since all of our spaces are quasiconvex and locally doubling. To give examples: hyperbolic $n$-space $\mathbb{H}^n$ is semi-locally doubling, but the space $\N$ equipped with the discrete metric $d(x,y) = 1$ if $x \neq y$ and the counting measure is only locally doubling up to scale $1/2$.
 \end{remark}

\begin{lemma} \label{lem:piatscale}
 Let $L \geq 1$ be arbitrary. The spaces $X^i_k$ are locally measure doubling 
and satisfy a local $(1,1)$-Poincar\'e inequality up to scale $Lm^{-k}$. 
\end{lemma}

\begin{proof}
 The space $X^k_k$ satisfies these properties since
 it has bounded degree and therefore bounded geometry at scales comparable to $s_k$. 
 The rows of a quotiented inverse system satisfy the inverse-limit axioms of
 \cite{cheegerkleinerinverse}.
 Thus, by \cite[Theorem 1.1]{cheegerkleinerinverse}, the spaces $X^i_k$ satisfy  
 a Poincar\'e inequality and measure doubling
up to a scale comparable to $s_k$.  In view of 
the connectivity and bounded geometry of $X^i_k$,
this can be strengthened to hold up to any scale $Ls_k$ by appealing to  the results from \cite{bjornbjornlocal}.  
 The constants will depend on $L$ and the constants defining $X^i_k$.
\end{proof}

The quotient condition leads to doubling bounds.

\begin{lemma} \label{lem:doubling}
The spaces $X^i_k$ in an admissible monotone quotiented inverse system are doubling up to scale $r_0=1$.
\end{lemma}

\begin{proof}
 Fix $\delta,C$ as in Condition (9) of Definition \ref{d:iqs}
 and equation \eqref{eq:quotient}. Let $x \in X^i_k$, and $r \in (0,\delta/2)$ be fixed.
 We will show doubling up to scale $\delta/2$, from which we can apply the
 local-to-semi-local techniques from \cite{bjornbjornlocal} to obtain a doubling constant 
at unit scale.
 By Lemma \ref{lem:piatscale} it follows that $X^i_k$ are 
$D_0$-doubling up to scale $2Cs_k$ for some $D_0$. If $r<C s_k$, then the doubling bound 
for $B(x,r)$ follows from this. Thus, assume $r>Cs_k$. Let $l<k$ be such
 that $\delta s_{l+1} < 2r \leq \delta s_l$ and pick a $z \in (q^i_{lk})^{-1}(x)$.
 We have from the quotient condition 
$$
B(z,\delta s_{l+1}) \subset (q^i_{lk})^{-1}(B(x,r)) \subset (q^i_{lk})^{-1}(B(x,2r)) \subset B(z,2Cs_l)\, .
$$ The desired 
doubling then follows  by a direct computation.
 Note that, since $s_{l}/s_{l+1} \leq M$, we have $s_l \leq \frac{M}{\delta} r$.
 \begin{align*}
      \frac{\mu^i_k(B(x,2r))}{\mu^i_k(B(x,r))} & = \frac{\mu^i_l((q^i_{lk})^{-1}(B(x,2r)))}{\mu^i_l((q^i_{lk})^{-1}B(x,r))}\\
      & \leq \frac{\mu^i_l(B(z,2Cs_l))}{\mu^i_l(B(z,\delta s_{l+1}))} \\
      & \leq D_0^{\log_2(CM/\delta) + 2}
 \end{align*}
\end{proof}

Next, using similar estimates we establish the Poincar\'e inequality.

\begin{proposition} \label{thm:pidoublnig}
 The spaces $X^k_i$ satisfy a local $(1,1)$-Poincar\'e inequality and doubling at unit scale. 
\end{proposition}
\begin{proof} Let $C,\delta$ be the constants from the quotient condition. The doubling was shown already in Lemma \ref{lem:doubling}. It suffices to prove the Poincar\'e inequality. We show that the Poincar\'e inequality holds up scale $r<\delta/2$. Fix $x \in X^i_k$, $r \in (0,\delta/2)$ and a Lipschitz function $f \co X^i_k \to \R$. By Lemma \ref{lem:piatscale} we have a Poincar\'e inequality on $X^i_j$ up to scale $C s_j$ for all $i \leq j$. Suppose $C_{PI}$ is the constant of this Poincar\'e inequality. We can thus assume $r>Cs_k$. Then, choose $l<k$ so that $\delta s_{l+1} < r \leq \delta s_{l}.$ Also, since $s_{l}/s_{l+1} \leq M$, we have $s_l \leq \frac{M}{\delta} r$. Denote $\Omega \defeq (q^i_{lk})^{-1}(B(x,r)) \subset X^i_l$ and $z \in (q^i_{lk})^{-1}(x)$. We know that
$$B(z,\delta s_{l+1}) \subset \Omega \subset B(z, Cs_{l}),$$

The Poincar\'e inequality then follows from the following computation.
\begin{align*}
\avint_{B(x,r)} |f-f_{B(x,r)}| ~d\mu^i_k & ~\leq~  2\vint_{B(x,r)} |f-(f\circ q^i_{kl})_{B(z, Cs_{l})}|~d\mu^i_k \\
                                    & =  2\vint_{\Omega} |f \circ q^i_{kl} -(f\circ q^i_{kl})_{B(z, Cs_{l})}|~d\mu^i_l \\
                                    & \leq  2D^{\log_2(CM/\delta)+1}\vint_{B(z,Cs_l)} |f-(f\circ q^i_{kl})_{B(z, Cs_{l})}|~d\mu^i_l \\
                                    & \leq  2D^{\log_2(CM/\delta)+1}  C\frac{M}{\delta} r C_{PI} \vint_{B(z, Cs_l)} \Lip[f \circ q^i_{lk}]~d\mu^i_l \\
                                    & \leq   2D^{2\log_2(CM/\delta)+2}  C\frac{M}{\delta} C_{PI} r \vint_{B(x, CM/\delta r)} \Lip[f]~d\mu^i_l.
\end{align*}
In the above, on the third line we used the Poincar\'e inequality up to scale  $Cs_l$. 
On the last line we used the almost everywhere equality 
$\Lip[f \circ q^i_{lk}] = \Lip[f]\circ q^i_{lk}$, since $q$ is a local isometry 
everywhere except on a discrete set of points. We also used the doubling property,
 a change of variables with $q^{i}_{lk}$, and 
 $B(x,r) \subset q^{i}_{lk}(B(x,Cs_l)) \subset B(x,CM/\delta r)$.
\end{proof}

\begin{remark} By \cite{bjornbjornlocal}, it follows from Proposition \ref{thm:pidoublnig} that the
 Poincar\'e inequality and doubling property also hold semi-locally. Also, combined with the Gromov-Hausdorff convergence in the following sections, we obtain that the limit spaces also satisfy Poincar\'e inequalities (see e.g. \cite{ChDiff99, keith2003modulus}).%
\end{remark}

\subsection{Approximation Lemmas} \label{subsec:approx}

We will need the following lemmas concerning the distance distortion between the spaces $X^k_l$ by the maps $q$ and $\pi$. Throughout, we will assume that $X_k^l$ is an admissible quotiented inverse system where each $X^k_i$ is compact.

We can first define $X_n^\infty$ as the inverse limits of the 
rows of the diagram. That is, take $$X_n^\infty = \left\{(x_n^k)_{k = n}^\infty \ |\  x_n^k \in X_n^k,\  \pi^{ij}_n(x_n^i) = x_n^j\right\},$$
and define a distance by
$$
d_{n,\infty}\left((x_n^k)_{k = n}^\infty,(y_n^k)_{k = n}^\infty \right) 
= \lim_{k \to \infty} d(x_n^k, y_n^k)\, .
$$
Since the sequence is increasing and bounded, the limit exists.

The maps $q^i_{kl} \co X_k^i \to X_{l}^i$ induce natural maps $q^\infty_{kl} \co X_k^\infty \to X_{l}^\infty$ since they induce
maps of the tails of the inverse limit sequences. Then, define $X_\infty^\infty$ as the
direct limit of the sequence $X_n^\infty$. 

\begin{remark}[Direct limits]
\label{r:dl}
Recall, that spaces $Y_l$ form a {\it directed system of metric
 spaces} if there are maps $q_{kl} \co Y_k \to Y_l$ for 
 any $k<l$ which are $1$-Lipschitz and surjective, and 
 such that $q_{kl} \circ q_{ln} = q_{kn}$. Then, one can 
 define the {\it direct limit} as  follows:
 $$
Y_\infty := \Big\{(y_k)_{k = 1}^\infty \ |\  y_k \in Y_k, \ q_{kl}(y_k)=y_l\Big\}\,,
$$
$$
d_{\infty}\left((a_k)_{k = 1}^\infty,(b_l)_{l = 1}^\infty \right) = \lim_{n \to \infty} d(a_n, b_n)\, .
$$
There are natural maps
 $q_{k\infty} \co Y_k \to Y_\infty$. The system
  is called a {\it measured direct limit}, if in addition, every space 
 $Y_k$ possesses a measure $\mu_l$ such that $q_{kl}^*(\mu_k) = \mu_l$. The induced 
map $q_{l\infty} \co Y_l \to Y_\infty$ defines a measure
 on $Y_\infty$ by $q_{l\infty}^*(\mu_l) = \mu_\infty$, which is independent of the choice of $l$. The 
 measure $\mu_\infty$ is called the \textit{direct limit measure}.
\end{remark}

In our cases, the spaces $X_l^\infty$ form a directed sequence, 
so they define a direct limit space $X_\infty^\infty$. The induced quotient 
maps $q_{l\infty}^\infty \co X_l^\infty \to X_\infty^\infty$ may decrease distances, 
but the quantity can be controlled using the following lemma. We note that it is an 
easy exercise to verify that the maps $q_{l\infty}^\infty$ satisfy the same quotient
 condition as in \eqref{eq:quotient} with $X^i_k$ replaced by $X^\infty_k$.

\begin{lemma}\label{lem:qdist} Assume that the spaces $X_k^k$ are compact. 
Assume also that $x,y \in X_\infty^\infty$ with $d(x,y) = r \leq 1$, and 
 that $l\geq 0$ is such that $\delta s_{l+1} \leq r$. Then for
 any lifts $a,b\in X^\infty_l$ such that  $q_{l\infty}^{\infty}(a)=x, q_{l\infty}^{\infty}(b)=y$, 
we have:
$$
d(x,y) \asymp d(a,b)\, .
$$
\end{lemma}

\begin{proof}  If $d(x,y) \geq \delta$, the assertion is clear since then $l=0$, 
and $C \geq \diam(X^\infty_0) \geq d(a,b) \geq r$ for some $C$. 
Thus we can assume $d(x,y)<\delta$. It then suffices to prove the assertion
 for the smallest $l$ such that $\delta s_{l+1} \leq r$.  
In this case, we also have $r < \delta s_l$. 
By the quotient condition $d(a,b) \leq Cs_l \leq CM s_{l+1} \leq \frac{CM}{\delta} r$,
 which gives the desired conclusion.
\end{proof}

Recall  that by \cite[Estimate 2.14 and Section 2.5]{cheegerkleinerinverse},
if $a,b \in X^\infty_l$, then
\begin{equation}\label{eq:distest}
d(a,b) - L s_l\leq d(\pi^{\infty l}_l(a), \pi^{\infty l}_l(b)) \leq d(a,b).
\end{equation}

\begin{lemma}\label{lem:approx} If $x,y \in X_\infty^\infty$ are two points with $d(x,y) = r \leq 1$. 
Then if $\max\{\delta s_{k+1},2Ls_k\} \leq r$ and $x_k, y_k \in X_k^k$, and $x'_k, y_k' \in X_k^\infty$ are any points
 such that $q_{k\infty}^{\infty}(x_k')=x, q_{k\infty}^{\infty}(y_k')=y$, and
 $\pi^{\infty k}_k(x_k')=x_k, \pi^{\infty k}_k(y_k') = y_k$, then
$$r \asymp d(x_k,y_k) \asymp d(x_k', y_k').$$ 
\end{lemma}

\begin{remark} Such a sequence $\{x_k\}$ 
is called an {\it approximating sequence for $x$}. 
The points $x_k$ individually are called {\it approximants for $x$}.
\end{remark}

\begin{proof} Fix an $L$ as in \eqref{eq:distest}. Choose any $k\geq 1$ such that $\max\{\delta s_{k+1},2Ls_k\} \leq r$. 
 Choose any lifts $x_k',y_k' \in X^\infty_k$ such that
 $q_{k\infty}^{\infty}(x_k')=x, q_{k\infty}^{\infty}(y_k') = y$, 
and set $\pi^{\infty k}_k(x_k')=x_k, \pi^{\infty k}_k(y_k') = y_k$.
 From Lemma \ref{lem:qdist}, we have
$$
r \leq d(x_k',y_k') \lesssim r\, .
$$
Also, from \eqref{eq:distest},
$$
r/2\leq d(x_k', y_k') - Ls_k\leq d(x_k, y_k) \leq d(x_k', y_k') \lesssim r\, .
$$
This gives the desired estimates.
\end{proof}

\subsection{Gromov-Hausdorff convergence of quotiented inverse system}
In this section, we prove that the sequence $X_k^k$ converges in the Gromov-Hausdorff sense. See, for example, \cite{bbi, fukaya, keith2003modulus, cheegerkleinerschioppa, ChColI} for terminology.
 In general, if $Y_k$ is a direct limit system with bounded diameter, one can show that direct 
systems behave well under Gromov-Hausdorff convergence. 

\begin{lemma}\label{lem:ghconvergence} If $Y_1$ is compact, then any direct limit system $Y_i$ as in Remark \ref{r:dl} also
 converges in the Gromov-Hausdorff sense to $Y_\infty$. Further, 
if the system is measured and $Y_\infty$ is equipped
 with the direct limit measure, then
the sequence converges in the measured Gromov-Hausdorff sense.
 \end{lemma}

\begin{remark} It is possible to derive better bounds on the distances for quotiented inverse systems.
 However, this is technical, and we do not need it here.
\end{remark}

\begin{proof}
Let $A_1= N_\epsilon$ be an $\epsilon$-net for $Y_1$. Then $A_l = q_{l1}(N_\epsilon)$ 
define a direct system of finite metric spaces. Clearly, by a diagonal argument, such a net 
converges to $A_\infty$. The Gromov-Hausdorff distance between $A_l$ and 
$Y_l$ is of the order $\epsilon$. Consequently, by a diagonal argument sending $\epsilon
$ to $0$ and $l$ to infinity it follows that $Y_l$ converges to $Y_\infty$.  

We claim that if the 
Gromov-Hausdorff approximation maps (here $q_{l\infty}$) are measurable, and the 
push-forwards of the 
measures under these approximations coincide with the limit measure, then the sequence 
converges in the 
measured Gromov-Hausdorff sense. This can be seen by considering any ball 
$B(x,r) \subset Y_\infty$, such that $\partial B(x,r)$ has measure $0$, and showing that the 
measures of the balls $B(x_n,r) \subset Y_n$ converge to 
$\mu_\infty(B(x,r))$ if $q_{n\infty}(x_n)$ converge to $x$. Since this holds for almost any 
$r$,  the 
desired weak convergence follows.
\end{proof}

The above procedure defines the direct limit $X_\infty^\infty$ and maps 
$q^\infty_{k\infty} \co X_k^\infty \to X_\infty^\infty$, which are $1$-Lipschitz. 
To proceed further,  we will need a general lemma concerning Gromov Hausdorff limits.

%

\begin{lemma} \label{lem:gromhauslimit} If $i_n,k_n \to \infty$ are any sequences, 
then $\lim_{k \to \infty} X_{k_n}^{i_n} = X_\infty^\infty$ in the measured Gromov-Hausdorff sense.
\end{lemma}

\begin{proof}
 Relation \eqref{eq:distest} implies that the Gromov-Hausdorff distance can be estimated by 
 $$
d_{GH}(X_{k_n}^{i_n}, X_{k_n}^\infty) \asymp s_{k_n}\, .$$ 
Since $X_{k_n}^\infty$ converges in the Gromov-Hausdorff sense to $X_\infty^\infty$, 
it follows that $X_{k_n}^{i_n}$ also converges to $X_\infty^\infty$. 
Thus, we have Gromov-Hausdorff approximation maps, 
which are $1$-Lipschitz $\pi_{k_n}^{\infty i_n} \co X_{k_n}^\infty \to X_{k_n}^{i_n}$, 
and $q_{k_n\infty}^\infty \co X_{k_n}^\infty \to X_\infty^\infty$, which preserve the measures. 
Then, arguing as in Lemma \ref{lem:ghconvergence}, 
using convergence of measures of balls and 
 with an appropriate measurable Gromov-Hausdorff approximation constructed from these maps between $X_{k_n}^{i_n}$ and $X_\infty^\infty$, we see that $X_{k_n}^{i_n}$ also converges in the measured Gromov-Hausdorff sense to $X_\infty^\infty$.
\end{proof}


\subsection{Quotient condition} \label{subsec:quotientcondition}

That the quotient condition \eqref{eq:quotient} holds can be ensured in various ways. 
One way is the following.

\noindent
{\it The star quotient condition.}
We stipulate that the identifications of $q^k_k$ occur only within half-stars
 of $X_{k-1}^{k-1}$, i.e for every $x \in X^k_k$, there exists a $v \in X_{k-1}^{k-1}$ such that for 
$\alpha>1/2$,
\begin{equation}
\label{eq:quotB}
 (q^{k}_k)^{-1}(x) \subset (\pi^{k-1}_{k-1})^{-1}\left(\overline{B(v,\frac{s_k}{2}-\alpha s_{k+1})}\right)\, .
\end{equation}
 In particular, since the maps $q$ are simplicial,
the identifications can only occur at subdivision points,
 or along an entire edge of a subdivision. Therefore,
the maps $q$ can only identify points $a,b$ for which $\pi^{k-1}_{k-1}(a),\pi^{k-1}_{k-1}(b)$ 
lie within the star of a vertex $v$, and within the edges with distance
 $\lfloor m_l/2 \rfloor -1$ from $v$.
\vskip3mm

The star-quotient condition  implies that any identifications 
by $q$ occur in the vicinity of vertices. This is ensured by preventing
 identifications by $q$ for the midpoints of edges if $m_k$ even, and along the middle 
edges of subdivisions or its endpoints if $m_k$ is odd.

\begin{lemma}
\label{lem:lifting}
  Let $A\subset X^k_k$ denote a connected set.
 Then $\pi^{k-1}_{k-1}((q^{k}_k)^{-1}(A)) \subset \overline{A} \subset X^{k-1}_{k-1}$,
 where $\overline{A}$ is a connected set and 
$\diam(h_{k-1}^{k-1}(\overline{A})) \leq \diam(h_k^k(A)) + s_{k-1}-2\alpha s_{k}.$
\end{lemma}

\begin{proof} 
Let $A' = \pi^{k-1}_{k-1}(q_k^{-1}(A))$. We will  construct 
$\overline{A}$ from $A'$ in such a way as to ensure that
 it is connected.

For each $x \in A$,
 let $P_x = \pi_k^k((q^{k}_k)^{-1}(x))$. This set may not be connected. However, by 
assumption there exists a vertex $v$ such that $P_x \subset \overline{B(v,\frac{s_{k-1}}{2}-\alpha s_{k})}$. We can define $S_x=P_x$ if $P_x$ is a singleton.

 If $P_x$ is not a singleton, then either the shorter arc 
between $h_k^k(x)$ and $h^{k-1}_{k-1}(v)$ is directed 
away from $h^{k-1}_{k-1}(v)$, or the opposite holds. 

In the first 
case let $S_x$  the edges of the subdivision of $X^{k-1}_{k-1}$ 
strictly contained in $\overline{B(v,\frac{s_{k-1}}{2}-\alpha s_{k})}$
 which are directed out of $v$.
In the second case, let  all of the edges be directed towards $v$.
Then, $P_x \subset S_x$ and $S_x$ is connected. Define
$\overline{A} := \bigcup_{x \in A} S_x$. 
 Since $S_x$ intersects  $A'$
we obtain $h_{k-1}^{k-1}(S_x) \cap h_{k}^{k}(A) \neq \emptyset$.
This, combined with
$\diam(h_{k-1}^{k-1}(S_x)) \leq \frac{s_{k-1}}{2}-\alpha s_{k}$, gives
$$
\diam(h_{k-1}^{k-1}(\overline{A})) \leq \diam(h_k^k(A)) + s_{k-1}-2\alpha s_{k}\, .
$$

Now, we show that $\overline{A}$ is connected. Note that by construction, each set $S_x$
 is connected since it is a union of edges on one side of a star or a singleton.
For each $a,b \in \overline{A}$ there are corresponding points $x,y \in A$ such that 
$a \in S_x, b \in S_y$. Moreover, there is a chain of edges (or parts of edges)  
$e_0, \dots e_n$ connecting $a$ to $b$ in $A$ by connectivity. Since the maps in question are simplicial,  for each edge $e_i$ the set $\pi^{k-1}_{k-1}(q_k^{-1}(e_i))$ is either a single edge, 
or finitely many edges.

 Pick $f_i$ in each of these collections.
 For consecutive edges $e_i, e_{i+1}$, let $p_i$ be their shared endpoint. 
By construction $f_i \cup S_{p_i} \cup f_{i+1}$ is connected. Thus,
 $\bigcup f_i \cup \bigcup S_{p_i} \cup S_a \cup S_b$ is a connected finite union of edges 
within $\overline{A}$. 

\end{proof}

The following Lemma is  a direct corollary of the estimate in \eqref{eq:distest}.
\begin{lemma}\label{lem:inversest} There is a $C \geq 1$ so that the following holds. If $A \subset X^i_k$, and $j>i$, then
$$\diam((\pi^{ij}_k)^{-1}(A)) \leq \diam(A) + Cs_j.$$
\end{lemma}

. 

\begin{lemma} \label{lem:quotbound} Assume $\delta<2\alpha-1$,  that $A \subset X^k_k$ is a connected subset with $\diam(A) < \delta s_{l+1}$, and that $l \leq k$. Then
for some $x \in X^k_l$ and universal $C \geq l$,
$$
(q^k_{lk})^{-1}(A) \subset B(x,Cs_{l})\, .
$$
\end{lemma}

\begin{proof}
Let $A \subset X^k_k$ denote a connected subset and $l \leq k$ with $\diam(A)<\delta s_{l+1}$. 
We construct sets $A_l \subset X^{k-l}_{k-l}$ by setting $A_0 = A$ 
and then recursively applying Lemma \ref{lem:lifting} to get connected sets $A_i \subset X^{k-i}_{k-i}$ with the property that $\pi^{k-i}_{k-i}((q^{k-i}_{k-i})^{-1}(A_i)) \subset A_{i+1}$,
and
$$\diam(h_{k-i-1}^{k-i-1}(A_{i+1})) \leq \diam(h_{k-i}^{k-i}(A_i)) + s_{k-i}-2\alpha s_{k}\, .
$$
This is possible as long as $i<k$.

Now, by inductively applying the previous estimate we get
$$\diam(h_{l}^{l}(A_{k-l})) \leq \diam(A_0) + s_{l}-(2\alpha-1) s_{l+1} 
< s_l-(2\alpha-1-\delta)s_{l+1}<s_l.$$

In particular, since $A_{k-l} \subset X_l^l$ is connected and $h_l^l$ is an isometry on edges, it must be entirely contained in one vertex star. Thus $\diam(A_{l-k}) \leq 2s_l$. 

Finally,  
$$(q^k_{lk})^{-1}(A) \subset (\pi^{lk}_k)^{-1}(A_{l-k})\, .
$$
Therefore, by Lemma \ref{lem:inversest}
$\diam((q^k_{lk})^{-1}(A)) \leq 2s_l + Cs_l.$
\end{proof}

\begin{lemma}\label{lem:quotcond} Assume that an quotiented inverse  system satisfies
  conditions (1)--(9) the definition of an admissible quotiented inverse system, 
Definition \ref{d:iqs}. If in addition, 
 the star-quotient condition holds, then the system
 satisfies  the quotient condition, (10), in Definition \ref{d:iqs}
as well. Equivalently, the quotiented inverse system is admissible.
\end{lemma}

\begin{proof} Let $r \in (0,\frac{\delta}{M} s_l)$. If we
 apply Lemma \ref{lem:quotbound} to the set $A = B(x,r)$, since $\diam(A) < \frac{\delta}{M} s_l \leq \delta s_{l+1}$, our assertion follows.
\end{proof}

\subsection{Uniformity}\label{sec:uniformity}

We give here a simple condition to ensure that the measure on
 the limit space is uniform i.e.\ $h$-uniform for some $h$; see Definition \ref{d:huniform}.
Assume $h \co [0,\infty) \to [0,\infty)$ is an increasing function (not necessarily continuous) such that $\lim_{t \to 0} h(t) = 0$ and doubling in the sense that $h(2t) \asymp h(t)$.

\begin{definition}
\label{d:iqu} ~ \vskip1mm

\begin{itemize}
\item[1)]
An admissible quotiented inverse system is
{\it $h$-uniform} if for each edge $e$ in $X_k^k$ we have $\mu_k^k(e) \asymp h(s_k)$. 
\vskip1mm

\item[2)]
A metric measure space $(X,d,\mu)$ is called {\it $h$-uniform} if 
$\mu(B(x,r)) \asymp_{C_\mu} h(r)$ for all $0<r\leq 1$. 
The implicit constant $C_\mu$ in this comparison is called the {\it uniformity constant}.
\end{itemize}
\end{definition}

\begin{lemma}\label{lem:homogeneous} Assume that each graph $X_n^i$ has bounded diameter. Let $X_n^i$ denote an admissible 
quotiented inverse system which is $h$-uniform.  Then the limit space $X_\infty^\infty$ is $h$-uniform as well. 
The uniformity constant of $X_\infty^\infty$ depends quantitatively on the parameters 
and the uniformity constant of the system.
\end{lemma}
\begin{proof}
By the doubling condition, it suffices to prove the statement for small enough $r$.
 Thus, assume $r < \delta$. Then we can choose a scale $k \geq 1$ 
such that $\delta s_{k+1} \leq r < \delta s_k$. By the quotient condition,
for any $y \in X_k^\infty$ such that
 $q^\infty_{l\infty}(y)=x$,
 we have $B(y,\delta s_{k+1}) \subset (q^{\infty}_{l\infty})^{-1}(B(x,r)) \subset B(y,Cs_k)$.
 Also, $B(z,Cs_k)=\pi^{\infty k}_k(B(y,Cs_k)))$ 
and $B(z,\delta s_{k+1}) = \pi^{\infty k}_k(B(y,\delta s_{k+1}))) $ 
for $z = \pi^{\infty k}_k(y)$ (see \cite[Equation 2.12]{cheegerkleinerinverse}). 

However,  $X_k^k$ has finite degree and edge length is $s_k$. Moreover, it is doubling 
and each edge has measure comparable to $h(s_k)$, then
$$
\mu_k^k(B(z,s_k)) \asymp \mu_k^k (B(z,Cs_k)) \asymp h(s_k)\, .
$$
This gives 
$$
\mu^\infty_l\big((q^{\infty}_{l\infty})^{-1}(B(x,r))\big) \asymp h(r)\, ,
$$
from which it follows that
$$
\mu_\infty^\infty(B(x,r)) \asymp h(r)\, .
$$
\end{proof}

\begin{remark}
The  special case, $h(r) = r^{Q}$ of Lemma \ref{lem:homogeneous} is
particularly important  for the proof of Theorem \ref{thm:existence} and 
from the perspective of Loewner spaces as well. In this case, the uniformity 
of the limit coincides with the limit space being  $Q$-Ahlfors regular. 
\end{remark}
%
%

\subsection{Planarity} \label{sec:planarity}

The Gromov-Hausdorff limits of planar graphs, under mild assumptions are planar. For many explicit substitution rules, this can be seen by explicitly constructing a planar embedding via a limiting process.  They do not arise in the case of Loewner carpets, or in our constructions. 

The following proof was provided to us by Bruce Kleiner.The planar embeddability can be understood in terms of forbidden subgraphs: the complete bipartite graph $K_{3,3}$ and the complete graph $K_5$. Kuratowski \cite{kuratowski} showed that a graph is planar if and only if it does not contain embedded copies of such a graph. Claytor \cite{claytor} strengthened this for continua, where additionally one needs to either prevent cut-points, or certain more complex spaces $L_1,L_2$ (which are limits of planar graphs). Indeed, Claytors  examples show that when cut points are permitted, a limit of a planar graph may fail to be planar. A slight strengthening of Kuratowski was proven by Wagner \cite{wagner}, concerning graph minors. A graph minor is a subgraph obtained from a subgraph by contracting/removing edges and identifying the endpoints. Recall, a subgraph is a graph obtained by removing a subset of vertices and edges.

\begin{proof}[Proof of Proposition \ref{thm:GHplanar}] 
 Since $X$ lacks cut points, by Claytor \cite{claytor} it suffices to show that neither
$K_{3,3}$ and $K_5$ admit a topological embedding in $X_\infty$. Suppose for the sake of contradiction that $S \subset X_\infty$ were homeomorphic to one of these.

Let $V=\{v_1, \dots, v_n\} \subset S$ be the vertex
set, and $E=\{e_1, \dots, e_m\} \subset S$  be the collection of edges, which are simple curves in $X_\infty$. We will now lift the ``subgraph'' $S$ to a forbidden graph minor in $X_i$. This process will succeed once $i$ is chosen large enough.

Let $\phi_i \co X_\infty \to X_i$ be a 
sequence of Gromov-Hausdorff 
approximation maps. Let $v^i_j=
\phi_i(v_j)$. By subdividing edges, 
we may assume $v^i_j$ are also 
vertices. The maps $\phi_i$ may fail to be continuous. However, 
by taking a discretization at a scale $\epsilon>0$ for the edges $e_i$ connecting $v_a$ to $v_b$, we obtain $\epsilon$-paths $P=(p^i_0=v_a, \dots, p^i_{N_i}=v_b)$ in $X_\infty$. By an $\epsilon$-path we mean one where $2 \epsilon \geq d(p^i_{k},p^{i}_{k+1}) \geq \epsilon $. For $i$ sufficiently large $\phi_i(P)$ will be a $3\epsilon$-path.  Since $X_i$ are graphs, we can subdivide and connect consecutive points by simple paths, and remove loops to obtain simple paths connecting $v^i_a$ to $v^i_b$. Call these simple paths $e^i_k$. 

If $\epsilon$ is small and $i$ large, this process gives a subgraph $S'_i$ of $S$ with a combinatorial structure similar to $K_{3,3}$ or $K_5$, except for some overlap close to the vertices $v^i_a$ where the simple paths $e^i_k$ may intersect.  For $i$ large enough, there are no intersections between these paths outside an $10\epsilon$-neighborhood. If we collapse every edge $e$ with an end point within $10\epsilon$ distance from $v^i_a$, for each $a$, we obtain a graph minor $S_i$ which is a subdivision of $K_{3,3}$ or $K_5$, respectively. But, by \cite{wagner} no planar graph, such as $X_i$, could contain $S_i$ as a minor. This yields a contradiction.

\end{proof}

Consequently, we obtain an immediate corollary.

\begin{corollary}\label{cor:limitplanar} Suppose $X_n^n$ are part of an admissible quotiented inverse system, and $X_n^n$ are all planar and $h$-uniform with $h(r) \sim r^\alpha$ for some $\alpha \in (1,2)$, then $X_\infty^\infty$ is $\alpha$-Loewner carpet.
\end{corollary}

\begin{proof}The limit space $X_\infty^\infty$ is $\alpha$-Ahlfors regular and Loewner by Lemma \ref{lem:homogeneous}. By Theorem \ref{thm:pidoublnig} and \cite{keith2003modulus} the limit space $X_\infty^\infty$ satisfies a $(1,1)-$Poincar\'e inequality. The limit is Loewner, and thus by annular linear connectivity from \cite{heinonenkoskela} does not have cut points. Consequently, from the previous proposition, we see that $X_\infty^\infty$ is planar. Finally \ref{lemma:carpet} gives that $X_\infty^\infty$ is a carpet. 
\end{proof}

\subsection{Analytic dimension, differentiability spaces and chart functions}
In this section we show that   our admissible quotiented inverse systems have analytic dimension one;
 see
Proposition \ref{p:adiq} . This holds for all limits of systems; i.e.\ our discussion of planarity and uniformity plays no role here. It will be used in Section \ref{sec:rigidity}. We begin
by recalling some relevant background.

A metric measure space $X$ is called a \emph{differentiability} space,
 if there exist countably many measurable subsets $U_i$ and Lipschitz
 functions $\phi_i \co U_i \to \R^{n_i}$ such that $\mu(X \setminus \bigcup U_i) = 0$
 and if for any Lipschitz function $f \co X \to \R^m$ and any $i$, and almost every
 $x \in U_i$, there exists a unique $df_i \co \R^{n_i} \to \R^m$ such that
\begin{equation}\label{eq:differential}
f(y)=f(x) + df_i(x)(\phi_i(y)-\phi_i(x)) + o(d(x,y)).
\end{equation}
In general, for any continuous function $f\co X \to R^m$, then we say that it is 
\textit{Cheeger differentiable} if for almost every $x \in U_i$ and any $i$ there exists $df_i$ as before. 
Such a $df_i$ is called the Cheeger differential. Naturally, many functions
 may fail to differentiate in such a way. However, as seen in \cite{balogh} continuous 
Sobolev functions with appropriately chosen exponents (large enough so that there is 
 Sobolev embedding theorem to H\"older spaces), are Cheeger differentiable almost everywhere.

If $ M=\max_{i} n_i$, then $M$ is referred to as the analytic dimension of $X$. 
For the following proofs it will be easier to consider our monotone maps $h_k$ 
with target in $\R$. However, the proofs would be analogous for $\frac{2}{\pi}S^1$,
 and can be obtained by collapsing the circle to $\R$.  If we can take $U_i=X$ for 
some $i$, then the corresponding function $\phi_i$ is called a global chart function.

The charts $U_i$ define trivializations of the measurable co-tangent bundle $T^* X$.
 In particular, we set $T^* X|_{U_i} = \R^{n_i}$. If $U_i$ and $U_j$ intersect, then by 
the uniqueness of derivatives $\phi_i$ is differentiable with respect to $\phi_j$ almost 
everywhere, and this induces a map $\rho_{ij}(x) \co \R^{n_j} \to \R^{n_i}$, which is 
defined for almost every $x \in U_j \cap U_i$, and which is invertible. Therefore, if 
charts intersect in positive measure subsets, then $n_i=n_j$. A section of the measurable
 cotangent bundle is an almost everywhere defined function $\sigma \co X \to T^* X$, 
where almost every point is associated with a point in the fiber of the measurable 
cotangent bundle at that point, and so that it commutes (almost everywhere) via the 
change-of-chart functions $\rho_{ij}$. Each Lipschitz function $f$ has an almost
 everywhere defined section $df$, its derivative in the charts, of the measurable
 co-tangent bundle. The measurable tangent bundle is defined as the dual bundle $TX$.
 It corresponds to tangent vectors of curves, as shown in \cite{cheegerkleinerschioppa,batediff}. 
See \cite{cheegerkleinerschioppa} for a detailed discussion on measurable vector bundles
 and natural norms defined on them.

\begin{definition}\label{d:smpres}
We call a $1$-Lipschitz map $q \co (X,\mu) \to (Y,\nu)$ \textit{$C-$strongly measure preserving}
 if $q^*(\mu) = \nu$ and $\mu(B(x,r)) \leq \nu(q(B(x,r)) \leq C\mu(B(x,r))$. 
 \end{definition}
 
 The quotient maps
 in an quotiented inverse system satisfy this condition.

\begin{proposition} \label{prop:dim} Let $Y_1 \to_{q_2} Y_2 \to_{q_3} Y_3 \cdots Y_\infty$ be a directed
 system of metric measure spaces which are uniformly doubling and satisfy uniform Poincar\'e inequalities and so that all $q_{ij} = q_i \circ q_{i-1} \circ \cdots \circ q_{j+1}$  are $C$-strongly measure preserving, and so that the metric on $Y_\infty$ is the direct limit metric. If $\phi_i \co Y_i \to \R^n$ are $L$-Lipschitz global chart functions so that $\phi_i \circ q_i = \phi_{i-1}$, then the direct limit function $\phi_\infty \co Y_\infty \to \R^n$ is also $L$-Lipschitz and a global chart function. That is, $Y_\infty$ has the same analytic dimension as $Y_i$.
\end{proposition}

\begin{proof} The existence of the direct limit function $\phi_\infty$ follows from universality and that it is $L$-Lipschitz is easy to verify. Let $f_\infty \co Y_\infty \to \R$ be a Lipschitz function, then it induces Lipschitz functions $f_i \co Y_i \to \R$, and each one is differentiable outside a null set $N_i$ with derivative $df_i \co Y_i \to \R^{n}$. Further, by enlarging the sets to another null set, we can assume that $\Lip(\langle a, \phi_i \rangle)(X) \neq 0$ for all $x \not\in N_i$ and all non-zero $a \in \R^n$, and any fixed inner product on $\R^n$. This follows from the uniqueness of $df_i$, as discussed in \cite{batespeight}.

 The spaces $Y_i$ are complete, and so by Borel-regularity each $N_i$ can be assumed Borel. Let $N_\infty \subset Y_\infty$ be the union of the images of all the sets $N_i$. This is a Suslin set as the image of Borel sets, and so measurable \cite{kechris}. It is a null set, since the quotient maps preserve measure. Next, outside the pre-images of $N_\infty$ in $Y_i$, which are null, the differential $df_i$ is defined, and must commute with $q_i$, since $\Lip(\langle a, \phi_i \rangle) \neq 0$. Indeed, $q_i \circ df_i$ would be a derivative for $f_{i-1}$, and the assumption $\Lip(\langle a, \phi_i \rangle) \neq 0$ ensures almost everywhere uniqueness. 
 
 Since $df_i$ commute with $q_i$ outside the pre-images of the measurable null-set $N_\infty$, then we can canonically define  the mesurable ``differential'' $df_{\infty} \co Y_\infty \to \R^n$. We will now show that if $x \not\in N_\infty$ is a Lebesgue point of $df_\infty$, then Equation \eqref{eq:differential} is satisfied at $x$.  Fix the vector $v = df_\infty(x)$. Let $B_\epsilon = \{y \in Y_\infty \ |\  ||df_{\infty}(y)-v|| \geq \epsilon, \text{ or } y \in N_\infty\}$. It is inconsequential which norm is used on $\R^n$ here. Next, for any $\epsilon>0$ and any $\delta>0$ there is some $r_\epsilon>0$ so that for all $r \in (0,r_{\epsilon,\delta})$ we have 
$$\frac{\mu(B_\epsilon \cap B(x,r))}{\mu(B(x,r))}<\delta.$$
 Fix such an $r$. We can define $x_i$ converging to $x$ in $Y_i$, and similarly define $B^i_\epsilon = \{y \in Y_i \  |\  ||df_{i}(y)-v|| \geq \epsilon, \text{ or } y \in N_i \}$. It is clear then that
$$\frac{\mu_i(B^i_\epsilon \cap B(x_i,r))}{\mu_i(B(x_i,r))}<C\delta$$
for large $i$ and for all fixed $r \in (0,r_{\epsilon,\delta})$.

Fix a large constant $L$ depending on the Poincar\'e inequality constants. If now $y  \in B(x,r/L)$, then for $i$ large enough also its pre-images satisfy $y_i \in B(x_i,r/L)$. Since the spaces satisfy uniform Poincar\'e inequalities, then there exists a curve $\gamma_{xy}$ connecting $x_i$ to $y_i$ with length at most $Ld(x_i,y_i)$ and so that $$\int_{\gamma_{xy}} 1_{B^i_\epsilon \cup N_{i}} ~ds \leq f(\delta),$$ with $\lim_{\delta \to 0} f(\delta) = 0$. See for example \cite{sylvesterpoincare}, or alternatively consider the modulus bounds in Keith \cite{keith2003modulus}. A similar argument is also employed below in Lemma \ref{lem:modulusboundstronger}. However, then, since $\phi_i$ is differentiable almost everywhere along $\gamma_{xy}$, by chain rule $(f \circ \gamma)' = df_i \circ \gamma \cdot (\phi_i \circ \gamma)'$, and since $df_i \approx v$ on most of $\gamma_{xy}$, we would get
$$f_i(y_i)-f_i(x_i) = df_i(x)(\phi_i(y_i)-\phi_i(x)) + g(\epsilon,\delta)r,$$
where $\lim_{\delta,\epsilon \to 0} g(\epsilon,\delta) = 0$ independent of $i$. Thus, letting $i \to \infty$, and $\delta,\epsilon \to 0$, as $r \to 0$ gives the desired equation.
\end{proof}

\begin{proposition}
\label{p:adiq} Let $X_\infty^\infty$ be the Gromov-Hausdorff limit of an admissible inverse-quotient system. Then $X_\infty^\infty$ is a differentiability space with analytic dimension $1$ and a chart given by $h_\infty^\infty \co X_\infty^\infty \to \R$.
\end{proposition}

\begin{proof} By \cite{cheegerkleinerinverse} each $X_\infty^i$ is analytically one dimensional and a differentiability space with the given chart. Further, since $X_\infty^i$ satisfy uniform Poincar\'e inequalities by Lemma \ref{lem:piatscale} and since the quotient maps of admissible quotient systems are strongly measure preserving, we get from Proposition \ref{prop:dim} that $X_\infty^\infty$ is a differentiability space with the given chart.
\end{proof}

 \section{Admissible quotiented inverse systems constructed
via substitution rules}
\label{s:gsiqs}

In this section, we will construct special quotiented inverse systems whose  limits have
explicit quasisymmetric embeddings in $\R^2$.
They are constructed 
by iteratively taking copies of the space and identifying these copies with
 ``two sets of identifications'': identifications satisfying inverse limit axioms and additional
 identifications within vertex stars. 
These latter identifications are used to ensure and maintain planarity.
 These methods do not exhaust all quotiented inverse systems which lead to planar PI-spaces
as discussed in Section \ref{sec:generalconstruction}. However,
they do suffice to construct for each
$Q,Q'$, the infinitely many Loewner carpets as in Theorem \ref{thm:existence}.
\vskip2mm

Let $G_1$ be any bounded degree graph with edge length $s_1=1$.
 Let $h_1 \co G_1 \to \frac{4}{2\pi}S^1$ be a monotone function.
Define $\mu_1$  to be Lebesgue measure on each edge with unit mass. 

Let $1 \leq K_i  \leq M, 4 \leq N_i \leq M$ be  sequences of integers. Inductively, define 
spaces $\overline{G_{k}}$,$G_{k}$ equipped with measures $\overline{\mu_{k}}, \mu_{k}$ and maps $\overline{h}_k \co \overline{G_{k+1}} \to \frac{4}{2\pi}S^1, h_k  \co G_{k+1} \to \frac{4}{2\pi}S^1$ and scales $s_k$ as follows. 

Recursively define $s_{k+1} = \frac{s_k}{N_k}$
Let $Y_{k+1} = \sqcup_{i=1}^{K_{k}} G_{k}^i$ be the disjoint union of spaces, 
where $G_k^i$ are identical disjoint copies of $G_k$. To each copy assign
the measure which $1/K_{k} $ times the measure of $G_k$.
 Let $\overline{\pi}_k \co Y_{k+1} \to G_k$ denote the map which
the identifies the copies with the original $G_k$. 

Next, let $\sim_{k+1}^I$ be any equivalence relation on $Y_k$
 that only identifies  pairs of points $x,y \in Y_{k+1}$ if the following hold.
\begin{itemize}

\item[1)]
$x \in G_k^i,y \in G_k^j$ for some $i \neq j$.
\vskip1mm

\item[2)]
$\overline{\pi}_k(x)=\overline{\pi}_k(y)$.
\vskip1mm

\item[3)]
 $\overline{\pi}_k(x)$ is a
subdivision point in the graph $G_k^{/N_k}$, which is not an original vertex of $G_k^i$. 
\end{itemize}

\begin{remark}
\label{r:usually}
Usually only a subset of the subdivision points get identified by $\overline{\pi}_k$. 
\end{remark}

Denote the quotient space $\overline{G_{k+1}}=Y_{k+1} / \sim_{k+1}^I$ and assign to
 it the push-forward measure $\overline{\mu_{k+1}}$.
The map $h_k$ induces trivially a map on $Y_{k+1}$ and therefore also a map $\overline{h}_{k+1}$ 
on $\overline{G_{k+1}}$, since the identifications correspond to matching points in $G_k$.
For a point $x \in G_k$ denote by $(x,i)$ its copy in $G_k^i$. We also use $(x,i)$ to denote 
a point in $\overline{G_{k+1}}$ although the same point can have different representations in 
this way. For a point $(x,i)$ in either $\overline{G_{k+1}}$ or $Y_{k+1}$, we will call $i$ its \emph{label}. The metric on $\overline{G_{k+1}}$ is the quotient metric.
\vskip2mm

In order to satisfy Condition (6) in Definition \ref{d:iqs}, it is necessary 
that the inverse limit identifications are sufficiently dense so as
to ensure that $d_{\overline{G_{k+1}}}((q,i),(q,j)) \leq C s_{k}$ for some $C$ independent of $q,i,j$. 
This will hold under the following condition. (The terminology is adapted from Laakso \cite{laaksospace}.)
\begin{definition}\label{d:wormhole}
The pair of identified points $(x,i) \sim^I_{k+1} (x,j)$ in $Y_{k+1}$ as well as the resulting point in $\overline{G_{k+1}}$ are called \textit{wormholes}. 
For each edge $e$ in $G_k$, we associate a \textit{wormhole graph} $G_{k,e}=(V_k, E_{k,e})$. The vertex set is given by the set of labels $V_k =\{1, \dots, K_k\}$ and 
edge set $E_{k,e}$ by all pairs $(i,j) \in E$ if there is an $x \in e$ such that $(x,i) \sim_{k+1}^I (x,j)$.
\end{definition}
 Otherwise put, the graph $G_{k,e}$ encodes when one can change labels $i$ to $j$ while remaining within copies of 
a single edge in $G_k$. The connectivity means that one can go back and forth along copies of an edge $e$ in $G_k$ within $\overline{G_{k+1}}$ to change the labels. In particular, $d((q,i),(q,j)) \leq (L+2) s_{k} \leq Cs_{k}$, where $L$ is the length of the longest path in the wormhole graph $G_{k,e}$ for the edge $e$ containing $q$.

 

The identifications $\sim^I_{k+1}$ are restricted to identify points 
with $\overline{\pi}_k(x)=\overline{\pi}_k(y)$. Next, we allow additional identifications.
 These identifications are denoted by $\sim_{k+1}^Q$, for which the following conditions must hold.

\begin{enumerate}
 \item The condition $(x,i) \sim_{k+1}^Q (y,j)$ implies that $h_{k}(x)=h_k(y)$. That is, $x,y$ might 
correspond to different points in $G_{k}$, but must have the same value for the monotone function.
\vskip1mm

 \item The condition $(x,i) \sim_{k+1}^Q (y,j)$ implies that there is a vertex $v$ in $G_k$ 
such that $x,y \in \str^{G_k}_{v}$ and $|h_k(x)-h_k(v)|\leq \frac{s_k}{2}-s_{k+1}$. 
\vskip1mm

 \item The condition $(x,i)\sim_{k+1}^Q (y,j)$ implies that $x,y$ are vertices in the
 subdivision $G_k^{/N_k}$ and not original vertices of $G_k$. This prevents too many 
identifications and ensures separation between identification points.
\vskip1mm

 \item Each equivalence class of $\sim_{k+1}^Q$ has size at most some
 universal $\Delta$ for some $\Delta$, and the equivalence classes of $\sim_{k+1}^Q$
 are disjoint from the ones for $\sim_{k+1}^I$.
\end{enumerate}

Assuming that (1)--(4) above hold, 
we define the graph
 $G_{k+1} := \overline{G_{k+1}} / \sim_{k+1}^Q$ which we equip
with the path metric and the push-forward measure $\mu_{k+1}$.

 Denote the quotient maps by $q_k^k \co \overline{G_{k+1}} \to G_{k+1}$, and the natural map by $\pi_{k}^k \co \overline{G_{k+1}} \to G_k$.
\begin{remark}
\label{r:allowed}
Note: To be clear, we allow that $\sim_{k+1}^Q$ does not make any identifications at 
all.
\end{remark}

 In order to define a quotiented inverse system, we begin by defining
 $X_k^k = G_{k}, X_k^{k+1} = \overline{G_{k+1}}$. Next, we will recursively 
define $X_k^l$ for $l \geq k$ and define maps 
$\pi_{k}^{l-1}$ and $q^l_{k+1}$, column by column. Since the cases $k=k,k+1$ are already defined, we will only need to consider
$1\leq k+1 <l$. The details are given below, where we assume that $X_n^l$ has been defined for all $n \leq l$ and  we  proceed to define $X_n^{l+1}$. 

 \begin{figure}[h!]
  \centering
    \includegraphics[width=0.4\textwidth]{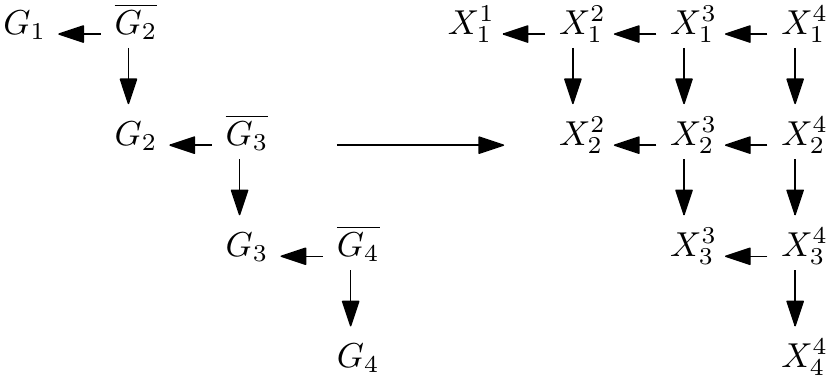}
    \caption{The graphs $G_k$ are converted into a quotiented inverse system by ``lifting identifications''.}
    \label{fig:conversio}
\end{figure}

\subsection{Additional details}
Define $X_{n}^{l+1}$ as follows. Take $M_{l}$ copies of $X_{n}^{l}$. Denote
 the copies by $X_n^{l,i}$. Denote the points in different copies by $(x,i)$,
 with $i=1, \dots, M_l$.  Again the index $i$ is referred to as the label of the point $(x,i)$. Then define an identification $\sim^I_{l+1,n}$ by

$$(x,i) \sim^{I}_{l+1,n} (x,j) \text{ if }(q^l_{nl}(x),i) \sim_{l+1}^I (q^l_{nl}(x),j).$$
Since, $\sim^{I}_{l+1}$ only identifies subdivision vertices, we must have that $x \in X_n^{l}$ is a vertex in the subdivided graph $(X_n^{l})^{/N_l}$. Finally, define $X_n^{l+1} = \sqcup_{i=1}^{M_l} X_n^{l,i} / \sim^{I}_{l+1,n}$ and equip it with the quotient metric. As above, we continue to denote points of $X_n^{l+1}$ by $(x,i)$, despite the ambiguity in the label $i$. 
 The measures on $X_n^{l+1}$ are defined by averaging the measures on each
copy of $X_n^l$.

Since $\sim^{I}_{l+1,n}$ defines a valid inverse limit step, we have a natural map 
$\pi_n^l \co X_n^{l+1} \to X_n^l$ (which sends $(x,i)$ to $x$ and is well defined).
 Also, the map $q_{n}^l \co X_{n-1}^{l} \to X_n^l$ lifts to a map $q_{n}^{l+1} \co X_{n-1}^l \to X_n^l$
for $n \geq 2$, by defining as follows:
$$q_n^{l+1}(x,i) :=(q_n^l(x),i).$$
\begin{remark}
It is easy to verify that the above is well-defined, since it respects the quotient relations $\sim^Q$ and $\sim^I$. 
\end{remark}

\begin{theorem} \label{thm:replacementpi}
 If the wormhole graphs
 are connected, then any construction as above provides 
an admissible inverse
 quotient system with spaces $X_k^i$.
\end{theorem}

\begin{proof}The numbers below refer to the conditions in Definition \ref{d:iqs}. 
Conditions (1), (5) and (8) are automatic from the fact that
the spaces arise via taking copies and identifying them along points
where the monotone function agrees. The conditions (4) and (9) follow since we distribute 
mass equally among copies. Condition (3) follows since we only identify copies along 
subdivision points, and each equivalence class has size at most $\max(\Delta, M_i)$.
 Thus, each vertex has degree at most $2 \max(\Delta,M_i)$.

The openness in condition $(7)$ follows since $X^{i+1}_k$ consists of
copies of $X^i_k$ identified along points $(x,i) \sim (x,j)$ for some
$x \in X^i_k$. That is, copies are identified along matching points.

The condition $(2)$ follows from $(6)$ and the fact that the relevant maps are simplicial. Condition $(6)$ follows for  $\pi^{k+1}_k$ from the connectivity of 
the wormhole graphs. Indeed, given $x \in G_k=X_k^k$ on an edge $e$ of $G_k=X_k^k$, the points $(x,i), (x,j)\in (\pi^{k+1}_k)^{-1}(x) \subset X^{k+1}_k$ can be connected by changing an arbitrary label $i$ to another $j$ within the wormhole graph $G_{k,e}$. This path then corresponds to a path of moving back and forth within copies of the edge $e$ in $(\pi^{k+1}_k)^{-1}(e)$ and passing through wormholes of length at most $Cs_k$. 

Next, we need to verify conditions $(6)$ for $\pi^{l}_k$ for $l>k$.  Take $x \in X_k^{l}$ and $(x,i),(x,j) \in  (\pi^{l}_k)^{-1}(x)$.  Given the previous paragraph, we can connect $(q^{l}_{kl}(x),i)$ to $(q^{l}_{kl}(x),i)$ with length bounded by $Cs_l$ only using wormholes given by $\sim^I_{l+1}$. Each of these wormholes is also a wormhole in $X^{l+1}_k$ by the definition of $\sim^I_{l+1,n}$. Thus, the path lifts to a path in $X_k^{l+1}$ with the same diameter bound.


Finally, the quotient condition $(10)$ follows by verifying the star-quotient condition 
of Subsection \ref{subsec:quotientcondition}. 
Namely, the sets $(q^{k}_k)^{-1}(x)$ for different $x \in X_k^k$ consist 
of  equivalence classes of $\sim^Q_{k}$, which are contained within vertex stars and 
which are within distance $ \frac{s_k}{2}-s_{k+1}$ from the center. 
This is exactly the star-quotient condition.
\end{proof}

\subsection{Substitution rules}
Since the quotient maps $q_k$ are only allowed to identify points within vertex stars, 
we can obtain $G_{k+1}$ from $G_k$ by using various substitution rules, 
in which vertex stars with the same ``type'' are iteratively replaced by new graphs 
containing multiple copies of the initial vertex star. These substitution rules consist of  
rules for subdivided vertex stars $\str^{G^{/2}}_{v}$ in the subdivided graph $G^{/2}$ for $v$ a vertex in $G_k$. 

\begin{remark}
\label{r:note}
Note that these vertex stars may have different types,
 depending on the degree, how $h_k$ is defined (specifying in and out edges), and on
additional data that may be 
specified (for example, the way the previous stage $G_k$ is embedded in the plane). 
The collection of all this data defines the ``type'' of a vertex.
\end{remark}
 At stage $k$ we will give a rule 
which specifies how 
each star of a given type is replaced by
 $K_k$ copies of its self with two sets of identifications, 
which are subject to the same requirements as above,
and to the following  additional requirements  as well:
  
\begin{enumerate}
 \item The wormhole graphs defined above are required to be connected.
\vskip1mm


\item
The previous identifications must lead to new vertices which fall into the same collection
of predefined types.\footnote{For example, degrees of verticies shouldn't increase beyond a certain limit. And, if additional data such as a planar embedding is preserved, then this data must be translated to the new graph.} 
\end{enumerate}

\begin{remark} 
\label{rmk:hausdim}
Suppose $G_{k+1}$ is constructed from $G_k$ by taking $K_k$ copies and 
subdividing by a factor $N_k$, and assigning equal measure to each edge. If 
$e$ is an edge of $G_{k+1}$, it follows by induction that 
$$
\mu_{k+1}(e) = \prod_{i=1}^k \frac{1}{K_i} \prod_{i=1}^k \frac{1}{N_i}\, .
$$
 Define $h(r) = 1$ for $r \geq 1$, and 
$$
h(r) = \prod_{i=1}^k \frac{1}{K_i} \prod_{i=1}^k \frac{1}{N_i} \qquad
{\rm for \,\,} s_{k+1}\leq r <s_k\, .
$$
Then we get an increasing function such that, by Lemma \ref{lem:homogeneous}, 
the limit space $X_\infty$ of $G_k$ will be a $h$-uniform space with 
constant $C_h$ depending only on $N,M,\Delta$. In particular, if we are using a
 finite set of substitution rules, then independent of the ordering of such
 substitution rules we get a uniform $C_h$.
\end{remark}
 
\section{Substitution rules and snowflake embeddings}
\label{sec:planarinversequot}

 \begin{figure}[h!]
  \centering
    \includegraphics[width=.6\textwidth]{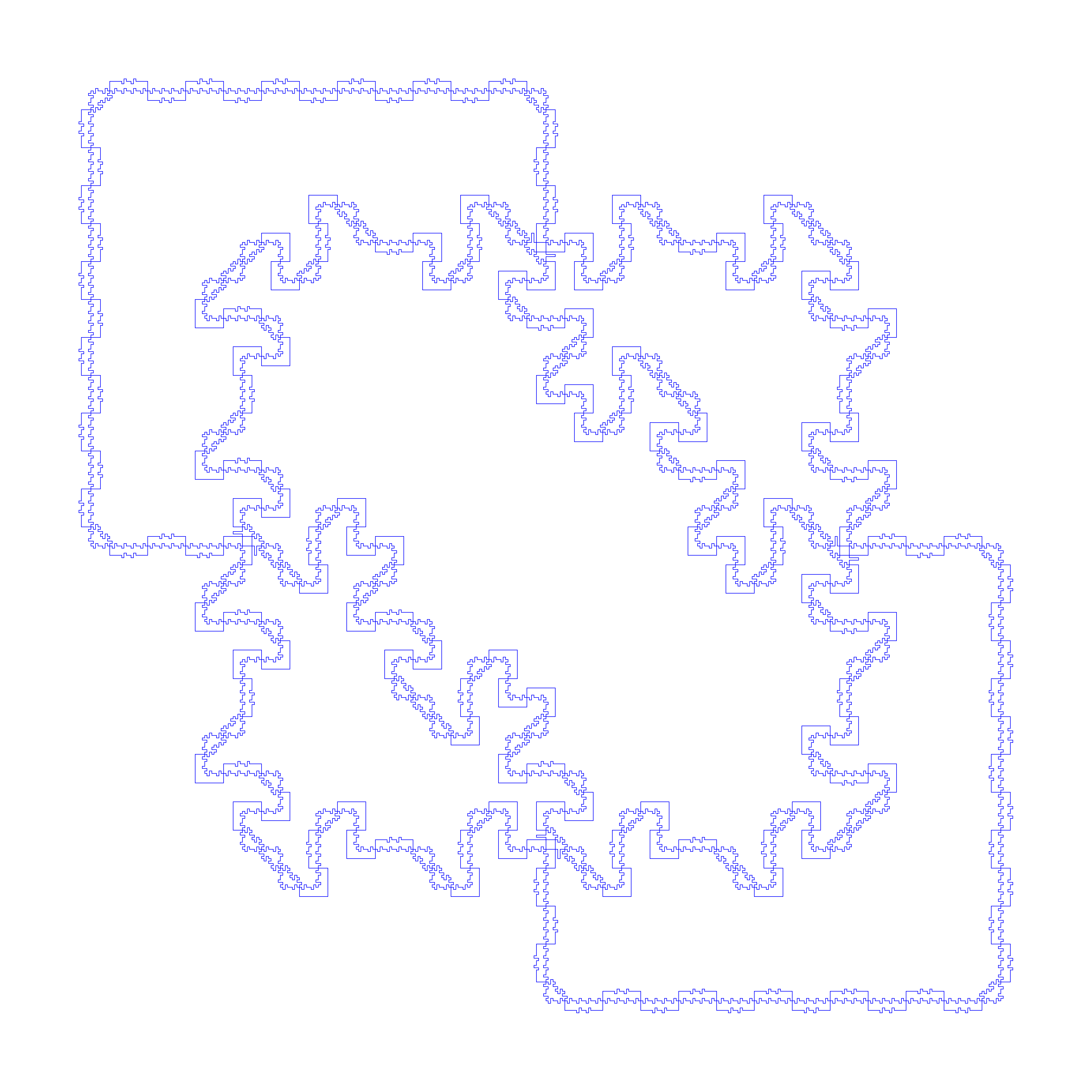}
    \caption{Three steps of our construction of an explicit planar Loewner carpet drawn using Python.}
    \label{fig:step2}
\end{figure}

In this section we use substitution rules to construct the Loewner carpets of Theorem \ref{thm:existence} as limits of certain explicit admissible quotiented inverse systems  and give their snowflake embeddings. These will be constructed using the framework of substitution rules as a special case of the construction in Section \ref{s:gsiqs}. In particular, some of the previous examples may fail to be planar. Additionally, we will impose many additional simplifying  conditions that allow us to present the examples explicitly. 

We will first do this by describing the general setting and the proof of the snow-flake embedding in Subsections \ref{ss:framework}, \ref{ss:snowflake}, and then giving a simple example in Subsection \ref{ss:basic} and then the more complicated examples in Subsection \ref{sec:generalexamples}. It is these latter constructions that resolve Theorem \ref{thm:existence}, while the simpler ones allow to give an explicit drawing in Figure \ref{fig:step2}.

\subsection{Framework}\label{ss:framework}

In what follows, the notation from Section \ref{s:gsiqs} will still be in force. The graphs $G_k$ will each be subgraphs of finer and finer integer lattices, and will possess two metrics: one as a subset of the plane and another given by the path-metric with rescaled edge lengths. The embedding will be the limit of the identity map between these different metrics.

Start with $G_0$ which is the unit square
with vertices at $(0,0),(1,0),(1,1)$ and $(0,1)$. Each graph $G_k$ will then be a subgraph of a finer integer lattice $l_k \Z^2$, where the lattice has edge length $l_k$. Set $l_0=1$. Throughout, we will assume that $l_{k+1} \leq (16)^{-1} l_k$, in order to obtain a summability condition below in Equation \eqref{eq:neighs}. The path metric on $G_k$ will be with respect to edge lengths $s_k$.

Further, we will maintain maps $h_k \co G_k \to \frac{2}{\pi}S^1$, which are isometries when restricted to edges. 
First, define the isometry (with respect to the path metric) $h_0 \co G_0 \to \frac{2}{\pi}S^1$. Since $\frac{2}{\pi} S^1$ has a natural  clockwise orientation, this induces for each edge an orientation. Therefore, $G_0$ becomes an directed graph, where each edge is directed clockwise. In $G_k$ each edge is directed so that $h_k$ is clockwise when restricted to the edge. For each of the vertices, the edges adjacent
to it are divided into in-edges and out-edges, depending on the direction of the edges.

The graph $G_{k+1}$ is obtained by substitution rules from $G_k$. For each  vertex $v \in G_k$ consider the $2^{-1} s_k$-neighborhood in the graph $G_k$, $\overline{B(v,s_k/2)} = B_v$, which is a $2^{-1}l_k$-neighborhood in the lattice.  In our substitution rules we will maintain the condition that these vertex stars have one of four types; which are determined by the identification of $G_k$ with a subgraph of $l_k\Z^2$.

\begin{itemize}
 \item[A.] If $v$ has degree two, then $B_v$ consists of an in-edge and an out-edge which form a right angle turning left.
 \vskip1mm
 \item[B.] As in A, but the turn is to the right.
  \vskip1mm
 \item[C.] The vertex $v$ has degree two and $B_v$ consists of two parallel edges adjacent to $v$.
  \vskip1mm
 \item[D.] The degree at $v$ is four with two in-edges and two out edges. The in edges are always adjacent to each other.
\end{itemize}

For the constructions in Subsection \ref{ss:basic}, we will additionally assign to each vertex a 
color, red or blue. This color is chosen so that neighboring vertices do not have 
matching color. For example, since each $v = (ml_{k},nl_k) \in l_k \Z^2$, then we 
can color a vertex red if $m+n$ is even, and blue if $m+n$ is odd. The 
substitution rules will also depend on the orientation of the edges in the lattice which is determined by 
specifying a suitably rotated version of one of the above types.


In the replacement steps, each neighborhood of a vertex $v \in G_k$ is 
replaced, abstractly, by $K_{k+1}$ copies of the neighborhood with two sets of 
identifications. These copies are then identified with isomorphic subgraphs of 
finer subgrids, and so that $v$, as a grid point, also is one of the grid points in 
the refined graph. These isomorphic copies are drawn in such a way that 
neighboring pieces only intersect at their boundaries. This drawing specifies for each new vertex its assigned type. 
The union of all the resulting copies is $G_{k+1}$ equipped with a rescaled 
metric with edge length $s_{k+1}$. Since this was constructed from copies  
$G_k$ identified at points where $h_k$ agrees, then it gives an induced map $h_{k+1} \co G_{k+1} \to \frac{2}{\pi}S^1$.

\subsection{Snow-flake embeddings}\label{ss:snowflake}

The graphs $G_k$ arise as planar subsets via replacing the vertices
 in the grid (depending on their degree, and type) by the substitution rules shown below.
Since the graphs are subgraphs of the grids, this gives natural drawings
 $f_k \co G_{k} \to l_k \Z^2 \subset \R^2$ which are homeomorphisms. We will identify $G_k$ 
with its image in the plane, but keep track of the fact that there are two metrics: one on the image, and one given by the path metric on $G_k$ with edge length $s_k$. Further, to ensure convergence in the plane, we assume that the vertices of $G_k$ are always a subset of vertices of $G_{k+1}$ (as a subset of the plane).  

We will assume that
\vskip2mm
\begin{quote}\it there is an $\alpha \in (0,1)$ so that $l_k \asymp s_k^\alpha$.
\end{quote}
\vskip2mm
This $\alpha$ will be the snowflake exponent. Without this assumption, we would obtain other moduli of continuity, and would still be able to prove quasisymmetric embeddings. This restriction, however, suffices for our examples.

The embeddings $f_k$, and the corresponding substitution rules, will be chosen so that each copy of a point $x$ on an edge $e \in G_k$ in $G_{k+1}$ will be within a $5l_k$-neighborhood of the previous edge $f_k(e)$. Note, if $A \subset X$ is a subset of a metric space $X$, then denote its $\delta$-neighborhood by
$$N_\delta(A) \defeq \bigcup_{a \in A} B(a,\delta).$$

That is, these embeddings  satisfy the following property:
\vskip2mm

{\it  If $e_0$ is an edge of $G_k=X_k^k$, and if $e_1= q_{k+1}^{k+1}((\pi^k_k)^{-1}(e_0))$, then $e_1$ is a union of edges in $G_{k+1}=X_{k+1}^{k+1}$ such that 
\begin{equation}\label{eq:inclusion}
f_{k+1}(e_1) \subset N_{5 l_{k+1}}(f_k(e_0)).
\end{equation}}
\vskip2mm

 Next, we  inductively define 
\begin{align}
\label{e:induct}
e_l & = q_{k+l}^{k+l}((\pi^{k+l-1}_{k+l-1})^{-1}(e_{l-1}))\notag\\
       &=q_{k,k+l}^{k+l}((\pi^{k+l,k}_{k})^{-1}(e_0))\, .
\end{align}
 Then, since $e_{l}$ is a union of edges in $X_{k+l}^{k+l}$, we can repeat edgewise
 the first part of the construction to obtain
 \begin{equation}\label{eq:inclusion-2}
 f_{k+l+1}(e_{l+1}) \subset N_{5l_{k+l+1}}(f_{k+l}(e_l)).
\end{equation}
Iterating \ref{eq:inclusion-2} and using the geometric decay $l_{k+1} \leq (16)^{-1}l_k$ we get:

\begin{equation}\label{eq:neighs}
f_{k+l}(e_{l}) \subset N_{5l_k\sum_{n=1}^\infty \frac{1}{16^k}}(f_{k}(e_k)) \subset \overline{N_{l_k/3}(f_{k}(e_k))}.
\end{equation}

Now, let $x \in X_\infty^\infty$, and let $x_n \in X_n^n$ such that $x \in q^{\infty}_{k,\infty}((\pi^{\infty,k}_k)^{-1}(x_n^n)) $. Recall
 from Subsection \ref{subsec:approx} 
 that such sequences are called \emph{approximating sequences}.
We can choose edges $e_n$ containing $x_n$, such that 
$e_{k+l} \subset  q_{k,k+l}^{k+l}((\pi^{k+l,l}_{k})^{-1}(e_k))$. By the same argument leading to estimate \eqref{eq:neighs} we obtain

$$
f_{k+l}(x_{k+l}) \subset \overline{N_{l_k/3}(f_{k}(e_k))}\, .
$$
In particular, the sequence $f_k(x_k)$ is a Cauchy sequence, and thus we can define $f_\infty(x)$ as its limit with
\begin{equation}\label{eq:limitpt}
f_\infty(x) \in \overline{N_{l_k/3}(f_{k}(e_k))}.
\end{equation} 
Moreover, we get the estimate
\begin{equation}\label{eq:cauchyestimates}
|f_\infty(x)-f_n(x_n)| \leq 2l_n.
\end{equation}

Also, if $x,y \in X_\infty^\infty$ are such that $x_k, y_k$ belong to non-adjacent edges $e^x_k, e^y_k$, then from Equation \eqref{eq:limitpt} and since the distance of non-adjacent edges in the lattice $l_k\Z^2$ is $l_k$, we get:
\begin{equation}\label{eq:lowerbound}
|f_\infty(x)-f_\infty(y)| \geq d(\overline{N_{l_k/3}(f_{k}(e^x_k))}, \overline{N_{l_k/3}(f_{k}(e^y_k))}) \geq l_{k}/3.
\end{equation}

At this juncture, what is left to prove is
 that the limit spaces $X_\infty^\infty$ within the above framework are PI-spaces admitting quasisymmetric embeddings. In the next sections, we will additionally specify the substitution rules and $s_k$ so that $X_\infty^\infty$ is Ahlfors regular and thus Loewner with the desired exponents.
 
 \begin{lemma} \label{lem:snow} Suppose $X_\infty^\infty$, $\alpha$ and $f_\infty$ are as above. The space $X_\infty^\infty$ satisfies a $(1,1)$ Poincar\'e-inequality and $f_\infty$ is an $\alpha$-snowflake embedding.
 \end{lemma}
\begin{proof}
Since the graphs $X_k^k=G_k$ were constructed as in Theorem \ref{thm:replacementpi}, they form the diagonal sequence of an admissible quotiented inverse system. In particular, the limit space $X_\infty^\infty$ exists and satisfies the Poincar\'e inequality by Theorem \ref{thm:pidoublnig} and Theorem \ref{lem:ghconvergence}, combined with the stability of Poincar\'e inequalities from \cite{keith2003modulus}. 

Next, we show that the mapping $f_\infty$ is uniformly $\alpha$-holder.
 For each $x,y \in X_{\infty}$, choose $s_l$ so that $s_l \asymp d(x,y)$ so that
 Lemma \ref{lem:approx} holds and  $s_l \leq d(x,y)/8$. Choose an approximating 
sequence $x_n, y_n \in X_n^n=G_n$. By Lemma \ref{lem:approx},  
$d(x_n,y_n) \asymp d(x,y) \asymp s_l$ for $n \geq l$. 

The points $x_l,y_l$ lie in 
distinct edges which are not adjacent since $d(x_l,y_l) \geq d(x,y)/2 \geq 4s_l$. 
Then, by the estimate in \eqref{eq:lowerbound} we have
$$l_l \lesssim |f_\infty(x)-f_\infty(y)|.$$

Since $d(x_l,y_l) \asymp s_l$ 
there is a bounded length edge path in $G_l$ connecting $x_l$ to $y_l$, and thus also of $f_l(x_l), f_l(y_l)$ in the lattice. Since the distance of $f_\infty(x_l)$ to $f_\infty(x)$ is controlled by $l_l$, and similarly for $y_l$ and $y$, we obtain
$$l_l \lesssim |f_\infty(x)-f_\infty(y)| \lesssim l_l.$$
Since $l_l \asymp s_l^{\alpha}$, we have 
$$|f_\infty(x)-f_\infty(y)| \asymp d(x,y)^{\alpha},$$
proving that $f_\infty$ is an $\alpha$-snowflake embedding.

Since the vertex sets of $G_k$ are contained in the vertex set of 
$G_{k+1}$, using similar techniques, one can show
that $f_\infty(X_\infty)$ is equal to the Hausdorff limit of the subsets 
$G_k$ of the plane. (For the notion of Hausdorff convergence of subsets see \cite[p. 281]{munkres}.) In particular $f_\infty(X_\infty)$ is the set arising from the
 infinite sequence of substitution rules. 
 \end{proof}

\subsection{A basic example} \label{ss:basic}
We will begin with a relatively simple example, in which 
 the substitution rules and embeddings can be made explicit, as shown in Figure \ref{fig:step2}. 

Here, at each stage we subdivide the lattice in the plane by $1/16$, giving $l_k = 16^{-k}$. The edge lengths in $G_k$ are rescaled by $32$ giving $s_k = 32^{-k}$.




Each vertex has additional to the type $A,B,C,D$ also a color, either red or blue as given above. We will now describe the substitution rules giving the identifications and the explicit isomorphic copies in the plane by figures. Let $v \in G_k$ be any vertex, and let $\overline{B(v,s_k/2)} = B_v$ be its neighborhood.

The neighborhoods of $G_k$ are replaced by taking two copies of the sub-graph and identifying them by two sets of identifications $\sim^I_{k+1}$ and $\sim^Q_{k+1}$. The first is easy to describe: $\sim^I_{k+1}$ simply identifies the boundaries of each copies, which coincide with mid-points of edges in $G_k$. 

 \begin{remark}
\label{rila}
These identifications satisfy conditions for $\sim^I_k$ imposed in Section \ref{s:gsiqs}. These correspond to the inverse limit axioms in \cite{cheegerkleinerinverse} and Definition \ref{d:iqs}.
\end{remark}

 \begin{figure}[h!]
  \centering
    \includegraphics[width=0.4\textwidth]{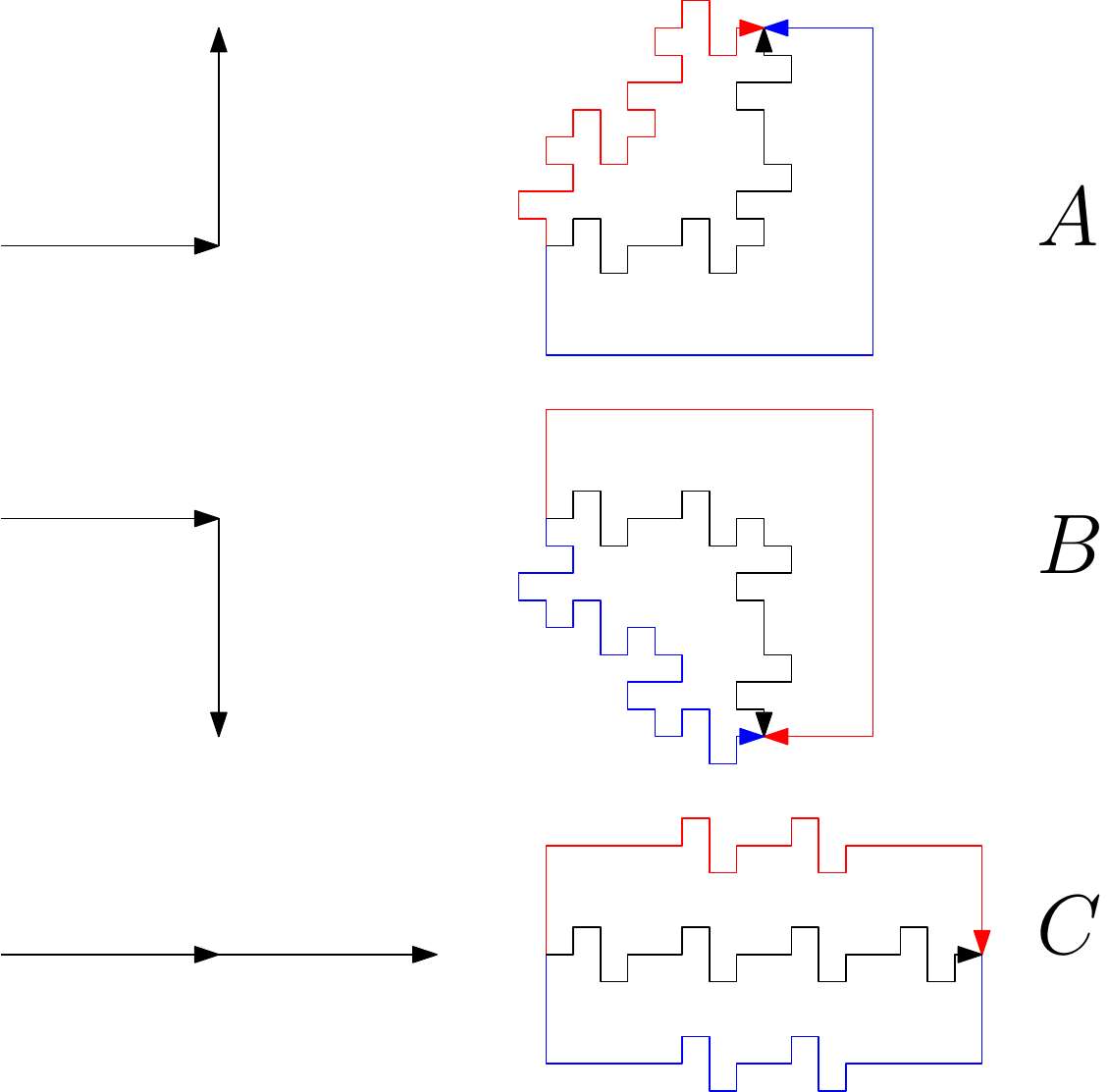}
    \caption{The substitution rules for the different types of degree two vertex neighborhoods. Whether the red or blue edges are added depends on the color of the vertex. The arrows indicate the orientation of the edges}
    \label{fig:explicit}
\end{figure}


In the case of degree two vertices, that is of type $A,B$ or $C$, these are all the identifications we will make. The isomorphic copies are rotated versions of the ones in Figure \ref{fig:explicit}, and depend on the color; the black part is always included but only one, either the blue or red, copy is added. In doing this,  the vertices and midpoints of the previous level remain fixed.



For vertices of type D, drawing in the plane 
introduces additional identifications $\sim^Q_{k+1}$, as two copies of a degree four vertex star 
must intersect. 
These identifications will depend on the orientation of the vertex, on their existing drawing in the plane, 
and on the function $h_k$, as follows. 

  \begin{figure}[h!]
  \centering
    \includegraphics[width=0.4\textwidth]{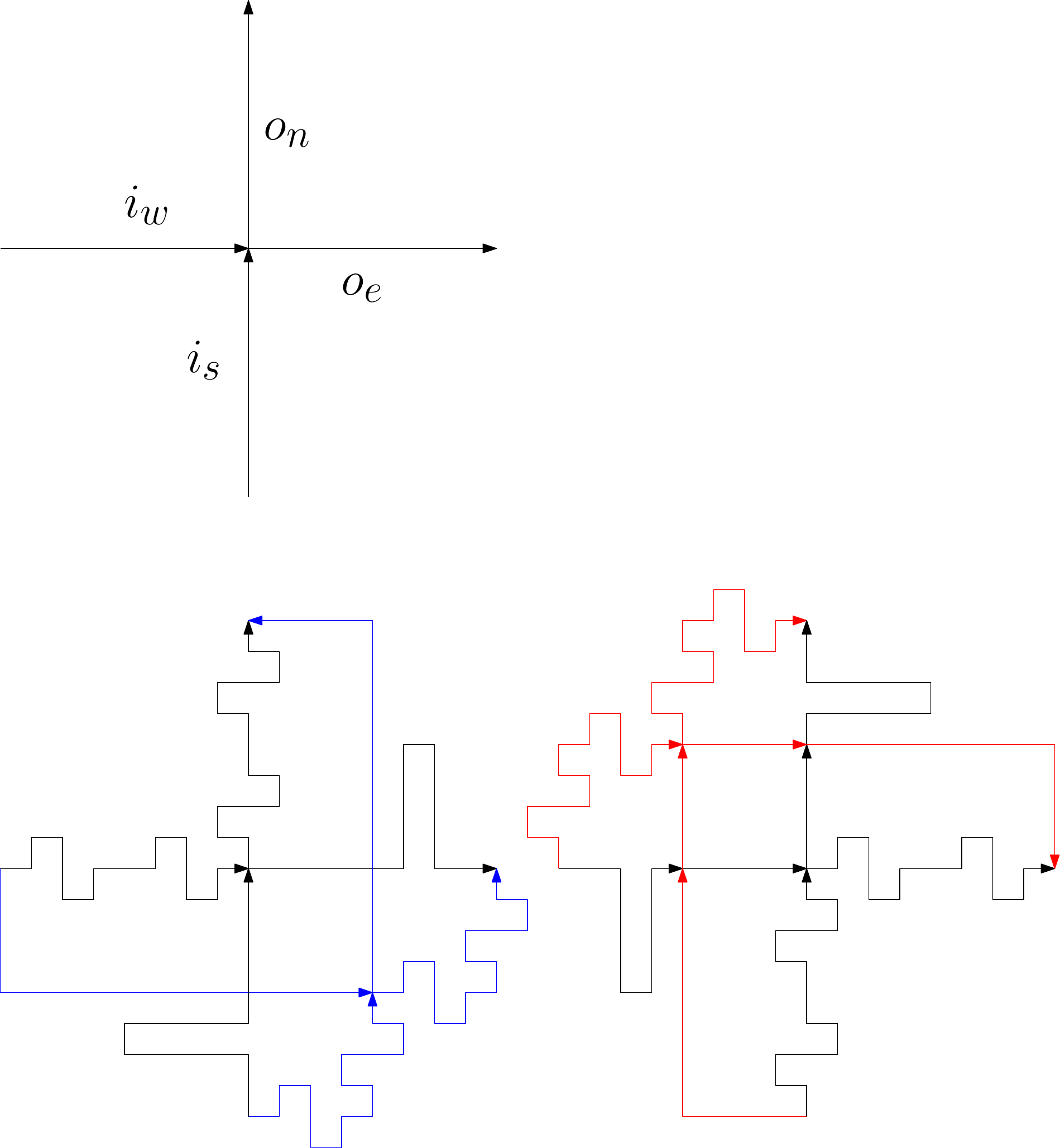}
    \caption{The substitution rule for the degree four vertex, for either the red or blue cases. The black copy corresponds to $i=1$ below.}
    \label{fig:4replacement}
\end{figure}

 Let $v$ be a degree four vertex, 
and let $i_s,i_w$ be the two in-edges (with respect to the natural 
orientation of edges induced by $h_k$), and so 
that $i_w$ is adjacent and immediately clockwise from $i_s$ in the planar drawing of $G_k$. Going in the clockwise direction, the remaining edges are the two out edges $o_n,o_e$. We will need to specify explicit identification points along the edges adjacent to $v$ by their values of of $h_k$. To describe the identifications we will
 use the following convention:

\vskip3mm

{\it 
If $t \in \frac{2}{\pi} S^1$ and $\theta \in \R$,
 then $t+\theta$ signifies the point on $\frac{2}{\pi}S^1$ obtained by traveling clockwise,
 if $\theta$ is negative, or counterclockwise, if $\theta$ is positive, along the circle by
 the distance $\theta$.}
\vskip3mm

We continue following the notation of Section \ref{s:gsiqs}. The graph $G_{k+1}$ is obtained from two copies of $G_k$, whose points are denoted by $(x,i)$ for $x \in G_k$, $i =1,2$, by the following 
identifications:
\vskip2mm

\begin{enumerate}
\item If $v \in G_k$ is a vertex of degree four which is blue, and $x \in i_e$ and $y \in i_s$,
 then we identify
$$(x,2)\sim^Q_{k+1}(y,1), \ \ \ \ \ \ \ \text{if $h_k(x)=h_k(y)=h_k(v)-4s_{k+1}$.}$$

\item  If $v \in G_k$ is a vertex of degree four which is red, and $x \in i_e$ and $y \in i_s$, then we identify
$$(x,1)\sim^Q_{k+1}(y,2),  \ \ \ \ \ \ \ \text{if $h_k(x)=h_k(y)=h_k(v)-4s_{k+1}$.}$$

\item If $v \in G_k$ is a vertex of degree four which is blue, and $x \in o_n$ and $y \in o_w$, then we identify
$$(x,2)\sim^Q_{k+1}(y,1),  \ \ \ \ \ \ \ \text{if $h_k(x)=h_k(y)=h_k(v)+4s_{k+1}$.}$$
\item If $v \in G_k$ is a vertex of degree four which is red, and $x \in o_n$ and $y \in o_w$, then we identify
$$(x,1)\sim^Q_{k+1}(y,2),  \ \ \ \ \ \ \  \text{if $h_k(x)=h_k(y)=h_k(v)+4s_{k+1}$.}$$

\end{enumerate}
\vskip2mm

Finally, the embedding of $G_{k+1}$ is obtained by first drawing the copies for each vertex neighborhood and then
allowing them 
to intersect at midpoints and at the points specified by $\sim^Q$. The interaction of multiple different rules on the graph in Figures \ref{fig:beforereplacement} is depicted in Figure \ref{fig:result}.

 \begin{figure}[h!]
  \centering
    \includegraphics[width=0.3\textwidth]{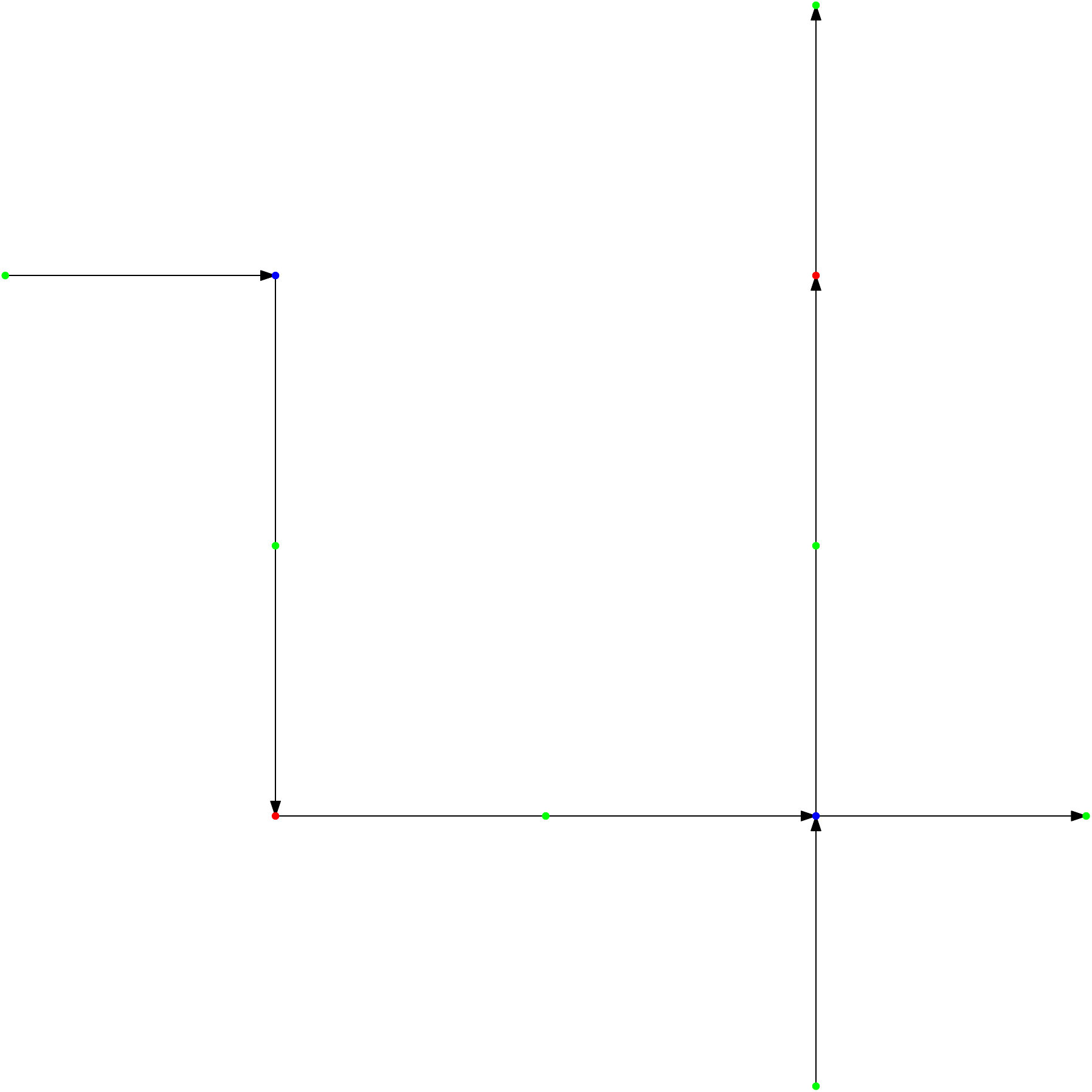}
    \caption{Before the substitution rule is applied. There are two red and two blue vertices.}
    \label{fig:beforereplacement}
\end{figure}

 \begin{figure}[h!]
  \centering
    \includegraphics[width=0.3\textwidth]{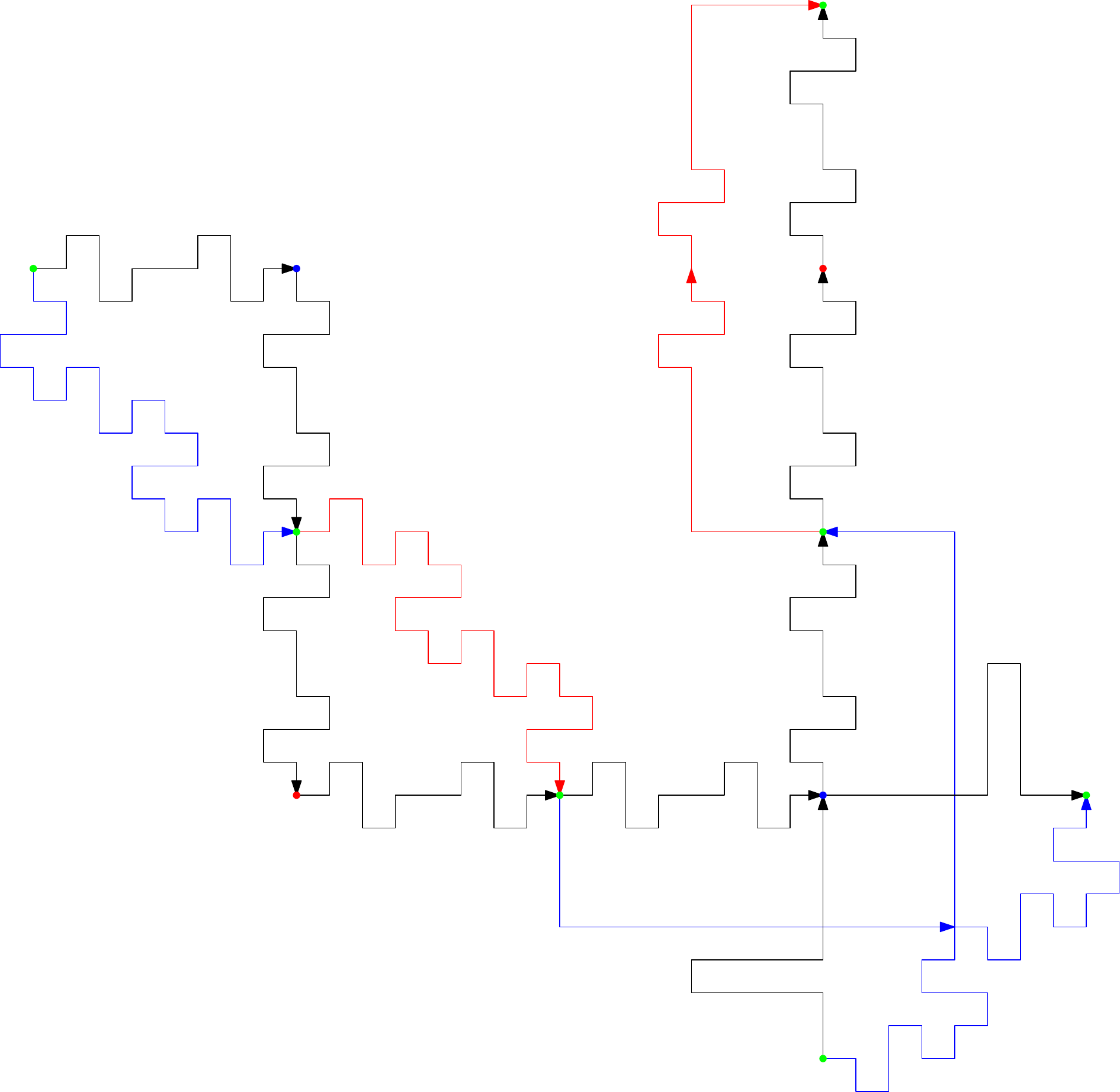}
    \caption{The result of the substitution rules applied to several neighborhoods simultaneously.}
    \label{fig:result}
\end{figure}


These substitution rules give the following example.

\begin{theorem} \label{thm:explicit} The space $X_\infty^\infty$ is $\frac{6}{5}$-Ahlfors regular, 
and satisfies a $(1,1)$-Poincar\'e inequality.
 It is quasisymmetric to the planar subset $\widetilde{X}$, which 
is $\frac{3}{2}$-Ahlfors regular, and the natural mapping $f_\infty \co X_\infty^\infty \to \widetilde{X}$
is $\frac{4}{5}$-snowflake.
\end{theorem}

\begin{proof}[Proof of Theorem \ref{thm:explicit}] By Theorem \ref{thm:replacementpi} equipping $G_k$ with the metric where each edge has length $s_{k}$ we get the
 diagonal sequence $X_k^k$ of a quotiented inverse system, which converges to 
$X_\infty^\infty$ in the measures Gromov-Hausdorff sense.  The Poncar\'e from the graphs (Theorem \ref{thm:pidoublnig}) passes to the limit space \cite{keith2003modulus}.
Since $K_k=2, N_k = 32$, it follows from   Lemma \ref{lem:homogeneous} 
and Remark \ref{rmk:hausdim}
that $X_\infty^\infty$
 is $h$-uniform, where $h(r) = r^{\frac{6}{5}}$, and thus $6/5$-Ahlfors regular. 
Now $l_k = (16)^{-k} \asymp s_k^{\frac{4}{5}}$. Thus Lemma \ref{lem:snow} with $\alpha = \frac{4}{5}$ shows that, the image of 
that $f_\infty$ is $\alpha$-snowflake. The fact that $f_\infty(X_\infty)$ is $\frac{3}{2}$-Ahlfors regular is then immediate.
\end{proof}

\subsection{General examples}\label{sec:generalexamples}

To obtain other dimensions, we need to introduce more complicated substitution rules, and alternate different ones at different levels.

We now describe three different substitution schemes $S_N, C_N$ and $WS_{N}$.
These describe how $G_{k+1}$ is obtained from $G_k$. The copies of a point $x \in G_k$ are denoted by $(x,i)$. We also describe the intermediate graphs $\overline{G_{k+1}}$ from Section \ref{s:gsiqs} and the identifications $\sim^I, \sim^Q$. These substitution rules have a free parameter $N \in \N$. 

\begin{itemize}
\item[$S_N:$] The graph $G_{k+1}=\overline{G_{k+1}}=G_k^{/N}$, that is the graph is subdivided by $N$, and the grid is equally subdivided by $N$. 
Then $s_{k+1}=s_k/N$, and $l_{k+1}=l_k/N$. There are no identifications. The map $h_{k+1}$ is equal to $h_k$. 
\vskip1mm
\item[$WS_{N}:$] We subdivide by a large factor $G_{k+1}=\overline{G_{k+1}}=G_k^{/(8(N+2N^2))}$ as before, but the drawing in the plane is wiggled to fit in a larger grid subdivided only by $8N$. The graph is drawn by appropriately wiggling the edges. This substitution rule is more naturally represented for each edge.
\vskip1mm

 \begin{figure}[h!]
  \centering
    \includegraphics[width=0.6\textwidth]{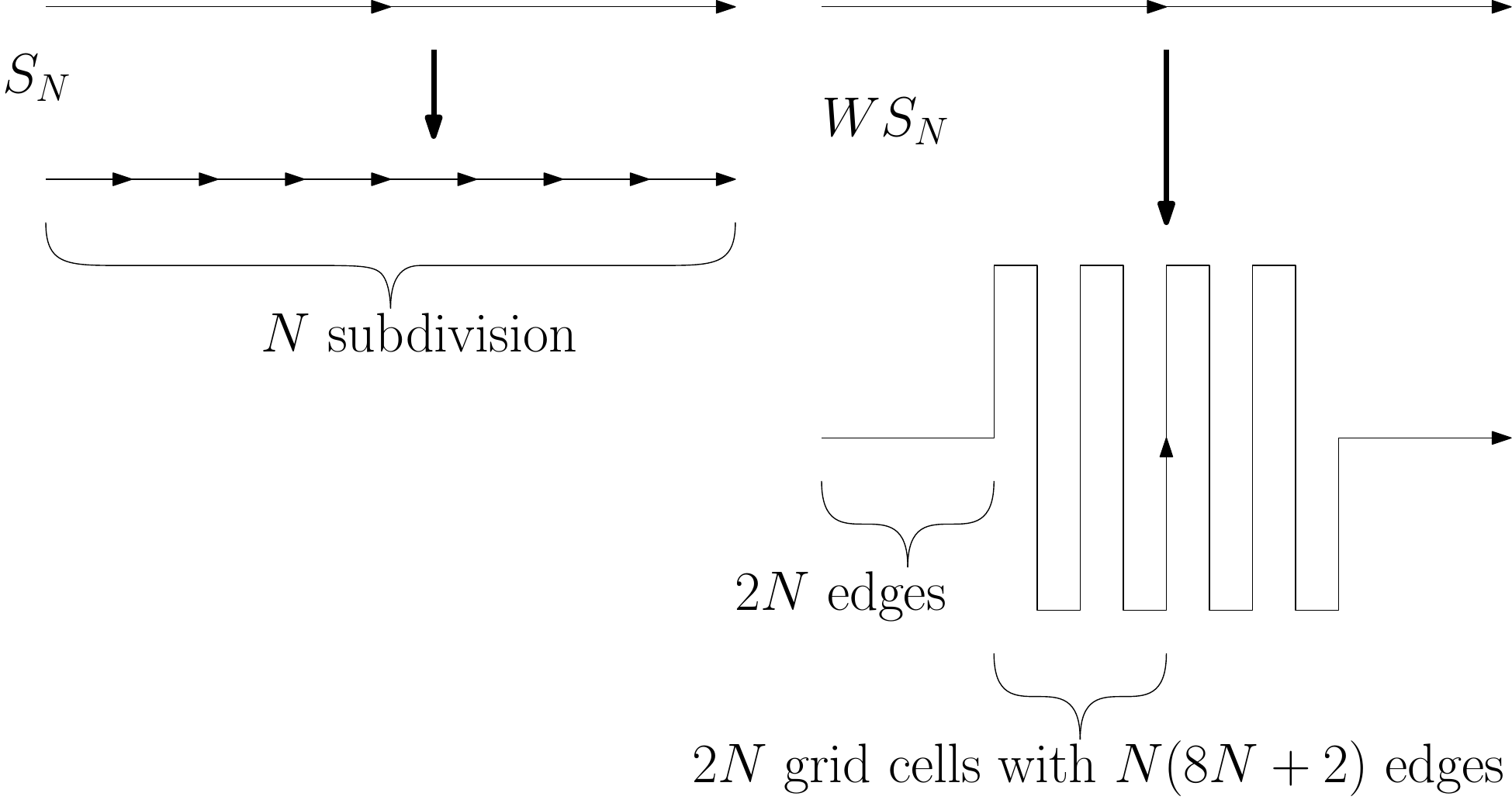}
    \caption{The figures for the substitution rules $S_N,WS_N$.  For edges oriented otherwise, one rotates the figures. The drawing rules here are for each edge of the graph, instead of vertex neighborhoods. }
    \label{fig:wsreplacement}
\end{figure}

\item[$C_N:$] The graph $\overline{G_{k+1}}$ is obtained by
 taking $2N+1$ copies of $G_k$ labeled $i=1, \dots, 2N+1$ and 
quotiented by the identifications $\sim^I_{k+1}$, which we now describe.
 Each graph is subdivided by $96N+26$ and the grids are subdivided by a factor $64N$. 
That is, $s_{k+1}=s_k/(96N+26)$ and $l_{k+1}=l_k/(64N)$. These are identified and 
embedded as shown in 
Figure \ref{fig:kfold}. 

The identifications can be expressed explicitly as follows. 
Each vertex star in $v$ has one of four types, A,B,C,D. 
The first three are subgraphs of type $D$, or rotations of one, and 
thus it suffices to describe in detail the identifications for the 
case $D$. Let $v$ be a vertex with two in edges $i_s,i_w$ 
(one from the south, and the second from the west), and two out 
edges $o_n,o_e$ (towards north and east). Every degree four vertex 
star can be rotated to this case.
 \begin{figure}[h!]
  \centering
    \includegraphics[width=0.5\textwidth]{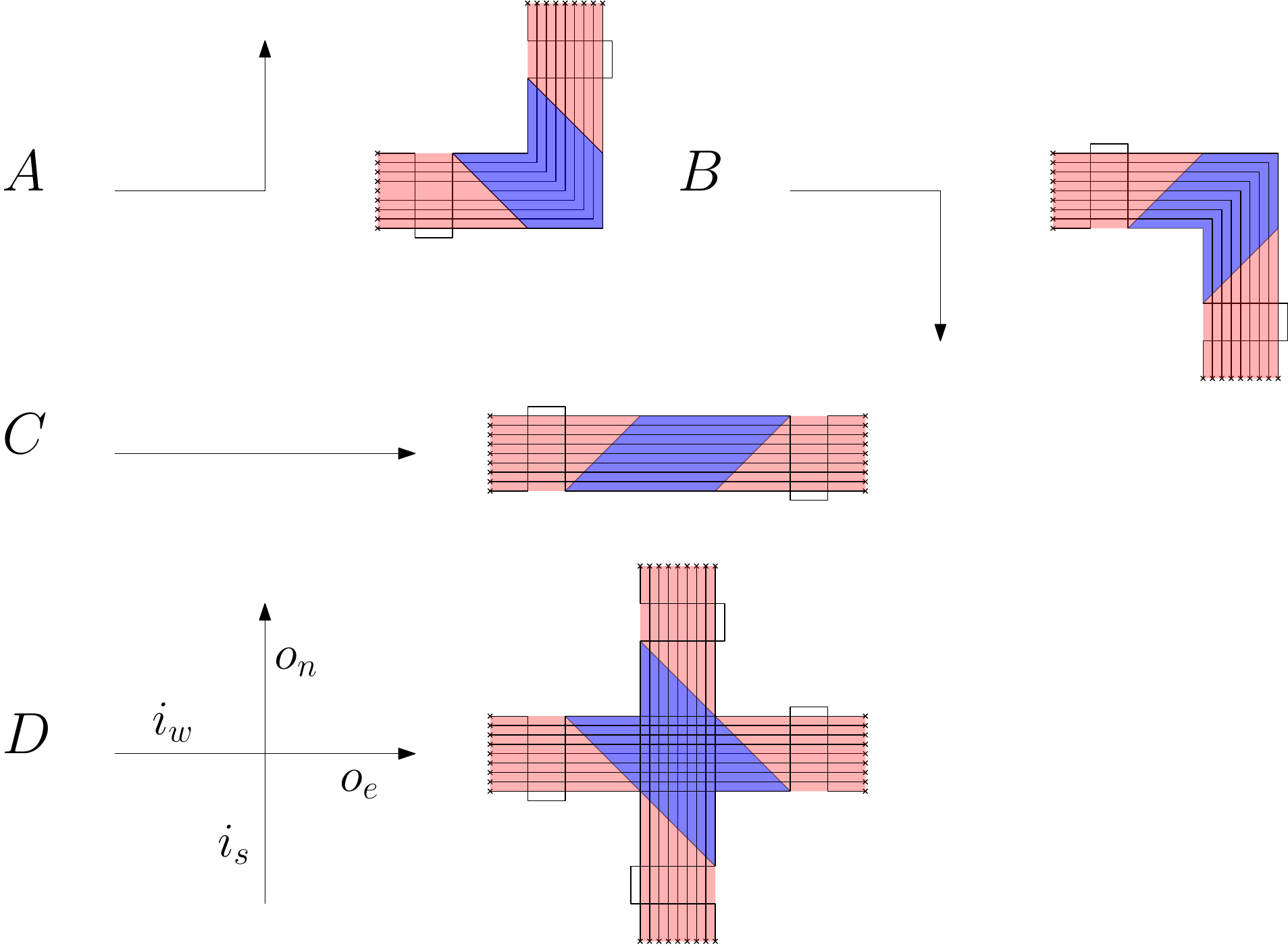}
    \caption{\textbf{$C_N$-substitution rule:} Here the neighborhood is replaced by $2N+1$ copies, $N$ above and $N$ below. Each copy has a label $i=1, \dots, 2N+1$ indexed linearly from the top to bottom. These are obtained by a translation in a diagonal direction
    of the original vertex, and then evening it out with a transitional red region shown in Figure \ref{fig:endpeace}, so that neighboring pieces match together.The identifications in the red zone are given by $\sim^I$, and in the blue by $\sim^Q$.}
    \label{fig:kfold}
\end{figure}

\begin{enumerate}
\item  $(x,1) \sim^I_{k+1} (x,j)$ if $x \in i_w$ and $h_k(x)=h_k(v)-4(2N+j-1)s_{k+1}$, or $h_k(x)=h_k(v)-4(8N+2+(2N+1-j))s_{k+1}$.
\vskip1mm

\item $(x,2N+1) \sim^I_{k+1} (x,j)$ if $x \in i_s$ and $h_k(x)=h_k(v)-4(4N+1-j)s_{k+1}$, or $h_k(x)=h_k(v)-4(8N+1+j)s_{k+1}$.
\vskip1mm

\item  $(x,1) \sim^I_{k+1} (x,j)$ if $x \in o_n$ and $h_k(x)=h_k(v)+4(2N+j-1)s_{k+1}$, or $h_k(x)=h_k(v)+4(8N+2+(2N+1-j))s_{k+1}$.
\vskip1mm

\item  $(x,2N+1) \sim^I_{k+1} (x,j)$ if $x \in o_e$ and 
$h_k(x)=h_k(v)+4(4N+1-j)s_{k+1}$, or $h_k(x)=h_k(v)+4(8N+1+j)s_{k+1}$.
\end{enumerate}
\vskip2mm

Additionally, to obtain $G_{k+1}$, we make identifications using $\sim^Q_{k+1}$ as follows.
\vskip2mm

\begin{enumerate}
\item $(x,i)\sim^Q_{k+1} (y,j)$ if $i>j$ and $x \in i_e$, $y \in i_s$ and 
$h_k(x)=h_k(y)=h_k(v)-4(i-j)s_{k+1}$.
\vskip1mm

\item $(x,i)\sim^Q_{k+1} (y,j)$ if $i>j$ and $x \in o_n$, $y \in o_w$ 
and $h_k(x)=h_k(y)=h_k(v)+4(i-j)s_{k+1}$. \\
\end{enumerate}
\end{itemize}

 \begin{figure}[h!]
  \centering
    \includegraphics[width=0.5\textwidth]{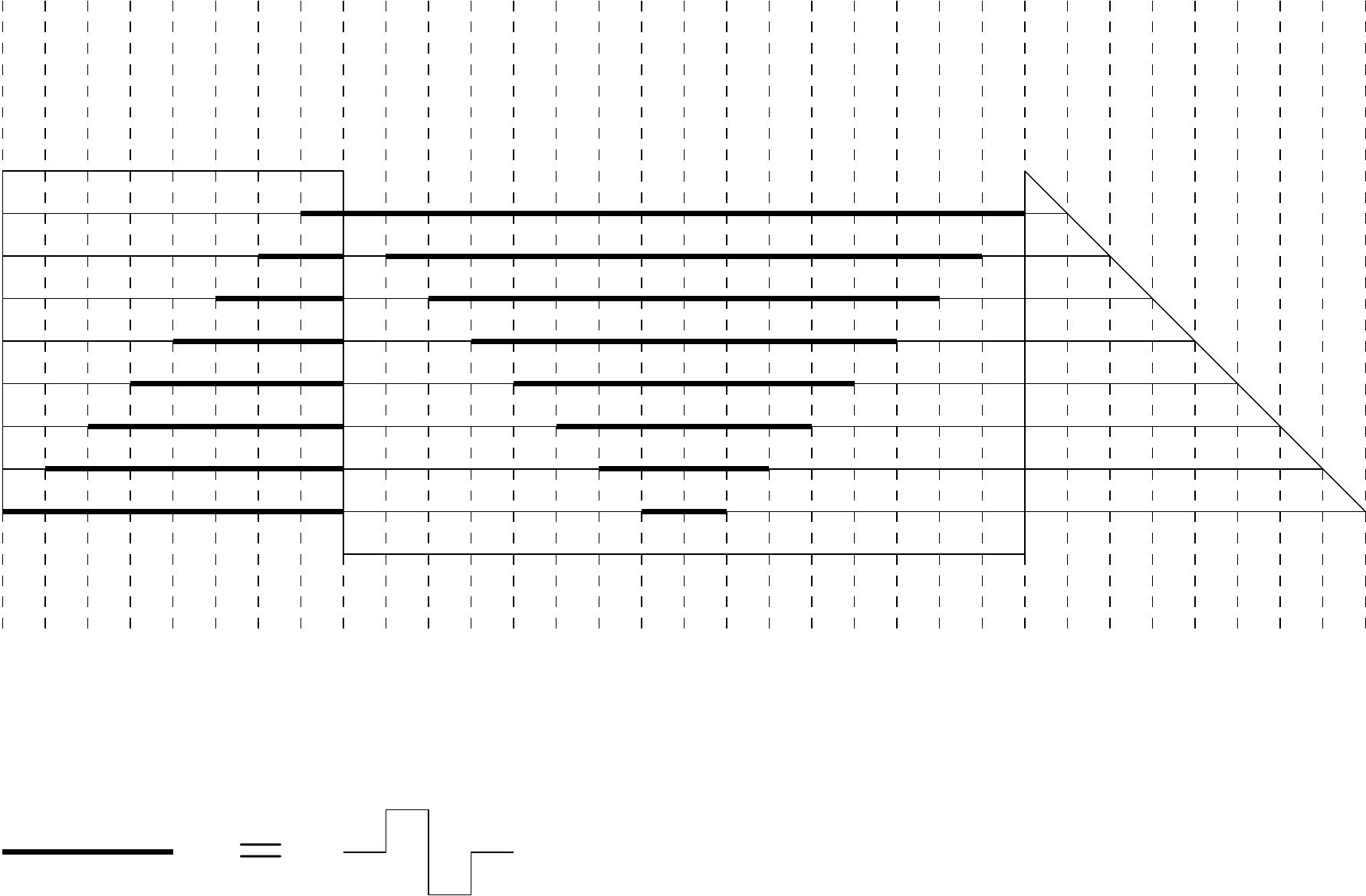}
    \caption{\textbf{$C_N$-substitution rule:} The end-piece in the $K$-fold substitution rule. The bolded edges are ``wiggled'' in a grid subdivided by a factor of four. }
    \label{fig:endpeace}
\end{figure}

This chosen subdivision factors are so that the copies fit isomorphically in the refined grids. 

These substitution rules lead to the Loewner carpets of Theorem \ref{thm:existence}. First, we prove existence and embeddability. Later, in Section \ref{sec:rigidity} we prove the assertion that for each $Q,Q'$, infinitely many of the constructed examples
are quasisymmetrically distinct. 

\begin{proof}[Proof of existence and embeddability in Theorem \ref{thm:existence}] Fix $N_0,N_1,N_3$
 to be determined
 later, and let $S_{N_0},C_{N_1},WS_{N_3}$ be the substitution rules described above. Let $\alpha_i$ for $i=0,1,2$ be 
 such that $\alpha_0+\alpha_1+\alpha_2=1$. Construct an infinite sequence $a = \{a_i\}_{i=1}^\infty$,
 with each $a_i=0,1,2$, with the property that
$$\limsup_{N \to \infty} ||\{a_i = j| i=1, \dots, N\}| -\alpha_ j N| < \infty.$$
Here, if $S \subset \N$, then $|S|$ is the number of its elements. 

\begin{remark}
Note that there exist uncountably many sequences as above. However, as mentioned in the
introduction, at present we can only prove that for each $Q,Q'$, countably many of our
examples are quasisymmetrically distinct; see Section \ref{sec:rigidity}.
\end{remark}

Next we construct 
a quotiented inverse system, by applying at the $i$'th stage the
 substitution rule $S_{N_0}$, if $a_i=0$, $C_{N_1}$ if $a_i=1$
 and $WS_{N_3}$ if $a_i=2$. 
 The substitution rules are applied to edges for $S_{N_0}, WS_{N_3}$, and to vertex stars in the second case. 

For such a sequence the edges have length
$$s_k \asymp \left(\frac{1}{N_0^{\alpha_0}} \frac{1}{(96N_1+26)^{\alpha_1}}  \frac{1}{(8(N+2N^2))^{\alpha_2}}\right)^k.$$ 
Define 
$$h(s_k) = \left(\frac{1}{N_0^{\alpha_0}} \frac{1}{((2N_1+1)(96N_1+26))^{\alpha_1}}  \frac{1}{(8(N_2+2N^2))^{\alpha_2}}\right)^k$$
and
$$Q=\frac{\alpha_0 \log(N_0) + \alpha_1 \log[(96N_1+26)(2N_1+1)] + \alpha_2\log(8(N_2+2N_2^2))}{\alpha_0 \log(N_0) + \alpha_1 \log(96N_1+26)) + \alpha_2 \log(8(N_2+2N_2^2))}.$$
Then, since $Q = \lim_{k \to \infty} \log(h(s_k))/\log(s_k)$, it is easy to see by Lemmas \ref{lem:homogeneous} and \ref{rmk:hausdim} that $X_\infty^\infty$ is $Q$-Ahlfors regular, and thus has Hausdorff dimension $Q$. Also, by Theorem \ref{thm:replacementpi}, as in the proof of Theorem \ref{thm:explicit}, it follows that $X_\infty^\infty$ is $Q$-Loewner, since it satisfies a $(1,1)$-Poincar\'e inequality and is $Q$-Ahlfors regular. 

Each graph $G_k$ is a subgraph of the grid $l_k \Z^2$, where each edge in the grid has length
$$l_k \asymp \left(\frac{1}{N_0^{\alpha_0}} \frac{1}{(64N_1)^{\alpha_1}}  \frac{1}{(8N_2)^{\alpha_2}}\right)^k.$$ 
Set 
$$
\alpha = \frac{\alpha_0 \log(N_0) + \alpha_1 \log(64N_1) 
+ \alpha_2\log(8N_2)}{\alpha_0 \log(N_0) + \alpha_1 \log(96N_1+26) 
+ \alpha_2 \log(8(N_2+2N_2^2))}\,  .
$$

We then have,
  $\alpha = \lim_{k \to \infty}\frac{\log(l_k)}{\log(s_k)}$. By Lemma \ref{lem:snow}, the maps $f_k$ are uniformly $\alpha$-snowflake. Further, these converge to $f_\infty$ which is a $\alpha$-snowflake embedding for $X_\infty^\infty$. Thus, the image $f_\infty(X_\infty^\infty)$ is $Q'$-Ahlfors regular with $Q'= Q/\alpha$. 
  
We wish to now find the values of $\alpha_i$ and $N_i$ which will enable to obtain
the desired Hausdorff dimensions $Q,Q'$. That is, we want to find solutions to the three equations
$$\alpha_0+\alpha_1+\alpha_2 = 1,$$
$$Q=\frac{\alpha_0 \log(N_0) + \alpha_1 \log[(96N_1+26)(2N_1+1)] + 
\alpha_2\log(8(N_2+2N_2^2))}{\alpha_0 \log(N_0) + \alpha_1 \log(96N_1+26))
 + \alpha_2 \log(8(N_2+2N_2^2))},
$$
and
$$
Q'=\frac{\alpha_0 \log(N_0) + \alpha_1 \log[(96N_1+26)(2N_1+1)] 
+ \alpha_2\log(8(N_2+2N_2^2))}{\alpha_0 \log(N_0) + \alpha_1 \log[(64N_1)] + \alpha_2\log(8N_2)}.$$

Define the following constants:
\begin{eqnarray*}
A_Q  &\defeq & (1-Q)\log(N_0) \\
 B_Q &\defeq & \log[(96N_1+26)(2N_1+1)]-Q\log(96N_1+26)\\
  C_ Q &\defeq & (1-Q)\log(8(N_2+2N_2^2)) \\
A_{Q'}&\defeq & (1-Q')\log(N_0) \\
B_{Q'}&\defeq &\log[(96N_1+26)(2N_1+1)]-Q'\log(64N_1)\\
 C_{Q'}&\defeq &\log(8(N_2+2N_2^2))-Q'\log(8N_2).
 \end{eqnarray*}
By multiplying with the denominators and simplifying, we get three linear equations
\begin{eqnarray}
\alpha_0 + \alpha_1 + \alpha_2 &=& 1 \label{eq:1}\\
\alpha_0 A_Q + \alpha_1 B_Q + \alpha_2 C_Q &=& 0 \label{eq:2} \\
\alpha_0 A_{Q'} + \alpha_1 B_{Q'} + \alpha_2 C_{Q'} &=& 0 \label{eq:3}
\end{eqnarray}
Additionally, we have the constraint $\alpha_i \geq 0$. 

First, choose $N_1$ and $N_2$ so big that $C_{Q'},B_{Q'},B_Q>0$. Also, choose $N_1$ so big that $Q\log(96N_1+26)<Q'\log(64N_1)$. We have $0>A_Q>A_{Q'}$ and $B_Q>B_{Q'}$ by these assumptions.

Now, solving the third and the first equations \eqref{eq:1}, \eqref{eq:3} with $\alpha_1=0$, we obtain the equation 
$$0=\alpha_0 A_{Q'}+(1-\alpha_0)C_{Q'},$$
and $\alpha_0 = \frac{-C_{Q'}}{A_{Q'}-C_{Q'}} \in (0,1)$. This gives a partial solution $( \frac{-C_{Q'}}{A_{Q'}-C_{Q'}}, 0,  \frac{A_{Q'}}{A_{Q'}-C_{Q'}}) = (\alpha_0,\alpha_1,\alpha_2) = \textbf{A}_-$. Similarly, setting $\alpha_2=0$, we obtain the partial solution
$( \frac{-B_{Q'}}{A_{Q'}-B_{Q'}}, \frac{A_{Q'}}{A_{Q'}-B_{Q'}},0) = (\alpha_0,\alpha_1,\alpha_2) = \textbf{A}_+$.

Let $f(\alpha_0,\alpha_1,\alpha_2)$ be the left hand side of the second equation \eqref{eq:2}.
 Plugging in $\textbf{A}_-$, and using $Q'>Q$ we get $f(\textbf{A}_-)<0$, since $Q>1$.
 Then, plugging in $\textbf{A}_+$ we get $f(\textbf{A}_+)>0$ since $A_Q>A_{Q'}$ and 
$B_Q>B_{Q'}$. Now, by continuity there must exists a $t \in (0,1)$, such that for
 $\textbf{A}_t = t\textbf{A}_+ + (1-t) \textbf{A}_-$ we have $f(\textbf{A}_t)=0$. 
Since this solution is a convex combination of the solutions to the first and third
 equations, and we have $\alpha_i \geq 0$, it follows that $\textbf{A}_t$ is 
the solution we are looking for.

Finally, since $f_\infty \co X_\infty^\infty \to \C$ is a quasisymmetric embedding and $X_\infty^\infty$ is ALLC, 
then its image is also an ALLC subset of the plane.
 Consequently, by Corollary \ref{prop:uniformization}, it follows that $f_\infty(X_\infty)$ 
can be uniformized by a circle, or square carpet by a planar quasiconformal map
 $g \co \C \to \C$.
 \end{proof}

  \section{Loewner rigidity}

\label{sec:rigidity}

The main goal of this section is the proof of  Theorem \ref{thm:rigidityquantitative}. Its statement in a similar form is originally due to Kleiner, but its proof has not appeared. The proof we provide was largely developed independently, but involves ideas discussed with Kleiner. This theorem
will enable us to show in Section \ref{sec:nonquasi}  that for fixed $Q,Q'$,
the infinitely many examples of Loewner carpets constructed Section \ref{sec:planarinversequot}
are pairwise quasisymmetrically distinct .

After some preliminary work, the rigidity Theorem \ref{thm:rigidityquantitative} is proved in Subsection 
\ref{ss:proofrigidity}. First we discuss properties of 
quasiconformal maps giving them a metric differential $\Lip[f](x)$ and absolute
continuity properties in Subsection \ref{ss:quasi}. A differential with respect to 
Cheeger charts is constructed in Subsection \ref{ss:differential}. The two 
notions of a metric derivative and Cheeger differentials are connected using 
\cite{cheegerkleinerschioppa} in Subsection \ref{ss:metricanddiff}. This control 
at certain regular points leads to the proof of the main theorem. In the end, we 
make a few remarks on related and useful results for quasiconformal maps in 
Subsection \ref{ss:quasiconf}.

The application to prove the quasisymmetrically distinctness is done in the final Section \ref{sec:nonquasi}.


\subsection{Statement of Rigidity theorem}

Recall, that if $\Gamma$ is any family of rectifiable curves, a non-negative Borel function $\rho$ is called admissible, if $\int_\gamma \rho ~ds \geq 1$ for every $\gamma \in \Gamma$. The modulus of $\Gamma$ is defined as
$$\Mod_p(\Gamma) \defeq \inf_{\rho} \int \rho^p ~d\mu,$$
where the infimum is over all admissible functions. See \cite{heinonenkoskela} for references and further discussion.  

We will need the following definition.
\begin{definition}
\label{d:cmon}
Let 
$$
\Gamma^C_\epsilon(x,y) 
\defeq \big\{\gamma ~|~  \gamma \text{ connects }B(x,\epsilon r) \text{ to } B(y,\epsilon r) \text{\,\,and }
\len(\gamma) \leq C(1+\epsilon)d(x,y)\big\}\, .
$$
A metric on $X$ is called {\it $C$-monotone} if for all $x_1,x_2\in X$,
 if there exists a positive function $\Phi \co (0,\frac{1}{2}] \to (0,\infty)$ 
such that for all $\epsilon \in (0,\frac{1}{2})$ and any $x,y \in X$ with $r\defeq d(x,y) >0$,
$$
\Mod_Q(\Gamma^C_\epsilon(x,y)) \geq \Phi(\epsilon)\, .
$$
\end{definition}

For PI spaces, and in particular, for Loewner spaces,
one easily gets such a modulus bound by the capacity bounds in \cite[Proposition 2.2 and Lemma 3.2]{kortepoincare}; compare also to \cite[Section 15]{ChDiff99}. The function $\Phi$ and the constant $C$ depend only on the constants in the 
Poincar\'e inequality and the doubling constant. In many cases, such as for inverse
 limits and for some cases of quotiented inverse systems, one can choose $C=1$. The case of $C=1$ corresponds to the condition of thickness studied in \cite[Section 15]{ChDiff99}. The terminology of tangent cones is given at the end of the introduction and the notion of a blow-up map is made rigorous in Subsection \ref{ss:proofrigidity}. 

Recall, that a map $f\co (X,d_X) \to (Y,d_Y)$ is \textit{$C-$bi-Lipschitz} if for all $x,y \in X$, $$d_Y(f(x), f(y)) \asymp_{C} d_X(x,y).$$

\begin{theorem}[Kleiner, unpublished]
\label{thm:rigidityquantitative}
 Let $X,Y$ be analytically one dimensional $Q$-Loewner
spaces and
assume that  $X$ and $Y$ have $C$-monotone metrics. If $f \co X \to Y$ is a quasisymmetric 
homeomorphism, then for almost every $x \in X$, and for any tangent cone $T_X$ of $X$ at $x$ 
there exists a tangent  cone $T_Y$ of $Y$ at $f(x)$  such that for any blow-up map 
$T_f \co T_X \to T_Y$ is $C^2$-bi-Lipschitz. 
\end{theorem}

Indeed, for many cases we can take $C=1$, and then the theorem gives that the maps $T_f$ are isometries.

\begin{remark} Although we will not need this here,
it follows that under the assumptions of Theorem \ref{thm:rigidityquantitative}, 
any quasisymmetric self map $f \co X \to X$ is $K$-quasiconformal with a uniform $K$; see  
 Corollary \ref{prop:uniformquasiconformal}. Moreover,  if the metrics on 
 $X$ and $Y$ are $1$-monotone,
 then the tangent map can be chosen to be an isometry and the map $f$  is a conformal map.
\end{remark}

\begin{remark}
The conclusion that the blow-up maps are bi-Lipschitz follows from the absolute continuity of $f$, the $L^Q$ integrability of its upper gradient $\Lip[f]$ and the Poincar\'e inequality, 
see \cite[Proof of Theorem 10.8]{heinonenquasi}, together with the theory of $H^{1,p}$ 
Sobolev spaces from \cite{ChDiff99, shanmugalingamsobolev}.
However, in full generality, there is no control on the bi-Lipschitz constant of the blow-ups. 
\end{remark}

\subsection{Results on Quasisymmetries}\label{ss:quasi}

We will need
 some background and auxiliary results related to 
 quasiconformal maps, modulus and differentiation. Below we 
 will  consider quasisymmetries $f \co X \to Y$, where
$X,Y$ are assumed $Q$-Loewner. We will denote 
the $Q$-Ahlfors regular measures of 
$X,Y$ by $\mu,\nu$ respectively. Where needed for explicit bounds, we denote by $C_{AR}$ the Ahlfors regularity constant. In other words, $\mu(B(x,r)) \asymp_{C_{AR}} r^Q$ and $\nu(B(x,r)) \asymp_{C_{AR}} r^Q$.

%
%
Positive modulus families also behave well under quasisymmetric maps.  This was shown by Tyson.

\begin{lemma}[Tyson, \cite{tysonconfdim}]  If $X,Y$ are 
$Q$-Loewner metric measure spaces and $f \co X \to Y$ is a quasisymmetry and 
$\Gamma$ is a family of curves in $X$, then there is a constant $K$ such that
$$\frac{1}{K} \Mod_Q(\Gamma) \leq \Mod_Q(f(\Gamma)) \leq K \Mod_Q(\Gamma).$$
\end{lemma}
\noindent
This is connected to the following fact which pertains to the issue of controlling the behavior of the restriction of $f$ to curves. First, a map $\gamma \co [a,b] \to X$ is called absolutely continuous, if the push-forwards measure $
\gamma^*(\lambda)$ is absolutely continuous with respect to the Hausdorff 
measure $\mathcal{H}^1|_\gamma$. In this case, by \cite{ambrosiotilli}, the 
curve is controlled by its metric derivative which is almost everywhere defined by:
$$
d_\gamma(t) = \lim_{h \to 0} \frac{d(\gamma(t),\gamma(t+h))}{h}\, .
$$
Indeed, the length of the curve can be given by
\begin{equation}\label{eq:lengam}
\len(\gamma(t)) = \int_a^b d_\gamma(t) 	~dt\, .
\end{equation}
Now, if  $\gamma \co [a,b] \to X$ is absolutely continuous, then 
 $f$ is said to be {\it absolutely continuous along $\gamma$} if $f \circ \gamma$
 is an absolutely continuous curve and  $f^*(\mathcal{H}^1|_\gamma)$ 
is absolutely continuous with respect to $\mathcal{H}^1|_{f \circ \gamma}$. In this case, the metric derivative $d_{f \circ \gamma}$ satisfies
$$\len(f\circ \gamma) \leq \int_a^b d_{f \circ \gamma(t)} 	~dt.$$
Further, by using 
$$\Lip[f](x) \defeq \limsup_{y \to x, y \neq x} \frac{d(f(x),f(y))}{d(x,y)}$$
we obtain $d_{f \circ \gamma} \leq \Lip[f] \cdot d_\gamma(t)$ almost everywhere. (Note that both sides are defined almost everywhere.) Hence, we obtain
\begin{equation}\label{eq:fgamma}
\len(f \circ \gamma) \leq \int_a^b \Lip[f](\gamma(t)) d_\gamma(t)~dt.
\end{equation}

The following lemma, which is adapted from
\cite[Theorem 8.1]{heinonenkoskela}, provides an  abundance of curves $\gamma$ along which $f \circ \gamma$ can be controlled.
\begin{lemma}\label{lem:abscontcurv} If $\Gamma$ is a family of curves and $f \co X \to Y$ is a quasisymmetry between two $Q$-Loewner spaces, then 
$$\overline{\Gamma} = \{\gamma \in \Gamma\  |\   f \text{ is absolutely continuous on } \gamma\}$$
satisfies $\Mod_Q(\overline{\Gamma}) = \Mod_Q(\Gamma)$ and $\Mod_Q(\Gamma \setminus \overline{\Gamma}) = 0$.
\end{lemma}
The lemma can be paraphrased as follows: $f$ is absolutely continuous on almost 
every curve $\gamma$. 
\vskip2mm

We will need a Muckenhoupt-property, or more precisely,
 a reverse H\"older property, for $\Lip[f]$ established in \cite[Theorem 7.11]{heinonenkoskela}. 
This  property 
is equivalent to various different versions of the Muckenhoupt condition.
 See \cite{steinharmonic} for more background and discussion of various equivalent definitions. 
We will not need most properties of these weights here.

\begin{proposition}\label{prop:ainftyhk}
Let $X$ and $Y$ be $Q$-Loewner spaces and  let $f \co X \to Y$ denote a quasisymmetry. Then, there is a $Q'>Q$ and constant $C$ 
such  that for any ball $B=B(x,r)$
$$
\left(\vint_{B(x,r)} \Lip[f]^{Q'} d\mu\right)^\frac{1}{Q'} \leq C \left(\vint_{CB(x,r)} \Lip[f]^Q ~d\mu \right)^\frac{1}{Q}\, .
$$
In particular, the function $\Lip f$ is a Muckenhoupt 
$A_\infty$-weight as defined in \cite{steinharmonic}.
\end{proposition}

As a consequence of the reverse H\"older property, 
we get the following result which states that if the set $E$ is small in measure, then not only is
 the integral of $1_E$
 along some curve $\gamma$ small, but the integral of $\Lip[f] 1_E$ as well.

\begin{lemma} \label{lem:modulusboundstronger} Let $f \co X \to Y$ 
be as in the previous lemma. Assume  $\Mod_{Q}(\Gamma)>c$, and 
that $\gamma \subset B(x,R)$,
for every $\gamma \in \Gamma$.
Then, there exists  $C_{\rm con}$ and $\alpha$ such that the following holds.
For every Borel set $E$ with
$\frac{\mu(E \cap B(x,R))}{\mu(B(x,R))} \leq \epsilon,$
 there exists $\gamma \in \Gamma$ such that:
$$
\int_\gamma \Lip f 1_E ~ds 
\leq C_{\rm con} \frac{\epsilon^\alpha}{c^\frac{1}{Q}}
 \left(\vint_{B(x,CR)} \Lip[f]^Q ~d\mu \right)^\frac{1}{Q}\, .
$$
\end{lemma}

\begin{proof} Let $C$ and $Q'>Q$ be the constants from
 Proposition \ref{prop:ainftyhk}. Let $C_{AR}$ be the Ahlfors regularity 
constant of $X$ and $A = \left(\vint_{B(x,R)} \Lip[f]^{Q} ~d\mu \right)^\frac{1}{Q}$. 
Set $\alpha = \frac{Q'-Q}{Q'Q}$ and $g = \frac{c^{\frac{1}{Q}}}{2C(C_{AR}A)^{\frac{1}{Q}}\epsilon^\alpha R}\Lip[f] 1_{E \cap B(x,R)}$. Then, by H\"older's inequality and  Proposition \ref{prop:ainftyhk}:
\begin{eqnarray*}
\int g^Q ~d\mu &\leq & \frac{c}{2^QAC^Q\epsilon^{\alpha Q}}\vint_{B(x,R)} \Lip[f]^Q 1_E ~d\mu \\
&\leq & \frac{c}{2^QAC^Q\epsilon^{\alpha Q}} \left(\vint_{B(x,R)} \Lip[f]^{Q'} ~d\mu \right)^\frac{Q}{Q'} \epsilon^{\frac{Q'-Q}{Q'}} <c\\
&\leq & \frac{c}{2^Q\epsilon^{\alpha Q}} \epsilon^{\frac{Q'-Q}{Q'}} <c.\\
\end{eqnarray*}
The function $g$ is then not admissible, and so, there is a curve 
$\gamma$ with
$$\int_\gamma g ~ds < 1.$$
Setting $C_{\rm con}=2C(C_{AR} A)^{\frac{1}{Q}}$ we get then the desired integral bound.
\end{proof}

 We will also need the following quantitative 
 version of absolute continuity, which is related to the Muckenhoupt condition. It follows from the arguments in \cite[Section 7.8]{heinonenkoskela}, 
the Muckenhoupt condition and the equivalent characterizations of Muckenhoupt weights
 in \cite[Chapter V]{steinharmonic}. (Recall, that the push-forward measure is defined as 
$f^*(\mu)(E) = \mu(f^{-1}(E))$.)

\begin{proposition}\label{prop:abscont} Let $f \co X \to Y$ be a quasisymmetry, then there is a constant $L \geq 1$ so that
$$f^{-1*}(\nu) \leq L\Lip[f]^Q \mu,$$
and
$$f^*(\mu) \leq L \Lip[f^{-1}]^Q \nu.$$
Moreover, there exists an $\alpha$ so that
$$\frac{f^{-1*}(\nu)(B \cap E)}{f^{-1*}(\nu)(B)} \leq L \left( \frac{\mu(E \cap L B)}{\mu(L B)}\right)^\alpha,$$
and
$$\frac{f^*(\mu)(B \cap E)}{f^*(\mu)(B)} \leq L \left( \frac{\nu(E \cap L B)}{\nu(L B)}\right)^\alpha,$$
holds for all balls $B$ and Borel sets $E$ in $X$ and $Y$ respectively. The 
constants depend only on the constants in the 
quasisymmetry and Loewner conditions.
\end{proposition}

\subsection{Existence of a differential and Sobolev spaces}\label{ss:differential}

In this section, using Cheeger derivatives we will define a linear differential $df$ for the function $f$. Later, we will show that the dilation $\Lip[f]$ of $f$ is controlled by $df$.

For the following proposition, we define $H^{1,p}_{\rm loc}(X)$ 
as the space of $f \co X \to \R \cup \{-\infty,\infty\}$ which possess
 weak upper gradients $g$ so that $f$ and $g$ are in $L^{p}_{\rm loc}$. Here, $p \in (1,\infty)$.
A function $g$ is an \emph{upper gradient} for $f$, if for every rectifiable curve 
$\gamma \co [a,b] \to X$ we have
$$|f(x)-f(y)| \leq \int_\gamma g ~ds,$$
if $f(x),f(y) <\infty$, and if one of them is infinity, then the right hand side is infinity. Upper gradients were initially defined in \cite{heinonenkoskela}. The function $g$ is called a \emph{weak upper gradient} if the previous holds for every curve $\gamma$ except for a family of exceptional curves with vanishing modulus, as defined in \cite{shanmugalingamsobolev}.

The above definition of a Sobolev space $H^{1,p}$ is the one best suited for our purposes. Another, earlier definition appeared in \cite{ChDiff99} involving a minimal generalized upper gradient defined using relaxation. By results in \cite{ChDiff99} we could derive the same estimates for this gradient using the upper gradients that we use below, but the approach we use is slightly more direct. As shown in \cite{shanmugalingamsobolev}, these definitions are equivalent. For a discussion using the definition using upper gradients see \cite{bjornbjornbook}.



The following proposition is related to \cite[Theorem 10.8]{heinonenquasi}. Our proof is somewhat different though. The derivative we construct can naturally be interpreted as the adjoint of the natural map defined on the measurable cotangent bundles $(df)_* \co T^*Y \to T^* X$.

\begin{proposition}\label{prop:differential}
Let $X$ and $Y$ be $Q$-Loewner, $f \co X \to Y$ be a quasisymmetry, and $U \subset X, V \subset Y$ be measurable sets. Suppose $\phi: U \to \R^n$ is a chart of $X$, and $f(U) \subset V$ for some chart $\psi : V \to \R^m$. 
Then, for almost every $x \in U$ the function $\psi \circ f$ is Cheeger differentiable with respect to $\phi$ at $x$ with derivative $d(\psi \circ f)\co \R^n \to \R^m$. Moreover, $m=n$ and $d(\psi \circ f)$  is invertible almost everywhere with inverse $d(\phi \circ f^{-1})(f(x))$.
\end{proposition}

\begin{proof}[Proof of Proposition \ref{prop:differential}] 
Fix a chart $\phi \co U \to \R^n$ in $X$ which is a $L$-Lipschitz and $f(U) \subset V$, where $\psi \co V \to \R^m$ is a $L$-Lipschitz chart in $Y$. By Heinonen-Koskela \cite[Theorem 7.11]{heinonenkoskela}, it follows that the function
$$\Lip[f](x) \defeq \limsup_{y \to x, y \neq x} \frac{d(f(x),f(y))}{d(x,y)}$$
satisfies $\Lip[f] \in L^{Q'}_{\rm loc}$, that is, it is 
locally in $L^{Q'}$, for some $Q'>Q$. Now, 
extend $\psi$ to all of $Y$ by a Lipschitz extension.  
Then, for each component $\psi_i$ of $\psi$,  $\Lip[\psi_i \circ f](x) \leq L \Lip[f](x)$. The function 
$\Lip[ \psi_i \circ f]$ is a (weak)-upper gradient\ for $\psi_i \circ f$; see \cite[Theorem 10.8]{heinonenquasi}. From this it follows that
 $\psi_i \circ f \in H^{1,Q'}_{\rm loc}$. Also note that $\psi_i \circ f$ is continuous since both functions, $\psi_i$ and $f$, are continuous. By \cite{balogh} the 
 function $\psi_i \circ f$ is Cheeger differentiable almost everywhere in $U$ with derivative $d(\psi_i \circ f)$ defined for almost every $x \in U$. Combining the components, we obtain the derivative $d(\psi \circ f)$. 
 
 The argument can be repeated for $f^{-1}$ on the measurable set $f(U)$ (which is a Suslin set, since $f$ is continuous \cite{kechris}).  If now $x$ is a point of differentiability for $\psi \circ f$, and $f(x)$ is a point of differentiability of $\phi \circ f^{-1}$, then the definition of Cheeger-differentiability ensures the following 
\begin{eqnarray*}
\phi(y)-\phi(x) &=& d(\phi \circ f^{-1})(f(x))(\psi(f(y))-\psi(f(x))) + 
o(d(f(x),f(y))) \\
&=& d(\phi \circ f^{-1})(f(x))d(\psi \circ f)(x)
(\phi(y)-\phi(x)) + o(d(x,y)) + o(d(f(x),f(y))).
\end{eqnarray*} 
Now, at generic points, as $\delta \to 0$, the vectors $\phi(y)-\phi(x)$, 
for $y \in B_\delta(x)$ span the space $\R^n$ (see e.g. 
\cite{batespeight}). Further, by combining the quasisymmetry property 
with a density argument from the proof of 
\cite[Theorem 10.8]{heinonenquasi}, we get also that the error $o(d(f(x),f(y)))$ can be controlled by an error of the form $o(d(x,y))$. 
These two facts then force the linear map $d(\phi \circ f^{-1})(f(x))d(\psi \circ f)(x)$ to be identity. Here, we are implicitly using the fact that $f$ and $f^{-1}$ are absolutely continuous, see Proposition \ref{prop:abscont}.
\end{proof}

\subsection{Controlling metric differential using Cheeger differential}\label{ss:metricanddiff}
From now on, without essential loss of generality,  we assume that $X,Y$ are analytically one dimensional equipped with global charts $\phi_X \co X \to \R$ and $\phi_Y \co Y \to \R$. This will simplify arguments below, while the general case could be obtained by restrictions to appropriate subsets. Also, in the application to rigidity in Theorem \ref{thm:existence}, we do in fact have global chart functions.
\vskip1mm

{\it From now on we assume $X,Y$ are analytically one dimensional equipped with
 global charts $\phi_X \co X \to \R$ and $\phi_Y \co Y \to \R$.}
\vskip1mm 

\begin{remark} 
By restricting to a subset $U_i$ so that $f(U_i) \subset V_i$, and $U_i$ and $V_i$ are charts of 
$X$ and $Y$ respectively, the notions here could be defined without assuming the existence of
a global chart
\end{remark}

Let now $\mathcal{D}_X$ denote a fixed dense set of points in $X$.
Let $\mathcal{F}_X = \{d(x,\cdot)~|~ x \in \mathcal{D}\}$ denote a collection of functions. 
Set $V_{X}=\{x \in X ~|~ dg(x) \text{ exists at } x, \ \ \forall g \in \mathcal{F}\}$. 
By the statement that ``$dg(x)$ exists'', we mean that $g$ is differentiable 
with unique derivative $dg(x)$ with respect to the fixed global chart. 

The following is the specialization of a definition from \cite{cheegerkleinerschioppa} to our setting in which we consider curves
rather than more general curve fragments. 

\begin{definition}
\label{d:generic}
A pair $(\gamma,t)$ is generic,
 if $\gamma \co I \to X$ is a rectifiable curve and $t \in I$ is such that the following properties hold.

\begin{enumerate}
\item For each $g \in \mathcal{F}_X$, $(g \circ \gamma)'(t)$ exists at $t$ and is approximately 
continuous at $t$ .
\vskip1mm

\item The point $t$ is in the interior of $I$ and an approximate continuity point of 
$\sup_{g \in \mathcal{F}_X}(|(g \circ \gamma)'(t)|$.
\vskip1mm

\item The derivative $(\phi_X \circ \gamma)'(t)$ exists and is approximately continuous at $t$.
\end{enumerate}
\end{definition}
\vskip1mm

Recall, that a function $g \co U \to \R$ for an subset $U \subset \R$ is said to be
{\it  approximately continuous at an interior point $t \in U$} if for every $\epsilon>0$ we have  
$$
\lim_{s \to 0^+}\frac{|[t-s,t+s] \cap \{|g-g(t)|>\epsilon\}|}{2s} = 0\, .
$$

The first two conditions in Definition \ref{d:generic}  are from \cite[Thm. 4.1.6]{ambrosiotilli} . They guarantee existence of the metric derivative $d_\gamma(t)$
 and the equality 
 \begin{equation}\label{eq:metricderiv}
 d_\gamma(t) = \sup_{g \in \mathcal{F}}|(g \circ \gamma)'(t)|\, .
 \end{equation}
The last condition in Definition \ref{d:generic} allows us to define an element $\gamma'(t)$ of
 $TX_{\gamma(t)}=\R$, by $(\phi_X \circ \gamma)'(t) \in \R$. 
\vskip1mm

For a fixed $x \in V_X$, there are two norms on $TX_x = \R$, one is the dual norm that 
$\Lip$ induces on the cotangent bundle:
$$
||t||_{\rm Lip*,x} \defeq \frac{|t|}{\Lip [\phi_X](x)}\, .
$$
The other is defined by the differentials of functions in $\mathcal{F}_X$:
$$||t||_{\mathcal{D},x} = \sup_{f \in \mathcal{F}_X} |df(t)|.$$
Both norms
 are defined on the full measure subset $x \in V_X$. 
Further, by \cite[Theorem 6.1]{cheegerkleinerschioppa} for almost every $x \in V_X$, the
above two norms
coincide.

Let $\overline{V}_X$ denote the  set of $x$ such that 
there exists a curve $\gamma$ s.t. $\gamma(t) = x$, $\phi_X \circ \gamma'(t) \neq 0$, and 
$(\gamma,t)$ is generic. By a slight modification of \cite{cheegerkleinerschioppa} involving filling in the curve fragments to give true curves, this set has full measure and for each $x \in \overline{V}_X$ the above two norms are equal.  
By combining \cite[Thm. 4.1.6]{ambrosiotilli} with the above,  if $(\gamma,t)$ is regular and $\gamma(t) \in \overline{V}_X$ we have 

\begin{equation}\label{eq:regularpt}
d_\gamma(t)=||\gamma'||_{\mathcal{D},x}=||\gamma'||_{\rm Lip*,x}\, . 
\end{equation}
 {\it In particular, the metric derivative $d_\gamma$ can be expressed using the 
differential structure.}

The global charts give a global trivialization $TX_x = \R$ and the norms are given by 
$||v||_x = \tau_X(x) |v|$ for some measurable $\tau \co \R \to \R$. In fact, 
 by the previous paragraph, we can set $\tau_X(x) = \frac{1}{\Lip[\phi_X](x)}$, 
which is almost everywhere defined. Further, in the context of Proposition \ref{prop:differential} and on analytically one dimensional spaces, the function $d(\phi_Y \circ f)$ will simply 
be a scalar quantity.

The maps $f \co X \to Y$ will be controlled at certain generic regular points defined as follows.
\vskip2mm

\begin{definition}
\label{d:regular}
The point $x\in X$ {\it regular}
 for a quasi-symmetric map $f \co X \to Y$ if the following conditions hold. 
\vskip1mm

\begin{enumerate}
\item $x \in \overline{V}_X, f(x) = \overline{V}_Y$.
\vskip1mm

\item $x$ is an approximate continuity point of $\tau_X$ and $\tau_X(x) >0$, 
and $y=f(x)$ is an approximate continuity point of $\tau_Y(y)>0$.
\vskip1mm

\item $x$ is an approximate continuity point of $d(\phi_Y \circ f)$ and $\phi_Y \circ f$ is Cheeger differentiable at $x$ with respect to $\phi_X$ with $d(\phi_Y \circ f) \neq 0$.
\vskip1mm

\item $f(x)$ is an approximate continuity point of $d(\phi_X \circ f^{-1})$ with $d(\phi_X \circ f^{-1}) \neq 0$ and $\phi_X \circ f^{-1}$ is Cheeger differentiable at $f(x)$ with respect to $\phi_Y$ with derivative $d(\phi_X \circ f^{-1})$.
\vskip1mm

\item
$x$ is a Lebesgue point of $\Lip[\phi_X]$ with $\Lip[\phi_X](x) \neq 0$.
\vskip1mm

\item $f(x)$ is a Lebesgue point of $\Lip[\phi_Y]$ with $\Lip[\phi_Y](f(x)) \neq 0$.
\vskip1mm

\item $d(\phi_Y \circ f)(x)^{-1} = d(\phi_X \circ f^{-1})(f(x))$
\end{enumerate} 
\end{definition}
\vskip1mm

\noindent
Note that If $x$ is regular for $f$, then $f(x)$ is regular for $f^{-1}$. Also, it follows from Proposition \ref{prop:differential} and Proposition \ref{prop:abscont} that almost every point $x$ is a regular point.
\vskip2mm

\subsection{Blow-up maps and proof of Theorem \ref{thm:rigidityquantitative}} \label{ss:proofrigidity}

 Next, we will blow up $f$ at a regular point $x$. 
Let $r_n \to 0$ be some sequence.
Choose another sequence $s_n =\diam(f(B(x,r_n)))/2$.
If $f \co X \to Y$ is a quasisymmetry,
 then also $\lim_{n \to \infty} s_n = 0$. Next, 
choose a sub-sequence so that $$(X,d_n/r_n,\mu_X/\mu_X(B(x,r_n)))$$
converges in the measured Gromov-Hausdorff sense to some tangent 
cone $T_X$ and 
$$(Y,d_n/s_n,\mu_Y/\mu_Y(B(f(x),s_n)))$$ converges in the measured
 Gromov-Hausdorff sense to some tangent cone $T_Y$. Since our spaces are doubling, 
 such subsequences exist.

Since $X,Y$ are quasiconvex, it follows that $X$ and $Y$ are uniformly perfect. This guarantees that $s_n \neq 0$ and that the division above is well defined. The technical details of the proofs of the various unproved statements below are
are classical; for these details, see the proof of \cite[Theorem 1.1]{liconf}.
Directly applying Definition \ref{d:iqs} shows that the maps 
$f_n=f \co (X,d_n/r_n,\mu_X/\mu_X(B(x,r_n))) \to (Y,d_n/s_n,\mu_Y/\mu_Y(B(f(x),s_n)))$ 
are equicontinuous. Therefore, using a variant of Arzel\`a-Ascoli, 
we obtain a map $T_f \co T_X \to T_Y$. Indeed, there exist Gromov-Hausdorff approximations $\phi_{n,X} \co T_X \to (X,d_n/r_n,\mu_X/\mu_X(B(x,r_n)))$ and inverse approximations $\psi_{n,Y} \co (Y,d_n/s_n,\mu_Y/\mu_Y(B(f(x),s_n))) \to T_Y$. Then, we can define $T_f(z) = \lim_{n \to \infty} \phi_{n,Y} \circ f_n \circ \phi_{n,X}(z)$. Here the limit can be shown to exist along a subsequence,
initially on a dense subset and then everywhere. One can also show that $T_f$ is a quasisymmetry. 
This construction is possible for any sequence of $r_n$, and so for any tangent cone
 $T_X$ of $X$ there is an associated $T_Y$ 
which is quasisymmetric to $T_X$.

The crucial fact concerning regular points is that, asymptotically, the change $d(f(x),f(y))$ can be 
estimated from above by a constant times $\frac{\tau_Y(f(y))}{\tau_X(x)}|d(\phi_Y \circ f)|$. This quantity is
essentially the metric derivative $\Lip[f](x)$ of $f$.  
As usual, denote by $o(d(y,x))$ some function such that
 $\lim_{y \to x} \frac{o(d(y,x))}{d(y,x))} = 0$.
\begin{lemma}\label{lem:distanceest} Assume that $X$ and $Y$ admit $C$-monotone metrics and 
let $\Delta >0$. If $x$ is a regular point, and 
$z,y$ satisfy $d(z,y) \geq \Delta d(x,y)$ and $d(z,x) \leq \frac{1}{\Delta} d(x,y)$, then 
$$
d(f(z),f(y)) \leq C \cdot \frac{\tau_Y(f(x))}{\tau_X(x)}\cdot |d(\phi_Y \circ f)| \cdot d(z,y) + o(d(y,x))\, .
$$
\end{lemma}

\begin{proof}
It suffices to show
$$\limsup_{z,y \to x} \frac{d(f(z),f(y))}{d(z,y)} \leq C \frac{\tau_Y(f(x))}{\tau_X(x)}|d(\phi_Y \circ f)|,$$
where the limit is along pairs $(y,z)$ such that $d(z,y) \geq \Delta d(x,y)$
 and $d(z,x) \leq \frac{1}{\Delta} d(x,y)$. Below, all the limit superiors
 are taking along sequences satisfying this assumption.

 Let $\eta,\epsilon>0$ be very small constants to be determined. 
Given a measurable function $g$  on $X$ and $q$ with $g(q)>0$, we define 
$$
A_{\eta,q,g} :=\left\{w \in X \ | \ |g(w)-g(q)|>\eta |g(q)|, \left|\frac{1}{g(w)}-\frac{1}{g(q)}\right| > \frac{\eta}{g(q)}\right\}\, .
$$

 Define the ``bad set'' by
$$
\mathcal{B} := A_{\eta, x, \tau_X} \cup A_{\eta, x, \tau_Y \circ f} 
\cup A_{\eta,x,|d(\phi_Y \circ f)|} \cup (X \setminus (\overline{V}_X \cup f^{-1}(\overline{V}_Y)))\, .
$$
From  the assumed regularity of $x$ and Proposition \ref{prop:abscont}, it follows that
\begin{equation}\label{eq:density}
\lim_{r \to 0} \frac{\mu(\mathcal{B} \cap B(x,6Cr))}{\mu(B(x,6Cr))}  = 0\, .
\end{equation}

In particular, there is an $r_\eta>0$ such that
\begin{equation}\label{eq:density2}
\frac{\mu(\mathcal{B} \cap B(x,6Cr))}{\mu(B(x,6Cr))} < \eta\, ,
\end{equation}
for all $r<r_0$. 

Consider the family $\Gamma^C_{\epsilon}(z,y)$ and the collection 
$\overline{\Gamma}^C_{\epsilon}(z,y)$ with the same modulus, where $f$ is absolutely 
continuous on each curve.
Then, by assumption and Lemma \ref{lem:abscontcurv}, we have 
$$\Mod_Q(\overline{\Gamma}^C_{\epsilon}(z,y))>\Phi(\epsilon)\,.$$
 By Lemma \ref{lem:modulusboundstronger} and the doubling property, 
there is a constant $C_D$ (depending on doubling and $\Delta$) with the following property.
There is a curve $\gamma \co [a,b] \to X$
 in $\overline{\Gamma}^C_\epsilon(z,y)$ connecting $B(z,\epsilon d(z,y))$ to $B(y,\epsilon d(z,y))$ with
\begin{equation} \label{est:abscont}
\int_\gamma \Lip[f](\gamma(t)) 1_{\mathcal{B}} ~ds \leq \frac{C_D (\Lip[f](x)+\eta) \eta^\alpha }{\Phi(\epsilon)^{\frac{1}{Q}}} d(x,y),
\end{equation}
for all $y,z$ with $d(y,x) \lesssim r_\eta$ and $d(z,x) \lesssim r_\eta$.

Here, we used that $d(y,z) \geq \Delta d(x,y)$, 
since this allows on to translate  estimates for $\mathcal{B}$ in \eqref{eq:density2} in
 $B(x,6Cd(x,y))$ to density estimates for the ball $B(z,Cd(y,z))$. 
 
Now, $\gamma(a) \in B(z,\epsilon d(z,y))$ and $\gamma(b) \in B(y, \epsilon d(z,y))$. 
Also, if $\gamma(t) \not\in \mathcal{B}$ then $\phi_Y \circ f \circ \gamma$ and 
$\phi_X \circ \gamma$ are differentiable, and by the chain rule, we get:

\begin{eqnarray}
d_{f \circ \gamma(t)} &=& \tau_Y(f(\gamma(t)) |(\phi_Y \circ f \circ \gamma)' | \nonumber \\
 &=& \tau_Y(f(\gamma(t)) |d(\phi_Y \circ f)(\gamma(t))| (\phi_X \circ \gamma)' \nonumber \\
&=&  \tau_Y(f(\gamma(t)) |d(\phi_Y \circ f)(\gamma(t))| \frac{d_\gamma(t)}{\tau_X(\gamma(t))} \nonumber\\
&\leq & (1+\eta)^3 \frac{\tau_Y(f(x))}{\tau_X(x)} |d(\phi_Y \circ f)(x)| d_\gamma(t). \label{eq:estimatetau}
\end{eqnarray}

For all $\gamma(t) \in \mathcal{B}$, we still have the worse estimate, which holds almost everywhere,
\begin{equation}\label{eq:lipest}
d_{f\circ \gamma(t)} \leq \Lip[f](x) d_\gamma(t).
\end{equation}

Put:
\begin{align}
\label{e:rhodef}
\rho :=  \frac{\tau_Y(f(x))}{\tau_X(x)} |d(\phi_Y \circ f)(x)| \, .
\end{align}
By the quasisymmetry Definition \ref{d:etaqs}, 
there exists a function $\psi\co [0,\infty) \to [0,\infty)$ satisfying
\begin{align}
\label{e:psi}
d(f(\gamma(a)),z) &\leq \psi(\epsilon) d(f(z),f(y))\, .\notag\\
d(f(\gamma(b)),y) &\leq \psi(\epsilon) d(f(z),f(y))
\, .
\end{align}
 By Definition \ref{d:etaqs} the function $\psi$ satisfies
 $\lim_{\epsilon \to 0} \psi(\epsilon) = 0$.
\vskip2mm

Next, from inequalities and equations \eqref{eq:lipest}, \eqref{est:abscont},  \eqref{eq:regularpt}, \eqref{eq:fgamma} and \eqref{eq:lengam}  we get
\begin{eqnarray*}
d(f(y),f(z)) &\leq & d(f(y),\gamma(b)) + d(f(z),\gamma(a)) + d(f(\gamma(a)),f(\gamma(b))) \\
&\leq & 2\psi(\epsilon) d(f(z),f(y)) + \len(f \circ \gamma) \\
& \leq & 2\psi(\epsilon) d(f(z),f(y)) + \int_a^b d_{f \circ \gamma(t)} ~dt \\
& \leq & 2\psi(\epsilon) d(f(z),f(y)) + \int_a^b d_{f \circ \gamma(t)} 	1_{\gamma(t) \not\in \mathcal{B}} ~dt  +  \int_a^b \Lip[f](\gamma(t)) d_\gamma(t) 	1_{\gamma(t) \in \mathcal{B}} ~dt   \\
& \leq &  2\psi(\epsilon)d(f(z),f(y)) + \int_a^b (1+\eta)^3 \rho d_\gamma(t)~dt  +  \frac{C_D (\Lip[f](x)+\eta) \eta^\alpha }{\Phi(\epsilon)^{\frac{1}{Q}}} d(z,y)\\
&\leq &  2\psi(\epsilon)d(f(z),f(y)) +  (1+\eta)^3 \rho  \len(\gamma) + \frac{C_D (\Lip[f](x)+\eta) \eta^\alpha }{\Phi(\epsilon)^{\frac{1}{Q}}} d(z,y)\\
&\leq &  2\psi(\epsilon)d(f(z),f(y)) +  (1+\eta)^3 \rho Cd(z,y) + \frac{C_D (\Lip[f](x)+\eta) \eta^\alpha }{\Phi(\epsilon)^{\frac{1}{Q}}} d(z,y)\\
\end{eqnarray*}

So, by reorganizing, we get
$$d(f(z),f(y)) \leq \frac{1}{1-2\psi(\epsilon)}(1+\eta)^3 \rho Cd(z,y) + \frac{1}{1-2\psi(\epsilon)} \frac{C_D (\Lip[f](x)+\eta) \eta^\alpha }{\Phi(\epsilon)^{\frac{1}{Q}}} d(z,y).$$

Finally, if we fix $\epsilon>0$ then $y,z \to x$ we can send $\eta \to 0$ in Equation \ref{eq:density2} and $r_\eta \to 0$, giving:
$$
\limsup_{y ,z \to x} \frac{d(f(y),f(z))}{d(z,y)} \leq \frac{1}{1-2\psi(\epsilon)}C \rho(x).$$
Then, since $\epsilon>0$ was arbitrary, we get the desired result:
$$
\limsup_{y,z \to x} \frac{d(f(y),f(z))}{d(z,y)} \leq C \rho(x)\, .
$$
\end{proof}

An application of the previous theorem gives the proof of our rigidity result.

\begin{proof}[Proof of Theorem \ref{thm:rigidityquantitative}] 
Let $\rho$ be as in \eqref{e:rhodef}.

 Note that by Lemma \ref{prop:differential}, we have
 $d(\phi_Y \circ f)(x)^{-1} = d(\phi_X \circ f^{-1})(f(x))$. 
 Also, by Lemma \ref{lem:distanceest} applied to $f$ and $f^{-1}$, if for fixed $\Delta$, we 
have $d(y,z) \geq \Delta d(x,y)$ and $d(z,x) \leq \frac{1}{\Delta} d(x,y)$, then
 we have:
$$
 \frac{1}{C} \rho(x) \leq \liminf_{y,z \to x} \frac{d(f(z),f(y))}{d(z,y)}
 \leq  \limsup_{y,z \to x} \frac{d(f(z),f(y))}{d(z,y)} \leq C\rho(x)\, .
$$

If $r_n \searrow 0$ is any sequence and $\epsilon >0$, 
 then for $n$ large enough we get:
$$s_n = \diam(f(B(x,r_n)))/2 \leq  (1+\epsilon)C \rho(x) r_n$$
 However, we also have:
$$\frac{\rho r_n}{C(1+\epsilon)}\leq s_n,$$
This can be obtained by choosing a generic $(\gamma,t)$ with $\gamma(t)=x$,
 and considering the sequence $\alpha^\pm_n \searrow 0$,
with $d(\gamma(t\pm\alpha^\pm_n),x) = r_n$. 

Therefore, 
$$\frac{1}{C} \leq \liminf_{n \to \infty} \frac{r_n}{s_n} \leq \limsup_{n \to \infty} \frac{r_n}{s_n}
 \leq C\, .$$ 
Then, by  Lemma \ref{lem:distanceest} again, fixing $R$ large, for any sequence $y_n,z_n \in B(x,R r_n)$ 
with $d(y_n,z_n) \geq \Delta r_n$, we get:
$$
\frac{1}{C^2} \leq \liminf_{n \to \infty} \frac{d(f(y_n),f(z_n)) r_n}{s_n d(z_n,y_n)} 
\leq \limsup_{n \to \infty} \frac{d(f(y_n),f(z_n)) r_n}{s_n d(z_n,y_n)} \leq C^2\, .
$$

Thus, with respect to the rescaled metrics $d_X/r_n$ and $d_Y/(s_n)$,
the maps $f|_{B(x,Rr_n)}$ converge to $C^2$-Bi-Lipschitz maps on any
 $\Delta r_n$-separated nets. 
By sending $\Delta \searrow 0, R \nearrow \infty$ we get the desired result for the tangent maps $T_f \co T_X \to T_X$, which is a limit of the previous maps. 
\end{proof}
\vskip4mm

\noindent
\subsection{Further remarks on  Quasiconformality:} \label{ss:quasiconf}
As an application of Theorem \ref{thm:rigidityquantitative} we study briefly the rigidity it imposes on quasiconformal maps. This is not needed elsewhere in this paper, but is of independent interest and provided for the sake of completeness.

Consider a homeomorphism $f \co X \to Y$.  
Put $L_f(x,t) \defeq	 \sup_{y \in B(x,t)} d(f(x),f(y))$,
 and 
$l_f(x,t) \defeq \inf_{y \not\in B(x,t)} d(f(x),f(y)).$ 
We define
$$H_f(x,t) \defeq \frac{L_f(x,t)}{l_f(x,t)}\, ,$$
$$H_f(x) \defeq \limsup_{t \to 0} H_f(x,t)\, .$$
\begin{definition}
\label{d:qc}
${}$
\begin{itemize}

\item[1)]
The map $f$ is quasiconformal
 if there exists $M$ such that for all $x$, we have
$H_f(x)<M$.
\vskip1mm

\item[2)]
The map {\it $K$-quasiconformal} if $H_f$ is bounded, and $H_f(x) \leq K$ 
almost everywhere; equivalently $||H_f(x)||_{\infty} \leq K$.
\end{itemize}
\end{definition}

By \cite{heinonenkoskela}, any quasiconformal map between
 Loewner spaces is quasisymmetric.
 That is, the infinitesimal boundedness of $H_f(x)$ implies
 the boundedness of $H_f(x,t)<M$ at all scales. 

 Next, we show a striking corollary of Theorem \ref{thm:rigidityquantitative}. Namely, that
 the quasicoformality constant can be uniformly controlled
 if the domain and target are analytically one dimensional.

\begin{proposition}[Kleiner, unpublished]\label{prop:uniformquasiconformal} Let $X,Y$ be 
 $Q$-Loewner and  analytically $1$-dimensional. Assume in addition that metrics on $X,Y$ are
 $C$-monotone. Any quasiconformal map $f \co X \to Y$ is $C^4$-quasiconformal.
\end{proposition}
\begin{proof} First, note that if $g \co X \to Y$ is a $L$-bi-Lipschitz map, and 
$X,Y$ are connected, then it is not hard to show that $H_g(x) \leq L^2$. 
Similarly, it is not hard to show that $H_f(x)=H_{T_f}(p)$, where $T_f$ is 
obtained by taking a tangent map along some sequence $r_n \searrow 0$ 
where $H_f(x)$ obtains its limit superior. Here $p$ is the base point of the tangent.
 By Theorem \ref{thm:rigidityquantitative} the map $T_f$ is $C^2$-bi-Lipschitz and 
so it is $C^4$-quasiconformal and $H_{T_f}(p) \leq C^4$. Consequently, $H_f(x) \leq C^4$. 
\end{proof}

In particular, if $C=1$, then any quasiconformal map is conformal.

\section{Quasisymmetrically distinguishing carpets using Loewner rigidity} 
\label{sec:nonquasi}

In this subsection we prove the remaining part of Theorem \ref{thm:existence}.
Namely, for all $Q,Q'$, infinitely many of our examples are pairwise quasisymmetrically distinct.

The distinctness requires introducing a new (quasi)-invariant for our construction. This invariant is fine enough to distinguish countably many of our examples from each other; while leaving the distinctness open for the full uncountable collection of examples.
\begin{definition}
If $h \co (0,\infty) \to (0,\infty)$ is a function, we define
\begin{equation}\label{eq:quasiinv}
\overline{h}_{\rm inf}(t) := \liminf_{r \to 0} \frac{h(tr)}{h(r)}\, .
\end{equation}
\end{definition}

The fact that $\overline{h}_{\rm inf}h(t)$ is invariant up to a constant factor is a consequence of the following proposition.

\begin{theorem}\label{thm:comparability}
Suppose $X,Y$ have measures $\mu,\nu$ which are $Q$-Ahfors regular. Suppose, moreover, that these measures are $h_X,h_Y$-uniform with comparability constants $C_X,C_Y$. Suppose also that $h_X,h_Y$ are uniformly doubling with constant $D$, and that the spaces admit $(1,Q)$-Poincar\'e inequalities with constants $C_{PI}$. Then, if there is a quasiconformal map $f \co X \to Y$, then for every $t \in (0,1]$ we have
$$
\overline{h}_{\inf,X}(t) \asymp \overline{h}_{\inf,Y}(t)\, .
$$
Here the comparability constant depends only on $C_X,C_Y,D,C_{PI}$.
\end{theorem}
\begin{remark}
In particular, the comparability does not depend on the constants in the Ahlfors regularity condition.
\end{remark}
\begin{proof} Fix $t>0$. Without loss of generality assume $t<\frac{1}{2}$, since otherwise
 the claim follows from the uniform doubling of $h_X,h_Y$.
\vskip1mm

 Let $x$ be a regular point for
 $f$. 
By Theorem \ref{thm:rigidityquantitative}, it follows that at $x$,
 the map $f$ induces a $L$-bi-Lipschitz map $T_X \to T_Y$ between
 any tangent cone $T_X$ at $X$, and some tangent cone
 $T_Y$ at $f(x)$ obtained along a compatible sequence in $Y$. 
The constant $L$ depends only on the doubling constant and the
 constant in the Poincar\'e inequality.  

We can choose to define $T_X$ along a sequence $r_n \searrow 0$ so that we have
$$\overline{h}_{\rm inf,X}(t) = \lim_{n \to \infty} \frac{h(tr_n)}{h(r_n)}.$$

Let $p_X$ be the base-point of $T_X$ and $\mu_X$ its limit measure, and 
$p_Y,\nu_Y$ the analogous quantities for $T_Y$. 
There is an $t$-net $N_X$ in $B(p_X,1) \subset T_X$ of size at least $D\overline{h}_{\rm inf,X}(t)^{-1}$. The set $f(N_X)$ will be a $\frac{1}{L} t$-net within $B(p_Y, L) \subset T_Y$.
Therefore, its size is bounded by
$$\frac{\nu_Y(B(p_Y,2L))}{\inf_{q \in T_Y}\nu_Y(B(q,\frac{1}{L}t))} \lesssim \frac{1}{\overline{h}_{\rm inf,Y}(t)}.$$
By combining these with the doubling property, 
we get $\overline{h}_{\rm inf,Y}(t) \lesssim \overline{h}_{\rm inf,X}(t)$. 
The other direction follows by switching the roles of $X$ and $Y$.
\end{proof}
\vskip2mm

\noindent
{\it Theorem \ref{thm:existence} (Completion of the proof.)} In 
Subsection \ref{sec:generalexamples}
 we showed that there exist positive integers $N_1,N_2,N_3$, and 
$\alpha_1,\alpha_2,\alpha_3$, 
so that the following holds. 
 Consider three substitution rules $S_{N_1}, C_{N_2}$ and 
$WS_{N_3}$ and index them by $1,2$ and $3$ respectively. 
Take any sequence $\mathbf{a}=(a_j) \in \{1,2,3\}^\N$, with
$$
\limsup_{n \to \infty}||\{j \ | \ a_j = k\}|-\alpha_k N|<\infty\, .
$$
Then, the sequence of spaces resulting from performing the substitution rule 
$a_i$ at the $i$'th step will have a limit space $X_{\mathbf{a}}$ which is $Q$-Loewner
and whose snow-flake image in the plane is $Q'$-Ahlfors regular. By choosing $N_i$ 
sufficiently large we can ensure that $\mathbf{a}$ is not constant.

Fix such a sequence, and define $\mathbf{a}^N$ by repeating every value of 
$\mathbf{a}$ $N$-times. That is, for example, $\mathbf{a}^2=(a_1,a_1,a_2,a_2,a_3,a_3, \cdots)$. 
Each of these sequences gives rise to a space $X_{N}=X_{\mathbf{a}^N}$, which is uniformly 
doubling and satisfies a uniform Poincar\'e inequality. These spaces are $h_N$-uniform with
functions $h_N(t)$ and with comparability constant $C_1$ independent of $N$. In fact, we can set 
$h_N(s^N_k)$ to be the measure of an edge at the $k$'th level of the construction, where 
$s_k^N$ are the edge lengths in that construction.  Let now $C_3$ be the constant from Theorem \ref{thm:comparability} corresponding to the uniform comparability constants $C_X=C_Y=C_1$, and the uniform constants in doubling and the Poincar\'e inequalities of $X_{N}$.

We have  for every $t \in (0,1]$
\begin{equation}\label{eq:liminf}
\lim_{N \to \infty}\overline{h}_{\inf,X_{N}}(t) \leq  C_2 t^{P}
\end{equation}
 for some $P > Q$. This is because rescaling along a long sequences of 
repeated substitution rules of the form $C_{N_2}$ gives scales where the space 
resembles a space slightly higher dimensional than $Q$

Fix $N_1 = 1$ and $t_1 = 1$, and proceed inductively as follows to define an increasing sequence of $N_i$ and a decreasing sequence $t_i \in (0,1]$. For each $N$ we have a constant $L_{N}$ so that
$$\overline{h}_{\inf,X_{N}}(t) \geq L_{N} t^Q.$$

Now, choose $t_2 \leq t_1$ so that $L_{N_1} t_2^Q \geq 4C_3 C_2 t_2^{P}$, and then choose an $N_2$ using Equation \eqref{eq:liminf} so large that $\overline{h}_{\inf,X_{N_2}}(t_2) \leq2C_2 t_2^{P}.$

Next, if $t_{n},N_n$ have been defined, then define $t_{n+1} \leq t_n$ so that  
\begin{equation}\label{eq:choice1}
\min\{L_{N_1}, \dots, L_{N_n}\} t_{n+1}^Q \geq 4C_3C_2 t_{n+1}^{P}\, ,
\end{equation}
and choose $N_{n+1}$ using Equation \eqref{eq:liminf} so large that 
\begin{equation}\label{eq:choice2}
\overline{h}_{\inf,X_{N_{n+1}}}(t_{n+1}) \leq 2C_2 t_{n+1}^{P}.
\end{equation}

We observe that the sequence of spaces $X_{N_i}$ so constructed are quasisymmetrically distinct. Namely, if  $X_{N_i}$ were quasisymmetric to $X_{N_j}$ for some $j>i$, then by Theorem \ref{thm:comparability} we would have $\overline{h}_{\inf,X_{N}}(t) \asymp C_3 \overline{h}_{\inf,X_{N_j}}(t)$. However, the choice of $t=t_j$ would contradict the choices made in \eqref{eq:choice1} and \eqref{eq:choice2}, as this would give $4C_2 t_{n+1}^{P} \leq  2C_2 t_{n+1}^{P}$, which is an impossibility.
\qed

\bibliographystyle{amsplain}

\end{document}